\documentclass[11pt]{article}

\usepackage{graphicx}
\usepackage{svg}
\usepackage{amsmath}
\usepackage{amsthm}
\usepackage{amssymb}
\usepackage{fancyhdr}
\usepackage{enumerate}
\usepackage[breaklinks]{hyperref}
\usepackage{lscape}
\usepackage{multirow}
\usepackage{relsize}
\usepackage{setspace}
\usepackage{natbib}
\usepackage{fullpage}
\usepackage{caption}
\usepackage{authblk}
\usepackage{totcount}
\usepackage{array}
\usepackage{longtable}
\usepackage{subcaption} 
\usepackage[titletoc]{appendix}
\usepackage{tikz} 
\usepackage[top    = 0.75in,
bottom = 1.22in,
left   = 1in,
right  = 1in]{geometry}
\usepackage{booktabs}
\usepackage{pgfplots}
\usepackage{algorithm}
\usepackage{algpseudocode}
\usepackage{makecell}
\usepackage{rotating}

\renewcommand{\vec}[1]{\mathbf{#1}}

\newcommand{\black}[1]{{\color{black}#1}}

\newcolumntype{L}[1]{>{\raggedright\arraybackslash}p{#1}}

\newcolumntype{C}[1]{>{\centering\arraybackslash}m{#1}}

\newtotcounter{citnum} 
\def\oldbibitem{} \let\oldbibitem=\bibitem
\def\bibitem{\stepcounter{citnum}\oldbibitem}

\newtheorem{definition}{Definition}[section]

\newtheorem{lemma}[definition]{Lemma}
\newtheorem{corollary}[definition]{Corollary}
\newtheorem{proposition}[definition]{Proposition}
\newtheorem{example}[definition]{Example}

\DeclareMathOperator*{\argmin}{arg\,min}
\DeclareMathOperator*{\argmax}{arg\,max}

\begin{document}
\onehalfspacing

\title{Dynamic Repair and Maintenance of Heterogeneous Machines Dispersed on a Network: A Rollout Method for Online Reinforcement Learning}

\vspace{4mm}

\author{\large Dongnuan Tian\\ \vspace{-4mm} \footnotesize{Department of Management Science, Lancaster University Management School, Lancaster University, Lancaster LA1 4YW, United Kingdom.  Email: d.tian2@lancaster.ac.uk.}\\\text{ }\\ \large Rob Shone\footnote{Corresponding author}\\ \footnotesize{Department of Management Science, Lancaster University Management School, Lancaster University, Lancaster LA1 4YW, United Kingdom.  Email: r.shone@lancaster.ac.uk.}\\}
\date{ }

\pagestyle{fancy}
\fancyhf{}
\rhead{\thepage}
\lhead{\footnotesize \textbf{Tian and Shone: }\textit{Dynamic Repair and Maintenance of Heterogeneous Machines Dispersed on a Network}}

\normalsize
\maketitle

\hrule
\text{ }\\

\noindent \Large \textbf{Abstract}\\
\normalsize

\noindent We consider a problem in which a single repairer is responsible for the maintenance and repair of a collection of machines, positioned at different locations on a network of nodes and edges. Machines deteriorate according to stochastic processes and incur increasing costs as they approach complete failure. The times needed for repairs to be performed, and the amounts of time needed for the repairer to switch between different machines, are random and machine-dependent. The problem is formulated as a Markov decision process (MDP) in which the objective is to minimize long-run average costs. We prove the equivalence of an alternative formulation based on rewards and use this to develop an index heuristic policy, which is shown to be optimal in certain special cases. We then use \textcolor{black}{rollout-based} reinforcement learning techniques to develop a novel online policy improvement (OPI) approach, which uses the index heuristic as a base policy and also as an insurance option at decision epochs where the best action cannot be selected with sufficient confidence. Results from extensive numerical experiments, involving randomly-generated network layouts and parameter values, show that the OPI heuristic is able to achieve close-to-optimal performance in fast-changing systems with state transitions occurring 100 times per second, suggesting that it is suitable for online implementation. \\

\noindent \textit{Keywords: }Repair and maintenance, policy improvement, reinforcement learning, rollout\\

\hrule

\newpage

 \section{Introduction}\label{sec:introduction}


 \black{A classical problem in operations research concerns the dynamic scheduling of maintenance and repair operations for a collection of machines that are prone to failure. These problems fall within the broader framework of reliability modeling, which is concerned with 
the performance of systems that degrade or fail over time due to random influences. 
Related problems have been studied extensively due to their numerous applications in computer systems, telecommunications, transportation infrastructure, manufacturing systems, and other areas (see, e.g., \cite{seyedshohadaie2010risk,ahmad2017reliability, poonia2021performance,geurtsen2023production})}. In the context of machine failures, it is often assumed that each machine has a finite number of possible states, ranging from the pristine or `as-good-as-new' state (in which it functions at its optimal performance level) to the `failed' state (in which it performs at its worst level, or does not function at all). If a machine is left unattended, a \emph{failure process} causes it to pass through a sequence of increasingly undesirable states until it eventually reaches the failed state. The times of state transitions are random, so that the speed of the failure process is unpredictable. The system controller has a set of resources available (in the form of repair engineers, for example) to stall or reverse the failure process by effecting transitions in the opposite direction, i.e. towards the pristine state, and must allocate these resources efficiently over time in order to optimize a pre-specified measure of performance. \black{These types of model assumptions, which combine stochastic degradation with resource-constrained control, have been extensively studied in the literature (see, e.g., \cite{nino2001restless, Glazebrook2005, abbou2019group}).}

The dynamic maintenance and repair problem can be extended and generalized in various ways. For example, one can assume that different kinds of maintenance and repair operations (known as `rate-modifying activities') are available (\cite{lee2001single, woo2017rule, mosheiov2021note}), or
allow machine failures to be governed by multiple dependent failure processes 
(\cite{yousefi2020reinforcement, ogunfowora2023reinforcement, ward2024comprehensive}. 
In this paper we consider a version of the problem that, to the best of our knowledge, is original and enables insightful analyses based on a network model.  We assume that there is a single repair operative on duty (referred to throughout the paper as a `repairer') who bears sole responsibility for the maintenance and repair of $m\geq 2$ machines.  Although the repairer can switch from repairing one machine to another, this cannot be done instantaneously, and the decision-making process should therefore take into account the time required to re-allocate effort from one machine to another in addition to the time needed to perform repairs.  Indeed, the machines might be in different physical locations, in which case some travel time would be needed to switch from one machine to another.  Alternatively, even if travel times are negligible, there might be setup 
requirements for repairing individual machines, so that the repairer must expend some time to perform the necessary preliminary tasks before commencing the next repair.  We refer to the time needed to re-allocate effort from one machine to another as a `switching time', and introduce further complexity by allowing these switching times to be interruptible, so that the repairer may begin the process of switching to one machine but then take another action (such as changing course and switching to a different machine) before completing the switch.

To motivate our network-based model, Figure \ref{intro_fig} shows an example with 4 machines ($m=4$).  The machines are represented as gray-colored nodes in the network, labeled $1$, $2$, $3$ and $4$. At any given time, the repairer occupies one node and can choose either to remain at the current node or attempt to switch to one of the adjacent nodes (connected by edges to the current node). In order to repair a particular machine, the repairer must occupy the corresponding node. Thus, in order to switch from repairing one machine to another, it must traverse an appropriate path through the network.  The white-colored nodes represent intermediate `stages' that must be completed during a switch.  For example, to switch from machine $1$ to $3$, the repairer would need to pass through nodes $5$, $6$ and $7$.  However, the direction of movement can be changed at any time, so the repairer might follow a path such as $(1,5,6,5,2)$.  This would represent a switch from machine $1$ to $3$ that gets `interrupted' in favor of a switch to machine $2$ instead. By designing the network appropriately, we can model intricate relationships between the travel times or setup requirements of different machines, including the effects on switching times when the repairer makes partial progress in switching from one machine to another. In Section \ref{sec:formulation} we provide details of how the system can be modeled as a Markov decision process (MDP). \\

\begin{figure}[hbtp]
    \begin{center}
        \includeinkscape[scale = 0.8]{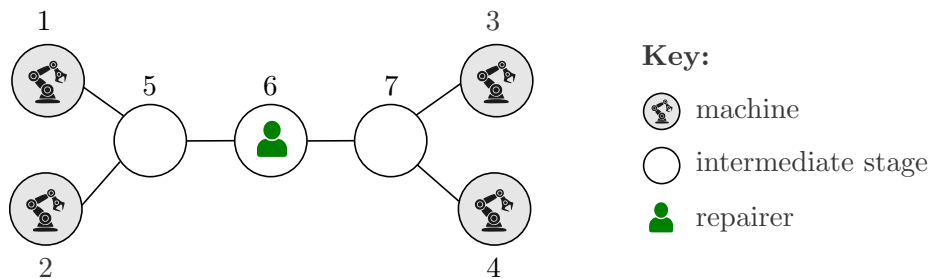}
    \end{center}
    \caption{A network with 4 machines, 3 intermediate stages and a single repairer. Gray-colored nodes represent machines, and white-colored nodes represent intermediate stages.}
    \label{intro_fig}
\end{figure}

One study of particular relevance to ours is that of \cite{Duenyas1996}, which considers the dynamic assignment of a single server in a system of parallel, heterogeneous queues.  Although their model is presented in a different context from ours (i.e. servicing of customers in a queueing system rather than repair of machines), there are similarities in its mathematical formulation, including the assumption of Poisson customer arrivals (analogous to machine degradations in our model) and the incorporation of random switching times.  
In our recent work (\cite{tian2025}), we have considered an extension of the job scheduling problem in \cite{Duenyas1996} that allows the server to interrupt processing and switching times and also uses a network formulation to encode the amounts of effort needed to switch between different tasks. The formulation in \cite{tian2025} has similarities to the one that we consider in the current paper, but there are two main differences. Firstly, in \cite{tian2025} it is assumed that there can be any number of jobs waiting to be processed at a demand point and this leads to the formulation of an MDP with an infinite state space, in which system stability is an important consideration. In this paper we obtain a finite-state MDP because the set of possible conditions for any machine $i$ is finite. While this removes the need to consider questions involving stability, it introduces other complications related to multichain policies and irregular value functions that we elaborate on in Sections \ref{sec:index} and \ref{sec:polimp}. Secondly, in \cite{tian2025} it is assumed that the cost of having $k$ jobs waiting at a particular demand point is linear in $k$, but in this paper we allow the cost function $f(k)$ to be more general, as the effect on a machine's performance may be non-linear in the number of degradation events that has occurred. We also note that the solution methods in \cite{tian2025} are focused on index policies, whereas in this paper we begin by developing an index policy but subsequently use an altered version of this as a base policy in a novel online policy improvement algorithm.

The main contributions of our paper are as follows:\\
\vspace{-2.5mm}
\begin{itemize}
\item We consider an original model of a repair and maintenance problem in which repair and switching times not only follow random distributions but are also interruptible.  In addition, our proposed network formulation allows switching times to depend not only on the machines involved in the switch, but also on the amount of progress made during any switches that have been interrupted.
\item We develop an index-based heuristic for choosing which machine to switch to next, with decisions being continuously re-calculated during repair and switching times in order to allow interruptions. In order to develop the heuristic, we first prove an equivalence between two cost formulations in the underlying MDP formulation. We then prove that the index heuristic is optimal in certain special cases of our problem.
\item In order to obtain a stronger-performing heuristic policy we apply reinforcement learning (RL) techniques and approximate the value function for the index heuristic.  We then use \textcolor{black}{a rollout method, which we refer to as online policy improvement (OPI), to search for improved decisions during online implementation. Our OPI method} includes a novel safety mechanism, based on reliability weights, which mandates that the decision given by the index policy is followed in situations where the best action cannot be identified with sufficient confidence.
\item We present the results of computational experiments with randomly-generated network layouts and parameter values in order to test the performances of our heuristics and investigate the effects of adjusting key parameter values.\\
\end{itemize}
\vspace{-2.5mm}

\textcolor{black}{Section \ref{sec:lit_review} provides a review of relevant literature, including a discussion of how our model could be applied in different application settings.} Section \ref{sec:formulation} provides details of our mathematical formulation. Section \ref{sec:index} introduces our index heuristic and also includes proofs of its optimality in certain special cases. Section \ref{sec:polimp} explains the use of online policy improvement for obtaining a stronger-performing heuristic policy. Section \ref{sec:numerical} provides details of results from our numerical experiments. Finally, our concluding remarks are given in Section \ref{sec:conclusions}.\\

\textcolor{black}{\section{Literature review}\label{sec:lit_review}}

\textcolor{black}{As discussed in Section \ref{sec:introduction}, our study builds upon the previous literature on repair and maintenance problems. In particular, it introduces a network formulation to model travel (or setup) times between heterogeneous machines that degrade according to stochastic processes.
However, our mathematical model and proposed solution approaches can also be applied more generally to other real-world settings in which limited resources are allocated dynamically to assets that degrade randomly over time. For example, we can consider network security problems in which connected devices are vulnerable to attacks or other disruptions, healthcare settings in which patient conditions worsen without effective intervention, or infrastructure monitoring systems in which bridges or roads are classified into discrete condition grades based on inspection criteria. In this section we summarize important contributions to the literature within each of these areas.}

\textcolor{black}{Within the maintenance optimization literature, foundational survey articles by \cite{Cho1991}, \cite{Dekker1997} and \cite{Nicolai2008} have established the need for multi-component models with stochastic degradation processes and limited maintenance resources. In the classical literature, one can find several articles that (like our paper) formulate MDPs for maintenance decisions; see, for example, \cite{Berg1976,Ohashi1982,Ozekici1988,Dekker1995,
Wijnmalen1997}. These papers use dynamic programming methods to establish structural properties of optimal policies or develop heuristics in different settings, but they do not include the spatial features of the problem that we consider. More recently, several studies have modeled repair and maintenance problems using restless bandit formulations, based on the principle that only a certain subset of machines can be attended to at any given time, and those that are unattended are liable to change state (i.e. degrade) rather than remain in the same condition.  \cite{Glazebrook2005} considered a discrete-time formulation with multiple repairers and used the Whittle's index method (see \cite{Whittle1988}) to derive index heuristics in order to minimize expected discounted costs.  \cite{abbou2019group} also used Whittle indices to minimize discounted costs but proposed a group maintenance formulation, in which each `bandit' is a production facility consisting of several machines with different possible states.  \cite{Ruiz2020} considered a scenario with \emph{imperfect} maintenance, in which repairs return a machine to an uncertain state of wear rather than the pristine state, and also used Whittle indices with discounted costs.  \cite{Ayesta2021} developed computational methods based on general expressions for Whittle indices in Markovian bandit models and, as a special case, considered a machine repairman problem in continuous time under the average cost criterion.  \cite{Demirci2025} have also applied the Whittle indices approach to a capacity-constrained maintenance problem. The restless bandits approach has some advantages in that, by considering a decomposition of the problem based on Lagrangian relaxation, one can easily extend the analyses to the case of multiple repairers.  On the other hand, the aforementioned studies do not consider switching times (interruptible or otherwise) and indeed bandit-type models become much less tractable in general when the assumption of immediate, costless switching is removed.}

\textcolor{black}{In network security applications, our paper's paradigm of stochastically degrading machines requiring resource allocation extends naturally to systems of connected devices, where vulnerabilities accumulate over time and must be addressed through patching, monitoring, or other defensive measures. Many examples of such problems can be found in the area of cybersecurity. \cite{Arora2010}, \cite{Ioannidis2012} \cite{Dey2015} and \cite{Fielder2016} all develop stochastic models for information systems in which security risks intensify over time, although they do not use MDP formulations like ours. \cite{Zebrowski2022} consider the problem of mitigating cyber threats to cyber-physical systems and propose a Bayesian framework, which allows for different attack scenarios, to optimize allocation of resources across vulnerable components. \cite{Mookerjee2023} develop a stochastic dynamic programming formulation for a problem with partially observable vulnerabilities, in which decisions must be made as to when the system will be inspected. They show that optimal policies follow a threshold structure, with one (lower) threshold associated with the decision to inspect a vulnerability before fixing the problem and another (higher) threshold for implementing a fix directly. \cite{Medina2025} use stochastic optimization via attack graphs to construct integrated cybersecurity systems, employing two-stage stochastic programming for countermeasure allocation under budget constraints. \cite{Luo2025} propose an automatic detection model, formulated as an MDP, that uses automatic learning strategies to respond to evolving attack patterns and improve detection accuracy. Within the cybersecurity literature, many authors also study partially observable stochastic games (POSGs) for patch strategy optimization, based on the principle that defenders cannot fully observe attacker states or intentions, and may be unaware of the full environmental state. 
Recent work by \cite{Horak2023} has provided a comprehensive theoretical analysis of zero-sum one-sided POSGs, in which only one agent has imperfect information; however, extending these results to more general multi-agent settings remains a difficult task. Deep reinforcement learning (DRL) methods can be applied to security games with high-dimensional state spaces (\cite{Zhang2021}), but these approaches face computational challenges, including difficulties in converging to stable solutions (\cite{Hernandez2019}). Our paper assumes that machine degradation occurs through exogenous stochastic processes rather than strategic adversarial behavior, but we regard the consideration of strategic attack patterns as a promising direction for future research.}

\textcolor{black}{In the field of healthcare applications, the home healthcare literature explicitly addresses routing and scheduling of mobile service providers in order to treat patients whose conditions worsen over time. Several interesting parallels with our work can be found within this area. \cite{Nikzad2021} develop matheuristic algorithms for stochastic home health care planning, with uncertain service times and travel times between patient locations included in their optimization framework. \cite{Cire2022} use approximate dynamic programming (ADP) methods to address the dynamic allocation of home care patients to medical providers, assuming that patients are heterogeneous with respect to health requirements, service durations and regions of residence. \cite{Naderi2023} develop integrated approaches for staffing, assignment, routing and scheduling under uncertainty and propose a mixed-integer program (MIP) with a special structure that allows for a Benders decomposition. \cite{Nasir2024} consider stochastic models for home care transportation with dynamically prioritized patients and propose a chance-constrained multi-period optimization model, in which patient priorities are translated into time-dependent healthcare costs. 
\cite{Khorasanian2024} also study dynamic home care routing within an MDP framework and consider uncertainty in the number of visits per patient referral. 
Emergency department and hospital settings provide additional healthcare applications involving patient queues and resource allocation. \cite{Patrick2008} develop MDP models for the dynamic scheduling of patients with different priorities to diagnostic facilities, with an objective based on waiting time targets. \cite{Lee2020} apply DRL to emergency department patient scheduling, using an MDP model that accounts for diverse patient acuity levels and personalized treatment plans. In the context of our paper, the relevant themes in the healthcare studies that we have cited include stochastically evolving demands across multiple facilities, limited service capacity, and the need to allocate resources dynamically in response to costs that become steeper as demand remains unresolved.}

\textcolor{black}{Infrastructure monitoring systems can be regarded as another important application domain. In these systems, critical assets are often classified into discrete condition grades based on inspection criteria, and limited maintenance resources must be allocated dynamically. Bridges, roads and other infrastructure naturally degrade over time due to traffic loads, environmental factors, and ageing, and this creates systems with natural finite-state representations and stochastic deterioration processes. Several relevant studies can be found in the recent literature. \cite{Tao2020} consider the maintenance of highway bridges and formulate a model that innovatively combines discrete-time and continuous-time MDPs (based on different possible causes of degradation) to find optimal maintenance policies. 
\cite{Xu2022} also consider bridge maintenance and begin by using a Q-learning algorithm to optimize the maintenance strategy for a single bridge, then use integer programming to optimize budget allocation for a network of bridges. 
In related work, \cite{Xu2022b} employ hazard-based approaches to optimize both inspection intervals and repair decisions via semi-Markov decision processes (SMDPs).
\cite{Xin2023} also apply SMDPs to bridge element maintenance, using Weibull distributions for state duration modeling and choosing between maintenance strategies based on specific defect types and development states. \cite{Liu2025} have also studied dynamic optimization methods for bridge networks under budget constraints. Recent work by \cite{Fairley2025} investigates the optimal design of failure-prone systems, with first-stage strategic planning decisions used to decide the number of components (which could represent transport links, for example) to install in a stochastic system that will be maintained dynamically in the second stage.}

\textcolor{black}{In summary, this section has shown that problems involving dynamic allocation of limited resources to stochastically degrading systems can be found across diverse application domains. Common modeling elements include stochastic arrival and service processes, finite-state representations of system conditions, cost functions that increase with unserviced demand or degradation levels, and explicit consideration of travel or transition times between service locations.}\\


\section{Problem formulation}\label{sec:formulation}

Let $G=(V,E)$ be a connected graph, referred to as a \emph{network}, where $V$ and $E$ are the sets of nodes and edges respectively. Let $M\subseteq V$ denote a subset of nodes referred to as \emph{machines}, and let $N:=V\setminus M$ denote the set of other nodes, referred to as \emph{intermediate stages}. We use $m:=|M|$ (resp. $n:=|N|$) to denote the number of machines (resp. intermediate stages). As a convention we label the machines $1,2,...,m$, while the stages in $N$ are labeled $m+1,m+2,...,m+n$.

At any given time, machine $i\in M$ is in state $x_i\in\{0,1,...,K_i\}$, where $0$ is the pristine or `as-good-as-new' state and $K_i$ is the `failed' state. When machine $i$ is in state $x_i\leq K_i-1$, transitions to state $x_i+1$ occur at an exponential rate $\lambda_i>0$; this is the \emph{degradation rate} for machine $i$. All machines in $M$ are assumed to degrade independently of each other. In addition, a cost $f_i(x_i)$ is incurred per unit time while machine $i$ is in state $x_i$. We assume that the functions $f_i$ are strictly increasing and that $f_i(0)=0$ for each $i\in M$. 

There is a single repairer who occupies a single node in $V$ at any given time. At each point in time, the repairer can decide either to remain at the currently-occupied node or attempt to move to one of the adjacent nodes (connected by edges to the current node). In the latter case, it must also choose which node to move to. Therefore, if the number of adjacent nodes is $l$ then the number of possible decisions for the repairer is $l+1$. If it remains at a node $i\in M$ (i.e. at a machine) and $x_i\geq 1$ then it is assumed to be \emph{repairing} the machine, in which case the state of the machine transitions from $x_i$ to $x_i-1$ at a rate $\mu_i>0$, where $\mu_i$ is the \emph{repair rate} for machine $i$. We note that degradations can still happen at a machine while it's being repaired, so that $\lambda_i+\mu_i$ is the total transition rate for a non-failed machine $i$ while it is under repair. If the repairer remains at a node $i$ with $x_i=0$, then the repairer is said to be \emph{idle}. Similarly, idleness can also occur if the repairer chooses to remain at some stage $j\in N$ rather than attempting to move. On the other hand, if the repairer chooses to move to an adjacent node, then it moves to that node at an exponential rate $\tau>0$, referred to as the \emph{switching rate}. Repair times and switching times are assumed to be independent of the degradation processes for machines in $M$. We assume that the switching rate is the same between any adjacent pair of nodes, so that the expected time to move from one node to another is proportional to the number of edges that must be traversed.  This implies that the amount of time to switch from one machine to another is distributed as a sum of i.i.d. exponential random variables and therefore follows an Erlang distribution.

The above assumptions allow the system to be formulated as a finite-state, continuous-time Markov decision process (MDP). Herein, we will follow similar notational conventions to those in \cite{tian2025}. Let $S$ be the state space, given by
$$S:=\left\{(i,(x_1,...,x_m))\;|\;i\in V,\;x_i\in\{0,1,...,K_i\}\text{ for }i\in M\right\},$$
where $i$ indicates the node currently occupied by the repairer. As a notational convention, we will tend to use $i$ for the repairer's current location throughout this paper, while $j$ will be used more generally for nodes in $V$. Also, $k$ will be used to represent the condition of an individual machine, which we will sometimes refer to as the state of the machine (not to be confused with the system state). 
Given that the repairer cannot be repairing and switching at the same time (and also cannot be carrying out more than one repair), the sum of the transition rates under any state cannot exceed $\sum_{j\in M}\lambda_j+\max\{\mu_1,...,\mu_m,\tau\}$ and we can therefore use the technique of uniformization (\cite{Lippman1975,Serfozo1979}) to consider an equivalent discrete-time formulation which evolves in time steps of size $\Delta:=\left(\sum_{j\in M}\lambda_j+\max\{\mu_1,...,\mu_m,\tau\}\right)^{-1}$.

For each state $\vec{x}=(i,(x_1,...,x_m))\in S$, let $R(\vec{x})$ denote the set of nodes adjacent to node $i$ and let $A_{\vec{x}}=\{i\}\cup R(\vec{x})$ denote the set of actions available under state $\vec{x}$. We can interpret action $a\in A_{\vec{x}}$ as the node that the repairer attempts to move to next, given that it is currently in state $\vec{x}$ (with $a=i$ indicating that the repairer remains at the current node).  Then, for two distinct states $\vec{x}:=(i,(x_1,...,x_m))$ and $\vec{x}':=(i',(x_1',...,x_m'))$, the transition probability of moving from state $\vec{x}$ to $\vec{x}'$ in the uniformized MDP can be expressed as
\begin{equation}p_{\vec{x},\vec{x}'}(a):=\begin{cases}\lambda_j\Delta,&\text{ if }i'=i,\;x_j'=x_j+1\text{ for some }j\in M\text{ and }x_l'=x_l\text{ for }l\neq j,\\[8pt]
\mu_i\Delta,&\text{ if }i'=i,\;x_i'=x_i-1,\;x_j'=x_j\text{ for }j\neq i\text{ and }a=i,\\[8pt]
\tau\Delta,&\text{ if }i'\in R(\vec{x}),\;x_j'=x_j\text{ for all }j\in M\text{ and }a=i',\\[8pt]
0,&\text{ otherwise.}\end{cases}.\label{trans_probs}\end{equation}
The first line in (\ref{trans_probs}) corresponds to a degradation event at machine $j$, the second line corresponds to a completed (partial) repair at the repairer's current node $i$ and the third line represents switching from node $i$ to $i'$. We also have $p_{\vec{x},\vec{x}}(a)=1-\sum_{\vec{y} \neq \vec{x}}p_{\vec{x},\vec{y}}(a)$ as the probability of remaining in the same state. Since the units of time are arbitrary, in the rest of the paper we will assume $\Delta=1$ without loss of generality. This implies that we can use quantities such as $\lambda_j$, $\mu_j$, $\tau$ to represent transition probabilities in the uniformized MDP.

For each time step that the system spends in state $\vec{x}$ a cost $c(\vec{x})$ is incurred, given by
\begin{equation}c(\vec{x})=\sum_{i\in M} f_i(x_i).\label{cost_eqn1}\end{equation}
Let $\theta$ denote the decision-making policy to be followed, i.e. the method of choosing actions in our MDP. 
If the same action $\theta(\vec{x})\in A_{\vec{x}}$ is chosen every time the system is in state $\vec{x}\in S$, then the policy is referred to as \emph{stationary} and \emph{deterministic} (\cite{Puterman1994}). The expected long-run average cost (or \emph{average cost} for short) under policy $\theta$, given that the system begins in state $\vec{x}_0\in S$, is
\begin{equation}g_\theta(\vec{x})=\liminf_{t\rightarrow\infty}t^{-1}\mathbb{E}_\theta \left[\sum_{n=0}^{t-1} c(\vec{x}_n)\;\Big |\;\vec{x}_0=\vec{x}\right],\label{avg_cost}\end{equation}
where $\vec{x}_n$ is the system state at time step $n\in\mathbb{N}_0$. The theory of uniformization implies that if $\theta$ is a stationary policy then the average cost $g_\theta(\vec{x})$ in the discretized system is equal to the average cost per unit time incurred under the same policy in the continuous-time system (in which we suppose that actions are chosen at the beginning of each new sojourn, i.e. every time the system transitions from one state to another).  The objective of the problem is to find a policy $\theta^*$ such that
$$g_{\theta^*}(\vec{x})\leq g_\theta(\vec{x})\;\;\;\;\forall \theta\in \Theta,\;\vec{x}\in S,$$
where $\Theta$ is the set of all admissible policies.  In other words, we aim to find a policy that minimizes the average cost.

Since $S$ is finite, the question of system stability does not arise. However, our MDP formulation does have some particular subtleties. For one thing, we cannot claim that the MDP is \emph{unichain} (a common property of MDPs based on queueing systems, for example), since one can easily define a stationary policy $\theta$ under which the resulting Markov chain has more than one positive recurrent class. The simplest way to do this is to define the actions $\theta(\vec{x})$ in such a way that the repairer always remains at whichever node it currently occupies. This policy allows all machines except one (if the machine is initially located at a node in $M$) to remain permanently in their failed states. The average cost under such a policy is finite, but clearly depends on the initial state. However, we can show that our MDP belongs to the `communicating' class of multichain MDP models, as discussed in \cite{Puterman1994} (p. 348), from which it follows that the average cost under an optimal policy must be independent of the initial state. We state this formally below and provide a proof in Appendix \ref{AppA}.
\begin{proposition}Let $\theta^*$ be an optimal policy. Then the average cost $g_{\theta^*}(\vec{x})$ is independent of the initial state $\vec{x}$, and can be expressed as a constant $g^*$.\label{prop1}
\end{proposition}

The theory in (\cite{Puterman1994} then implies that there exists a function $h:S\rightarrow\mathbb{R}$ satisfying the optimality equations
\begin{equation}g^*+h(\vec{x})=c(\vec{x})+\min_{a\in A_{\vec{x}}}\left\{\sum_{\vec{y}\in S}p_{\vec{x},\vec{y}}(a)h(\vec{y})\right\}\;\;\;\;\;\;\forall \vec{x}\in S.\label{opt_eqs}\end{equation}
\textcolor{black}{In theory, we can use dynamic programming (DP) techniques to compute a constant $g^*$ and function $h$ satisfying the optimality equations (\ref{opt_eqs}). DP can be summarized as a framework for solving MDPs via the recursive solution of subproblems using Bellman's optimality principle (\cite{Bellman1954}), which essentially states that the optimal solution of a multi-stage problem is composed of optimal solutions to smaller subproblems. Common DP solution algorithms include value iteration, which iteratively updates estimates of the value function $h$ using Bellman's optimality equation, and policy iteration, which involves alternating between evaluating a fixed policy and improving it via a greedy improvement step.} These algorithms are made more complicated by the existence of policies with multichain structure, but the `communicating' property ensures their theoretical feasibility (see \cite{Puterman1994}, pp. 478-484). Although DP algorithms are useful for finding optimal solutions in small problem instances, they become computationally intractable as $S$ increases in size. We therefore propose the use of index heuristics and reinforcement learning in Sections \ref{sec:index} and \ref{sec:polimp} respectively in order to find strong-performing policies in systems with large state spaces.

Before closing this section, we present a short example to show that the finite-state nature of our problem prevents direct application of certain results that have been proved for infinite-state problems with similar formulations. As mentioned in our introduction, \cite{Duenyas1996} considers an infinite-state problem in which a server must be allocated dynamically between different queues. Their formulation is different from ours in several respects; for example, it assumes that switching and service times are non-interruptible. They also assume linear holding costs and, under this assumption, are able to show that different queues can be assigned priorities based on the values of $c_i\mu_i$, where $c_i$ is the linear holding cost for customers at queue $i$ and $\mu_i$ is the service rate. More specifically, they show that any queue $i$ belonging to the set $\argmax_i \{c_i\mu_i\}$ can be defined as a `top-priority' queue and, under an optimal policy, the server should never switch away from such a queue if it is non-empty. The next example, which considers the case of linear holding costs, shows that the analogous property does not hold in our system.

\begin{example}Consider a network consisting of only 2 nodes connected by a single edge.  Both nodes are machines; that is, $M=\{1,2\}$.  We assume that $\lambda_1=\lambda_2=0.4$, $K_1=K_2=2$ and $f_1(x)=f_2(x)=x$ for $x\in\{0,1,2\}$; that is, the degradation rates for both machines are the same, they both have 3 possible states (with state 2 being the `failed' state) and they both have the same cost function, which is a linear function of the degree of degradation.  The repair rates are $\mu_1=1.1$ and $\mu_2=1$, so machine 1 can be repaired slightly faster than machine 2.  The switch rate is $\tau=100$.  Table \ref{counterex_table} shows an optimal policy for this system, computed using dynamic programming. To clarify, part (a) of Table \ref{counterex_table} shows the decisions specified by the optimal policy at the 9 possible states where the repairer is at machine 1, and part (b) shows the decisions specified at the 9 possible states where the repairer is at machine 2.

\begin{table}[htbp]
    \captionsetup{font=normalsize} 
    \caption{Decisions made under an optimal policy for Example \ref{example1}}
    \vspace{5pt}
    \centering
    \begin{minipage}{.5\textwidth}
        \centering
        
        \captionsetup{font=footnotesize} 
        \subcaption{when the repairer is at machine 1}
        \begin{tabular}{@{\extracolsep{1pt}}lccc} 
            \hline 
             & \multicolumn{1}{c}{$x_2=0$} & \multicolumn{1}{c}{$x_2=1$} & \multicolumn{1}{c}{$x_2=2$}\\ 
            \hline 
            $x_1=0$ & 1 & 2 & 2 \\ 
            $x_1=1$ & 1 & 1 & 1\\ 
            $x_1=2$ & 1 & \fcolorbox{red}{white}{2} & 1\\ 
            \hline
        \end{tabular} 
    \end{minipage}%
    \begin{minipage}{.5\textwidth}
        \centering
        \captionsetup{font=footnotesize} 
        \subcaption{when the repairer is at machine 2}
        \begin{tabular}{@{\extracolsep{1pt}}lccc} 
            \hline 
            &\multicolumn{1}{c}{$x_2=0$} & \multicolumn{1}{c}{$x_2=1$} & \multicolumn{1}{c}{$x_2=2$}\\ 
            \hline 
            $x_1=0$ & 1 & 2 & 2\\ 
            $x_1=1$ & 1 & 1 & 1\\ 
            $x_1=2$ & 1 & 2 & 1 \\ 
            \hline
        \end{tabular} 
    \end{minipage}
\label{counterex_table}\end{table}

Under the optimal policy, the decision under state $\vec{x}=(1,(2,1))$ is to switch to machine 2. Thus, the repairer moves away from machine 1 despite the fact that $c_1\mu_1>c_2\mu_2$ and machine 1 requires repair ($x_1>0$).  This shows that in a finite-state problem, the `top-priority' rule described in (\cite{Duenyas1996}) no longer applies.\\
\label{example1}\end{example} 

\section{Index heuristic}\label{sec:index}

In order to develop an index-based heuristic for our dynamic repair and maintenance problem, we begin by proving that the cost function (\ref{cost_eqn1}) is equivalent (in terms of average cost incurred under a fixed stationary policy) to an alternative function that focuses attention on the repairer's current location.  For a general state $\vec{x}=(i,(x_1,...,x_m))\in S$ and action $a\in A_{\vec{x}}$, let $\tilde{c}(\vec{x},a)$ be defined as follows:
\begin{equation}\tilde{c}(\vec{x},a)=\begin{cases}\sum\limits_{j\in M}f_j(K_j)-\dfrac{\mu_i}{\lambda_i}[f_i(K_i)-f_i(x_i-1)],&\text{ if }i\in M,\;x_i\geq 1\text{ and }a=i,\\[16pt]
\sum\limits_{j\in M}f_j(K_j),&\text{ otherwise.}\end{cases}\label{cost_eqn2}\end{equation}
Notably, the function $\tilde{c}$ (unlike the original cost function $c$) depends on both the state $\vec{x}$ and the action $a$ chosen under $\vec{x}$.  We observe that $\tilde{c}(\vec{x},a)$ includes a negative term if and only if the repairer is located at a damaged machine $i$ and chooses to perform repairs. This negative term increases in magnitude as the machine approaches the pristine state. The next result establishes the equivalence between the cost functions $c$ and $\tilde{c}$ under stationary policies. We provide a proof in Appendix \ref{AppB}.

\begin{proposition}Let $\theta$ be a stationary policy, and let 
\begin{equation}\tilde{g}_\theta(\vec{x}):=\liminf_{t\rightarrow\infty}t^{-1}\mathbb{E}_\theta \left[\sum_{n=0}^{t-1} \tilde{c}(\vec{x}_n,\theta(\vec{x}_n))\;\Big |\;\vec{x}_0=\vec{x}\right],\;\;\;\;\vec{x}\in S,\label{avg_cost2}\end{equation}
where $\vec{x}_n$ is the system state at time step $n\in\mathbb{N}_0$.  That is, $\tilde{g}_\theta(\vec{x})$ is the average cost under policy $\theta$ given that we replace the cost function $c(\vec{x})$ by $\tilde{c}(\vec{x},a)$.  Then
\begin{equation}\tilde{g}_\theta(\vec{x})=g_\theta(\vec{x})\;\;\;\;\forall\vec{x}\in S,\label{prop2_eqn}\end{equation}
where $g_\theta(\vec{x})$ is the average cost defined in (\ref{avg_cost}).
\label{prop2}\end{proposition}

The intuition for Proposition \ref{prop2} is as follows.  If the repairer adopts a completely passive policy and never repairs any machines, then the long-run average cost is trivially equal to $\sum_{j\in M}f_j(K_j)$. Suppose we modify the passive policy by choosing a specific state $\vec{x}=(i,(x_1,...,x_m))$ with $x_i\geq 1$, and specifying that the repairer should always choose to repair machine $i$ under this state. Given that the repairer chooses this action at a particular time step, there is a probability $\mu_i$ that the machine's state transitions to $x_i-1$ at the next time step. Following this transition, we obtain a cost rate advantage of $[f_i(x_i)-f_i(x_i-1)]$ compared to the passive policy (under which the machine would have remained in state $x_i$). This advantage persists until the machine next degrades, which is expected to happen $1/\lambda_i$ time steps later. At that point, the cost rate advantage becomes $[f_i(x_i+1)-f_i(x_i)]$ and this also continues for an expected $1/\lambda_i$ time steps. By considering a sequence of further degradations until machine $i$ eventually degrades to state $K_i$, we can quantify the advantage of choosing the repair action under state $\vec{x}$ as $\mu_i[f_i(K_i)-f_i(x_i-1)]/\lambda_i$. It follows that this quantity should be deducted from the average cost of the passive policy every time the repair action is chosen at $\vec{x}$. In Appendix \ref{AppB} we give a proof using more rigorous arguments, by considering the detailed balance equations of the continuous-time Markov chain induced by a fixed stationary policy. 

Proposition \ref{prop2} enables us to consider an equivalent MDP formulation based on maximizing rewards, rather than minimizing costs.  The reward under a particular state $\vec{x}\in S$ is equal to the absolute value of the negative term in (\ref{cost_eqn2}) if the repairer performs a repair under $\vec{x}$, and zero otherwise.  The next result formalizes this equivalence.

\begin{corollary}Consider an MDP formulation in which the single-step reward for being in state $\vec{x}=(i,(x_1,...,x_m))$ and choosing action $a\in A_{\vec{x}}$, denoted by $r(\vec{x},a)$, is given by
\begin{equation}r(\vec{x},a)=\begin{cases}\dfrac{\mu_i}{\lambda_i}[f_i(K_i)-f_i(x_i-1)],&\text{ if }i\in M,\;x_i\geq 1\text{ and }a=i,\\[16pt]
0,&\text{ otherwise,}\end{cases}\label{reward_eqn}\end{equation}
and the objective is to maximize the long-run average reward, defined as 
\begin{equation}u_\theta(\vec{x}):=\limsup_{t\rightarrow\infty}t^{-1}\mathbb{E}_\theta \left[\sum_{n=0}^{t-1} r(\vec{x}_n,\theta(\vec{x}_n))\;\Big |\;\vec{x}_0=\vec{x}\right],\;\;\;\;\vec{x}\in S.\label{avg_reward}\end{equation}
All other aspects of the MDP formulation, including states, actions and transition probabilities, are the same as described in Section \ref{sec:formulation}.  Then any stationary policy which maximizes the long-run average reward (\ref{avg_reward}) in the new MDP formulation also minimizes the long-run average cost (\ref{avg_cost}) in the previous MDP formulation.
\label{corollary}\end{corollary}

\textit{Proof. }By Proposition \ref{prop2} it is sufficient to show that maximizing the long-run average reward (\ref{avg_reward}) is equivalent to minimizing the modified average cost (\ref{avg_cost2}) in which the alternative cost function $\tilde{c}(\vec{x},a)$ is used.  Indeed, the cost function $\tilde{c}$ defined in (\ref{cost_eqn2}) includes a constant term $\sum_{j\in M}f_j(K_j)$ that is incurred under all states.  It follows that $\tilde{g}_\theta(\vec{x})=\sum_{j\in M}f_j(K_j)-u_\theta(\vec{x})$, and hence minimizing $\tilde{g}_\theta(\vec{x})$ is equivalent to maximizing $u_\theta(\vec{x})$.  \hfill $\Box$\\

Recall that, by Proposition \ref{prop1}, the optimal average cost $g^*$ is independent of the initial state. Therefore Corollary \ref{corollary} implies that the optimal average reward $u^*=\sup_\theta u_\theta(\vec{x})$ also has no dependence on $\vec{x}$. 

The main advantage of using the reward-based formulation (\ref{avg_reward}) is that it allows the reward under a particular state $\vec{x}$ to be computed purely in terms of the parameters for the repairer's current location $i$ (with the reward being zero if $i$ is an intermediate stage, rather than a machine), whereas the cost formulation (\ref{avg_cost}) involves aggregating cost rates over all of the damaged machines under a particular state. This greatly assists the development of an index-based heuristic policy. A similar approach was used by (\cite{Duenyas1996}), who formulated a problem based on minimizing average costs but then developed an index heuristic based on maximizing rewards. In the infinite-state problem considered by (\cite{Duenyas1996}), the reward function takes the much simpler form $c_i\mu_i$ (where, in their notation, $c_i$ is the linear cost rate associated with queue $i$).  The reward function (\ref{avg_reward}) in our model is more complicated as it depends on the level of degradation, $x_i$, and also on the degradation rate $\lambda_i$.

We now proceed to describe the development of an index-based heuristic for our MDP, focusing on the single-step reward function $r(\vec{x},a)$ rather than the cost function $c(\vec{x})$ from Section \ref{sec:formulation}. The basic principle underlying the heuristic is that, under any given state $\vec{x}\in S$, the repairer calculates `average reward rates' (referred to as \emph{indices}) associated with different possible actions, and then chooses an action that maximizes the index. These indices are calculated over a short time horizon, by assuming that the repairer will commit to fully repairing the next machine that it visits. More specifically, if the repairer is located at a machine $i\in M$ then the index for remaining at node $i$ is calculated by assuming that the machine remains there until $x_i=0$, whereas the index for switching to an alternative node $j\neq i$ is calculated by assuming that the machine moves directly to node $j$ and then remains there until $x_j=0$. If the repairer is at an intermediate stage $i\in N$ then the procedure is similar, except we only calculate indices for moving to machines in $j\in M$ and do not calculate an index for remaining at $i$. These indices inform the action chosen by the repairer at the current time step, but the indices are recalculated at every time step, so (for example) the repairer might begin moving towards a particular machine $j$ but change direction and move towards a different machine $l$ if the latter machine experiences a degradation event.


Firstly, for a particular machine $i\in M$, let $T_i(k)$ denote the random amount of time taken for the machine to return to its pristine state, given that it begins in state $k\in\{0,1,...,K_i\}$ and the repairer begins to repair it immediately and continues without any interruption.  Also, let $R_i(k)$ denote the cumulative reward earned while these repairs are in progress, i.e. until the machine returns to the pristine state.  We define both $T_i(0)$ and $R_i(0)$ to be zero.  Expressions for $\mathbb{E}[T_i(k)]$ and $\mathbb{E}[R_i(k)]$ for $k\geq 1$ can be derived quite easily in terms of the system parameters, by considering the dynamics of an $M/M/1$ queue with finite capacity. These expressions are given in the next lemma.

\begin{lemma}For each machine $i\in M$ and $k\in\{1,...,K_i\}$, we have
\begin{equation}\mathbb{E}[R_i(k)]=\sum_{p=1}^{K_i}C_{ip}(k)s_i(k),\;\;\;\;\;\;\;\;\mathbb{E}[T_i(k)]=\sum_{p=1}^{K_i}C_{ip}(k),\label{lemma_equation}\end{equation}
where $s_i(k):=\mu_i[f_i(K_i)-f_i(k-1)]/\lambda_i$ is the reward rate in state $k$ and the coefficients $C_{ip}(k)$ are given by
$$C_{ip}(k)=\begin{cases}\dfrac{1}{\mu_i^p}\;\mathlarger{\sum}\limits_{r=0}^{p-1}\lambda_i^{p-1-r}\mu_i^r,&p=1,2,...,k,\\[20pt]
\dfrac{1}{\mu_i^p}\;\mathlarger{\sum}\limits_{r=0}^{k-1}\lambda_i^{p-1-r}\mu_i^r,&p=k+1,k+2,...,K_i.\end{cases}$$
\label{lemma1}\end{lemma}

\textit{Proof. }Given that the repairer performs uninterrupted repairs at machine $i$, standard arguments based on the dynamics of continuous-time Markov chains imply that
\begin{equation}\mathbb{E}[R_i(k)]=\begin{cases}\dfrac{s_i(k)}{\lambda_i+\mu_i}+\dfrac{\lambda_i}{\lambda_i+\mu_i}\mathbb{E}[R_i(k+1)]+\dfrac{\mu_i}{\lambda_i+\mu_i}\mathbb{E}[R_i(k-1)],&k=1,2,...,K_i-1,\\[16pt]
\dfrac{s_i(K_i)}{\mu_i}+\mathbb{E}[R_i(K_i-1)],&k=K_i.\end{cases}\label{lemma_proof_eqn}\end{equation}
Thus, we have a system of $K_i$ equations in $K_i$ unknowns and can solve these to find the values of $\mathbb{E}[R_i(1)],...,\mathbb{E}[R_i(K_i)]$ stated in (\ref{lemma_equation}).

For the $\mathbb{E}[T_i(k)]$ values, the equations (\ref{lemma_proof_eqn}) take exactly the same form except that we replace each $R_i(k)$ by $T_i(k)$ and each reward rate $s_i(k)$ should be replaced by $1$.  The equations can then be solved to give the results stated in (\ref{lemma_equation}).\hfill $\Box$\\

Suppose the system is in state $\vec{x}=(i,(x_1,...,x_m))$, where $i\in M$. We define the index for remaining at machine $i$, denoted by $\Phi_i^{\text{stay}}(x_i)$, as
\begin{equation}\Phi_i^{\text{stay}}(x_i)=\begin{cases}
\dfrac{\mathbb{E}[R_i(x_i)]}{\mathbb{E}[T_i(x_i)]},&\text{ if }x_i\in\{1,...,K_i\},\\[12pt]
0,&\text{ if }x_i=0.\end{cases}\label{stay_index}\end{equation}

Next, let the repairer's current location $i$ be an arbitrary node in $V$, and let $D_{ij}$ denote the random amount of time needed for the repairer to move from node $i$ to a machine $j\neq i$, assuming that it follows a shortest path in the network. We define the index for moving to machine $j$, denoted by $\Phi_{ij}^{\text{move}}(x_j)$, as
\begin{equation}\Phi_{ij}^{\text{move}}(x_j):=\sum_{k=x_j}^{K_j}\mathbb{P}(X_j=k)\;\frac{\mathbb{E}[R_j(k)]}{\mathbb{E}[D_{ij}|X_j=k]+\mathbb{E}[T_j(k)]},\label{move_index}\end{equation}
where $X_j$ is the (random) state of machine $j$ upon arrival of the repairer at node $j$, suppressing the dependence on $x_j$ for convenience. 
We note that $X_j$ is distributed as a sum of i.i.d. geometric random variables and therefore follows a truncated version of the negative binomial distribution, as follows:
$$\mathbb{P}(X_j=k)=\begin{cases}\dbinom{d_{ij}+k-x_j-1}{d_{ij}-1}\left(\dfrac{\tau}{\lambda_j+\tau}\right)^{d_{ij}}\left(\dfrac{\lambda_j}{\lambda_j+\tau}\right)^{k-x_j},&k=x_j,...,K_j-1,\\[16pt]
1-\mathlarger{\sum}\limits_{r=x_j}^{K_j-1} \mathbb{P}(X_j=r),&k=K_j,\end{cases}$$
where $d_{ij}$ is the number of nodes that the repairer must traverse on a shortest path from $i$ to $j$. For $k=x_j,...,K_j-1$ we have
$$\mathbb{E}[D_{ij}|X_j=k]=\frac{d_{ij}+k-x_j}{\tau+\lambda_j},$$
while for $k=K_j$ we use the law of total expectation, which implies that
$$\mathbb{E}[D_{ij}]=\frac{d_{ij}}{\tau}=\sum_{k=x_j}^{K_j}\mathbb{P}(X_j=k)\mathbb{E}[D_{ij}|X_j=k],$$
and hence
$$\mathbb{E}[D_{ij}|X_j=K_j]=\frac{(d_{ij}/\tau)-\sum_{k=x_j}^{K_j-1}\mathbb{P}(X_j=k)\mathbb{E}[D_{ij}|X_j=k]}{1-\sum_{k=x_j}^{K_j-1}\mathbb{P}(X_j=k)}.$$

We may interpret the index $\Phi_i^{\text{stay}}(x_i)$ as an approximate measure of the average reward per unit time that the repairer can obtain by remaining at machine $i$ until $x_i=0$, while $\Phi_{ij}^{\text{move}}(x_j)$ is the approximate average reward per unit time for switching to machine $j$ and remaining there until $x_j=0$. These are only approximate measures because, in fact, random variables such as $R_j(k)$ and $T_j(k)$ are not independent of each other and so the quotients that appear in (\ref{stay_index}) and (\ref{move_index}) cannot be regarded as exact closed forms for the average rewards that would be obtained during the relevant time periods.

There is one further index quantity that should be defined. Suppose the system is in state $\vec{x}=(i,(x_1,...,x_m))$ and let $\Phi_{ij}^{\text{wait}}(x_j)$ be defined in a similar way to $\Phi_{ij}^{\text{move}}(x_j)$ except that it represents the (approximate) average reward that the repairer would gain by waiting for one further degradation event to occur at machine $j\neq i$, then carrying out the same actions as in the case of the previous index (that is, switching to $j$ and repairing it until $x_j=0$). If $x_j=K_j$ then no further degradations at $j$ are possible, but we still suppose that the repairer waits for $1/\lambda_j$ time units before moving to $j$. Specifically, we define
\begin{align}\Phi_{ij}^{\text{wait}}(x_j):=&\sum_{k=x_j}^{K_j-1}\mathbb{P}(X_j=k)\;\frac{\mathbb{E}[R_j(k+1)]}{1/\lambda_j+\mathbb{E}[D_{ij}|X_j=k]+\mathbb{E}[T_j(k+1)]},\nonumber\\
&\;\;+\mathbb{P}(X_j=K_j)\;\frac{\mathbb{E}[R_j(K_j)]}{1/\lambda_j+\mathbb{E}[D_{ij}|X_j=K_j]+\mathbb{E}[T_j(K_j)]}.\label{wait_index}\end{align}
Note that in (\ref{wait_index}), $X_j$ has the same definition as before (i.e. the state of machine $j$ when the repairer arrives, given that the repairer begins moving immediately), but we calculate the index using $R_j(k+1)$ and $T_j(k+1)$ rather than $R_j(k)$ and $T_j(k)$ because the repairer waits for an extra degradation event before beginning to move. If $\Phi_{ij}^{\text{move}}(x_j)<\Phi_{ij}^{\text{wait}}(x_j)$ then we can interpret this as an incentive for the repairer to wait longer before moving to node $j$, since the occurrence of a further degradation event will increase the average reward that can be earned by switching to $j$ and repairing the machine fully. Note that if $x_j=K_j$ then it is guaranteed that $\Phi_{ij}^{\text{move}}(x_j)>\Phi_{ij}^{\text{wait}}(x_j)$, so there is no incentive to wait in that case. In our index heuristic, we prevent the repairer from switching from machine $i$ to machine $j$ if $\Phi_{ij}^{\text{move}}(x_j)<\Phi_{ij}^{\text{wait}}(x_j)$.

Before defining the index heuristic, we discuss how it should make decisions under states with $x_j=0$ for all $j\in M$.  States of this form are quite different from others in $S$, as the repairer does not have any work to do.  However, one might suppose that the repairer should still take some kind of preparatory action in order to ensure that it is able to respond to any future jobs in the most efficient manner.  Accordingly, we aim to identify a particular node in the network that can be regarded as an `optimal position' for the repairer while it is waiting for new tasks to materialize. Specifically, for each node $i\in V$, let $\Psi(i)$ be defined as follows:


{\begin{equation}\Psi(i):=\sum_{j\in M}\left(\dfrac{\lambda_j}{\lambda_1+...+\lambda_m}\right)\mathbb{E}[D_{ij}],\label{new_psi_formula}\end{equation}
Thus, $\Psi(i)$ is the expected amount of time spent moving to the next machine that degrades, where the expectation is taken with respect to which machine this will be. We define the repairer's optimal idling position as the node $i\in V$ that minimizes $\Psi(i)$.  If there are multiple such nodes, then an arbitrary choice of idling position can be made.

The next definition specifies how actions are chosen by the index heuristic.

{\begin{definition}(Index policy.) Let $\theta^{\text{IND}}$ be a stationary policy that uses the following rules to select an action under state $\vec{x}=(i,(x_1,...,x_m))\in S$:\\
\vspace{-3mm}
\begin{enumerate}
\item If $x_j=0$ for all $j\in M$ (that is, all machines are in the pristine condition) then let $i^*\in\argmin_{i\in V}\Psi(i)$. The repairer's action should be to move to the first node on the shortest path from $i$ to $i^*$ (if $i\neq i^*$), or to remain at $i^*$ (if $i=i^*$).
\item If $x_j\geq 1$ for at least one $j\in M$, then consider two possible cases:
\begin{enumerate}
\item If $i\in M$ (that is, the repairer is at a machine) then let the set $J$ be defined by
\begin{equation}J:=\left\{j\in M\setminus \{i\}\;:\;\Phi_{ij}^{\textup{move}}(x_j)\geq \Phi_{ij}^{\textup{wait}}(x_j)\right\}.\label{J_formula}\end{equation}
If $J$ is non-empty, then let $j^*$ be the machine in $J$ that maximizes $\Phi_{ij}^{\textup{move}}(x_j)$. If $\Phi_{ij^*}^{\textup{move}}(x_{j^*})>\Phi_i^{\textup{stay}}(x_i)$ then the repairer's action should be to move to the first node on the shortest path from $i$ to $j^*$. On the other hand, if either $J$ is empty or $\Phi_{ij^*}^{\textup{move}}(x_{j^*})\leq\Phi_i^{\textup{stay}}(x_i)$, then the repairer's action should be to remain at node $i$.
\item If $i\in N$ (that is, the repairer is at an intermediate stage), then let $j^*$ be the machine that maximizes $\Phi_{ij}^{\textup{move}}(x_j)$.  The repairer's action should be to move to the first node on the shortest path from $i$ to $j^*$.
\end{enumerate}
%
\end{enumerate}
\vspace{3mm}
In cases where two or more actions are viable choices according to the above rules, we assume that the action chosen under $\theta^{\text{IND}}$ is based on a pre-determined priority ordering of the nodes in $V$.  We refer to $\theta^{\text{IND}}$ as the index heuristic policy.
\label{index_definition}\end{definition}

It is clearly of interest to investigate how well the index heuristic policy performs in comparison to an optimal policy, and also how well it performs relative to other heuristic policies. In Section \ref{sec:numerical} we provide results from numerical experiments in order to demonstrate its performance empirically. Theoretical performance guarantees (e.g. suboptimality bounds) are difficult to establish in the full generality of our model.  However, we are able to show that the index policy is optimal in certain special cases of our problem. These special cases are discussed in the next two propositions. 

\begin{proposition}Suppose the following 3 conditions are satisfied:
\begin{enumerate}[(i)]
\item All machines are directly connected to each other; that is, $V$ is a complete graph.
\item Each machine has only two possible states; that is, $K_j=1$ for each $j\in M$.
\item The degradation rates, repair rates and cost functions are the same for all machines. That is, we have $\lambda_1=...=\lambda_m$, $\mu_1=...=\mu_m$ and $f_1(1)=...=f_m(1)$. 
\end{enumerate}
Then the index policy is optimal.\label{prop3}\end{proposition}

The proof of Proposition \ref{prop3} can be found in Appendix \ref{AppC}. This proposition describes a homogeneous type of problem in which all machines have the same parameter values, and furthermore assumes that the state of a machine is a binary variable, so that it is effectively either `working' or `not working'. In this type of binary model, the cost function $f_j(k)$ for machine $j\in M$ is determined by a single positive number, $f_j(1)$. While binary models such as these might appear quite simplistic, they are still \textcolor{black}{potentially relevant to real-world applications and have been studied in the field of reliability modeling (see \cite{Lagos2020,Cancela2022}),} and we include them in our numerical experiments in Section \ref{sec:numerical}. Another \textcolor{black}{important condition} in Proposition \ref{prop3} is that all nodes are connected to each other, so that we can assume there are no intermediate stages (as if the network has intermediate stages, these will clearly never be visited under an optimal policy). 

\textcolor{black}{We have been able to confirm using experiments that all three conditions in Proposition \ref{prop3} are necessary; that is, if any of them are omitted then the index policy may not be optimal. Details of these experiments can be found in Appendix \ref{Appagainstprop3}. However, if condition (i) is relaxed (that is, machines are no longer directly connected to each other) then the index policy may still be optimal, subject to another condition on the system parameters. To illustrate this we introduce another type of network, referred to as a `star network', which we define below.}

\begin{definition}Suppose there exists an intermediate stage $s\in N$ such that, for any two distinct machines $i,j\in M$, the unique shortest path from $i$ to $j$ is a path of length $2r$ (where $r\in\mathbb{N}$) that passes through node $s$. Then the network $V$ is called a \emph{star network}, with $s$ and $r$ referred to as the \emph{center} and \emph{radius} respectively.\label{star_def}\end{definition}

In a star network, all machines are equidistant from each other (see Figure \ref{star_figure}). The next proposition shows that the index policy is optimal in a star network under certain conditions, including two of the conditions from Proposition \ref{prop3}.

\begin{proposition}\textcolor{black}{Let $V$ be a star network with radius $r\geq 1$, and suppose that conditions (ii) and (iii) from Proposition \ref{prop3} hold.} Then the index policy is optimal, provided that
\begin{equation}\tau>2r\lambda,\label{prop4_eqn}\end{equation}
where $\lambda$ is the common degradation rate for all machines.\label{prop4}\end{proposition}

Proof of Proposition \ref{prop4} can be found in Appendix \ref{AppD}. We note that the condition in \ref{prop4_eqn} is equivalent to $2r/\tau<1/\lambda$, which states that the expected amount of time for the repairer to move from one machine to another is smaller than the expected amount of time for a degradation event to occur at a non-degraded machine. \textcolor{black}{This condition should often be satisfied in applications where failures or degradation events occur at a relatively slow rate, so that the repairer can frequently change location without additional machines degrading in the meantime. For example, this situation could arise in computing infrastructures where hardware faults occur on a much slower timescale than communication or switching delays.} Under the stated conditions, the index policy is a `greedy' policy that always prioritizes the nearest degraded machine (or moves towards the center of the network if there are no degraded machines), and we can show that such a policy is also optimal. Please see Appendix \ref{AppD} for full details of the proof.

\begin{figure}[hbtp]
    \begin{center}
        \includegraphics[scale = 0.5]{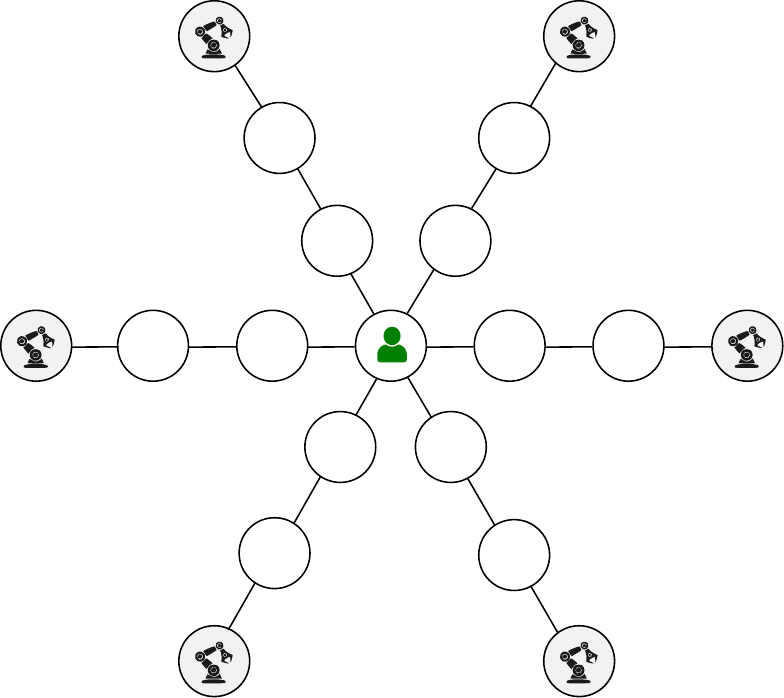}
    \end{center}
    \caption{A star network with 6 machines and a radius $r=3$.}
    \label{star_figure}
\end{figure}

\section{Online policy improvement}\label{sec:polimp}

The index heuristic developed in Section \ref{sec:index} yields a dynamic policy that, at any given time, either chooses a machine to prioritize next (if at least one machine requires repair), or directs the repairer to an optimal idling position (if no machines require repair). Importantly, in the former case, the indices in (\ref{stay_index})-(\ref{wait_index}) are recalculated at each time step, so that the repairer can take advantage of the ability to interrupt switching and repair times as permitted under our problem formulation. However, a weakness of the index heuristic is that it works in a `short-sighted' (myopic) way, as the decision to prioritize a particular machine is made without taking into account that machine's proximity to other machines in the network. It would be preferable to use a more sophisticated approach that takes the network topology into account when making decisions.

As discussed in Section \ref{sec:formulation}, an optimal policy can be characterized by a set of values $\{h(\vec{x})\}_{\vec{x}\in S}$ satisfying the optimality equations in (\ref{opt_eqs}). The well-known method of \emph{policy improvement} (PI), applicable to finite-state MDPs under the average cost criterion, works by starting from an arbitrary base policy and performing successive improvement steps, yielding a sequence of policies that eventually converges to an optimal one (\cite{Puterman1994}). An optimal policy will naturally make long-sighted decisions that depend on the network topology, and any step of PI is therefore helpful in obtaining a policy that avoids short-sighted behavior. Although multiple improvement steps may be needed in order to obtain an optimal policy, it has been observed that the first improvement step often yields the largest improvement (\cite{Tijms2003}). Furthermore, heuristics based on applying one step of policy improvement have been successfully applied to various types of queueing control problems (\cite{krishnan1990,argon2009,shone2020queueing}). 


As stated in Proposition \ref{prop1}, an optimal policy for our system is always unichain. However, the policy given by the index heuristic in Section \ref{sec:index} is not guaranteed to be unichain and indeed we have found that it can have a multichain structure in some problem instances. It will therefore be desirable to introduce a modified version of the index heuristic in which the resulting policy must be unichain, in order to avoid certain complications that arise when applying PI to a multichain base policy. The modified index heuristic is defined below.

\begin{definition}(Modified index policy.) Let $\theta^{\textup{M-IND}}$ denote a stationary policy that uses the following rules to select an action under state $\vec{x}=(i,(x_1,...,x_m))\in S$:\\
\vspace{-3mm}
\begin{enumerate}
\item If $x_j=K_j$ for all $j\in M$ (that is, all machines are at their maximum state of degradation), then let $j^*\in M$ be the smallest-indexed machine satisfying
\begin{equation}j^*\in\argmax_{j\in M}\left\{\frac{\mathbb{E}[R_j(K_j)]}{\mathbb{E}[T_j(K_j)]}\right\}.\label{full_system_node}\end{equation}
Then the action chosen under state $\vec{x}$ should either be to move to the first node on the shortest path from $i$ to $j^*$ (if $i\neq j^*$), or to stay at node $i$ (if $i=j^*$). 
\item Otherwise, the action chosen under state $\vec{x}$ should be the same action given by the unmodified index policy $\theta^{\textup{IND}}$ in Definition \ref{index_definition}.\\
\end{enumerate}
\vspace{-3mm}
\label{modified_index_definition}\end{definition}

Note that under any admissible policy, it is always possible for the system to reach a state with $x_j=K_j$ for all $j\in M$, since this can occur randomly via an unbroken sequence of machine degradations. Under the modified index policy $\theta^{\textup{M-IND}}$, the machine $j^*$ satisfying (\ref{full_system_node}) is unique and has no dependence on the system state $\vec{x}$. Thus, the state $(j^*,(K_1,...,K_m))$ is accessible from any state under $\theta^{\textup{M-IND}}$. This ensures that there cannot be more than one positive recurrent class of states under $\theta^{\textup{M-IND}}$, and therefore the policy is indeed unichain.

In general, the performance of $\theta^{\textup{M-IND}}$ may be worse than that of $\theta^{\textup{IND}}$, since the rule that it uses under states with $x_j=K_j$ for all $j\in M$ is not consistent with the rules used under other states. The only purpose of introducing $\theta^{\textup{M-IND}}$ is so that it can be used as a convenient base policy in our PI method. In the computational experiments that we report in Section \ref{sec:numerical} we use the original definition of the index policy (Definition \ref{index_definition}) when evaluating the performance of the index heuristic, and compare its performance with that of the simulation-based PI method.

\textcolor{black}{In this section we propose an approximate one-step PI method, referred to as \emph{online policy improvement} (OPI) for short, that works by simulating the random evolution of the system in the style of a reinforcement learning (RL) algorithm. RL refers to a broad class of methods in which decision-makers adaptively learn and improve their policies by interacting with a real or simulated environment, gaining experience by observing the costs or rewards from choosing different actions at the states that they encounter (\cite{Sutton2018,Powell2022}). Our OPI approach is divided into two stages, which we call the \emph{offline} part and the \emph{online} part. In the offline part, we repeatedly simulate the evolution of the system under the modified index policy $\theta^{\textup{M-IND}}$ in order to estimate the value function $h_{\theta^{\textup{M-IND}}}(\vec{x})$ under this policy. This closely resembles the method of \emph{temporal difference learning} (TD-learning), discussed extensively in classical RL texts (see \cite{Sutton2018}, pp. 119-138), although our method modifies the classical approach by storing value function estimates only for a limited number of `exemplar' states and allowing each simulation to continue until it reaches one of these exemplar states.} 
Subsequently, in the online part, we simulate the evolution of the system and, at each discrete time step, choose an action by applying an improvement step to the base policy $\theta^{\textup{M-IND}}$. Actions must be selected at the beginning of each time step, but we assume that after an action is selected, the next transition of the system does not occur until a further $\delta$ time units have elapsed (here, $\delta$ represents the length of a time step and is analogous to the uniformization parameter $\Delta$ used in Section \ref{sec:formulation}).
During this $\delta$-length time period we run nested simulation experiments (within the main simulation) with the aim of improving the estimates of $h_{\theta^{\textup{M-IND}}}(\vec{x}')$ for states $\vec{x}'$ that are near the current state $\vec{x}$. The results of the nested online simulations supplement the results of the offline simulations in order to estimate the value function as accurately as possible within a neighborhood of the current state. 
Once the time limit $\delta$ is reached, we use a criterion based on a confidence interval to determine whether to select the best `improving' action identified during the $\delta$-length period, or whether to select the action prescribed by the base policy (the latter option is chosen in situations where we cannot say with sufficient confidence which action is best). 

\textcolor{black}{The method that we use to find improving actions in the online part may be recognized as a type of \emph{rollout} approach, as it involves comparing different actions under the current state by simulating future trajectories following each possible action, assuming that a fixed base policy will be used to choose actions at future time steps. Each simulated trajectory ends when the system reaches a state for which a value function estimate is stored, and the results from all simulated trajectories are aggregated to obtain estimates of the relative future costs under all possible actions. Rollout and direct lookahead methods have been extensively discussed in the recent RL literature (\cite{Bertsekas2020,Powell2022} and offer the ability to make locally optimal decisions via simulation, thereby avoiding the modeling and computational challenges of other RL methods that involve constructing approximate value functions over the entire state space.}

Some additional explanations of how the offline and online parts work are provided in the subsections below. To simplify notation we use $\theta$ rather than $\theta^{\textup{M-IND}}$ to denote the modified index policy, referred to as the base policy. Due to space restrictions, we defer full details of the algorithms (including pseudocode) to Appendix \ref{AppPI}.

\subsection{Offline part}

In the offline part we simulate the evolution of the discrete-time MDP formulated in Section \ref{sec:formulation}. In the preparatory stage, we conduct an initial simulation in order to identify a subset of states in $Z\subset S$ that are visited often under the base policy $\theta$, and also obtain an estimate $\hat{g}_{\theta}$ of the average cost under this policy. We then initialize an array $U$, used to store information about states visited during the simulation. Initially, the array $U$ contains a single state which we denote by $\vec{u}_0$. 

Next, we conduct a large number of variable-length simulations, where each simulation begins from one of the states in $Z$ and continues until the system reaches one of the states in $U$. In general, suppose the $r^{\text{th}}$ variable-length simulation (for $r=1,2,...$) begins from state $\vec{z}_r\in Z$ and ends in state $\vec{u}_r\in U$. Also, let $\vec{x}$ be one of the first $p$ states visited during this simulation, where $p$ is a small integer (we use $p=5$ in our experiments). Then we calculate an estimate of $h_{\theta}(\vec{x})$, denoted by $\hat{h}_\theta^{(s)}(\vec{x})$ (where $s$ is a counter of how many estimates have been obtained for $h_{\theta}(\vec{x})$ so far), as follows:
\begin{equation}\hat{h}_\theta^{(s)}(\vec{x})=C_r(\vec{z}_r,\vec{u}_r)+\hat{h}_\theta(\vec{u}_r)-\hat{g}_\theta T_r(\vec{z}_r,\vec{u}_r),\label{offline_equation}\end{equation}
where $C_r(\vec{z}_r,\vec{u}_r)$ is the total cost incurred during the variable-length simulation, $T_r(\vec{z}_r,\vec{u}_r)$ is the total number of time steps in the simulation, and $\hat{h}_\theta(\vec{u}_r)$ is the latest available estimate of $h_{\theta}(\vec{u_r})$ which is calculated using all of the estimates $h_{\theta}^{(1)}(\vec{u}_r), h_{\theta}^{(2)}(\vec{u}_r)$, etc. Note that equation (\ref{offline_equation}) is derived using the well-known policy evaluation equations in dynamic programming, which are similar to the optimality equations (\ref{opt_eqs}). It can be explained intuitively as follows: in order to obtain an estimate of the relative value of being in state $\vec{x}$, we take the total cost incurred from starting in that state, which is $C_r(\vec{z}_r,\vec{u}_r)+\hat{h}_\theta(\vec{u}_r)$ (where $\hat{h}_\theta(\vec{u}_r)$ represents the `terminal cost' associated with finishing in state $\vec{u}_r$) and subtract the cost that would have been earned if we had incurred the long-run average cost $\hat{g}_\theta$ over $T_r(\vec{z}_r,\vec{u}_r)$ time steps.

After obtaining the new estimate $\hat{h}_\theta^{(s)}(\vec{x})$ we add state $\vec{x}$ to the array $U$ if it is not included already. We also store some other information about $\vec{x}$ in $U$, including the sum of the squared estimates $[\hat{h}_\theta^{(s)}(\vec{x})]^2$, which can be used to calculate a variance; this is needed in the online part of the algorithm (please see Appendix \ref{AppPI} for full details). As the offline part of the algorithm continues, the array $U$ becomes populated with more states and more information about those states. After a large number of iterations, the offline part ends. The intended outcome is that $U$ should include a subset of states that are visted often under the base policy and, for each state $\vec{x}\in U$, an estimate $\hat{h}_\theta(\vec{x})$ which should be a reasonably accurate estimate of the true value $h_\theta(\vec{x})$. We have been able to verify using experiments that the algorithm is able to output estimates $\hat{h}_\theta(\vec{x})$ that are very close to the true values $h_\theta(\vec{x})$ given by dynamic programming.

\subsection{Online part}

In this part we take the array $U$ obtained from the offline part as an input and simulate the evolution of the system again, but instead of using the actions given by the base policy $\theta$, we apply an approximate policy improvement step. This means that if $\vec{x}$ is the current state, we aim to select the action $a$ that minimizes
\begin{equation}\sum_{\vec{y}\in S}p_{\vec{x},\vec{y}}(a)\hat{h}_\theta(\vec{y}),\label{online_eq1}\end{equation}
where $\hat{h}_\theta(\vec{y})$ is the latest available estimate of $h_\theta(\vec{y})$. This implies that if $\vec{x}=(i,(x_1,...,x_m))$ denotes the current state then we require estimates $\hat{h}_\theta(\vec{y})$ for all possible `neighboring' states, i.e. all states $\vec{y}$ such that $p_{\vec{x},\vec{y}}(a)>0$ for some $a\in A_\vec{x}$. In general, let $\vec{x}^{j+}$ (resp. $\vec{x}^{j-}$) denote a state identical to $\vec{x}$ except that the component $x_j$ is increased (resp. decreased) by one, and let $\vec{x}^{\rightarrow j}$ denote a state identical to $\vec{x}$ except that the repairer's location is changed to $j$, where $j\neq i$. We also write $\vec{x}^{\rightarrow i}:=\vec{x}$. Note that degradation events occur independently of the action chosen, so any state of the form $\vec{x}^{j+}$ for some $j\in M$ can be ignored when it comes to minimizing (\ref{online_eq1}), because it occurs with the same probability under any action. Therefore, to find the action that minimizes (\ref{online_eq1}), we only need estimates $\hat{h}_\theta(\vec{y})$ for states $\vec{y}$ that are either (i) equal to $\vec{x}$, (ii) of the form $\vec{y}=\vec{x}^{\rightarrow j}$ where $j$ is a node adjacent to $i$, or (iii) of the form $\vec{x}^{i-}$ where $i\in M$ is the repairer's current location and $x_i\geq 1$. Let $F(\vec{x})$ denote the set of states satisfying either (i), (ii) or (iii).

Immediately after the system enters state $\vec{x}$, we identify the action that minimizes (\ref{online_eq1}) by using the latest available estimates $\hat{h}_\theta(\vec{y})$ for states $\vec{y}\in F(\vec{x})$. 
However, rather than simply selecting the action that minimizes this expression, instead we use the information in $U$ to calculate a confidence interval of the form $[\hat{h}_\theta^{-}(\vec{y}),\hat{h}_\theta^{+}(\vec{y})]$ for each $\vec{y}\in F(\vec{x})$. We then say that an action $a_1\in A_\vec{x}$ is better than another action $a_2\in A_\vec{x}$ if and only if 
\begin{equation}\sum_{\vec{y}\in S}p_{\vec{x},\vec{y}}(a_1)\hat{h}_\theta(\vec{y})<\sum_{\vec{y}\in S}p_{\vec{x},\vec{y}}(a_2)\hat{h}_\theta(\vec{y})\;\;\;\;\;\;\forall \hat{h}_\theta(\vec{y})\in [\hat{h}_\theta^{-}(\vec{y}),\hat{h}_\theta^{+}(\vec{y})],\;\vec{y}\in F(\vec{x}).\label{online_eq2}\end{equation}
In practice, (\ref{online_eq2}) yields some simplifications. For example, if the repairer is at an intermediate stage of the network ($i\notin M$) then the only possible actions are to switch to an adjacent node or to idle at the current node. After substituting the relevant transition probabilities into (\ref{online_eq2}), we find that $a_1$ should be preferred to $a_2$ if 
$$\hat{h}_\theta^{+}(\vec{x}^{\rightarrow a_1})<\hat{h}_\theta^{-}(\vec{x}^{\rightarrow a_2}),$$
which means that the confidence interval for $h_\theta(\vec{x}^{\rightarrow a_1})$ is non-overlapping with that of $h_\theta(\vec{x}^{\rightarrow a_2})$ and lies entirely below it. On the other hand, if the repairer is at a machine $i\in M$ with $x_i\geq 1$ then we also need to consider whether remaining at node $i$ (i.e. carrying out repairs) should be preferred to switching to $j\neq i$. In this case, remaining at $i$ is better than switching to $j$ if the inequality
$$\mu_i \left(\hat{h}_\theta(\vec{x}^{i-})-\hat{h}_\theta(\vec{x})\right)<\tau\left( \hat{h}_\theta(\vec{x}^{\rightarrow j})-\hat{h}_\theta(\vec{x})\right)$$
holds for all $\hat{h}_\theta(\vec{x}^{i-})\in[\hat{h}_\theta^{-}(\vec{x}^{i-}),\hat{h}_\theta^{+}(\vec{x}^{i-})]$, $\hat{h}_\theta(\vec{x}^{\rightarrow j})\in[\hat{h}_\theta^{-}(\vec{x}^{\rightarrow j}),\hat{h}_\theta^{+}(\vec{x}^{\rightarrow j})]$ and $\hat{h}_\theta(\vec{x})\in[\hat{h}_\theta^{-}(\vec{x}),\hat{h}_\theta^{+}(\vec{x})]$. If $\tau\geq\mu_i$ then this is equivalent to
\begin{equation}\mu_i\hat{h}_\theta^{+}(\vec{x}^{i-})-\tau\hat{h}_\theta^{-}(\vec{x}^{\rightarrow j})+(\tau-\mu_i)\hat{h}_\theta^{+}(\vec{x})<0.\label{online_eq3}\end{equation}
On the other hand, if $\tau<\mu_i$ then (\ref{online_eq3}) still applies except that $(\tau-\mu_i)$ is negative and therefore $\hat{h}_\theta^{+}(\vec{x})$ should be replaced with $\hat{h}_\theta^{-}(\vec{x})$ in order to obtain an upper bound for the left-hand side. The confidence intervals are calculated using the method of reliability weights (\cite{galassignu}}, Sec. 21.7); please see Appendix \ref{AppPI} for full details. If there exists an action $a^*\in A_\vec{x}$ which is better than all other actions according to the criteria described above, then this action is chosen by the improving policy. Otherwise, we cannot say with sufficient confidence which action should be chosen, and therefore as a `safety measure' we set $a^*=\theta(\vec{x})$, so that the action prescribed by the base policy is chosen. By reverting to the base policy in the latter situation, we take a conservative approach and aim to ensure that improving actions are chosen only when they are statistically likely to be better than the action of the base policy, which helps to ensure that the policy given by approximate PI is a genuine improvement over the base policy. This is akin to the method of `safe policy improvement', for which there is some existing literature (\cite{laroche2019safe, simao2023safe}).

After an action $a^*\in A_\vec{x}$ has been chosen according to the rules above, the next transition of the system does not occur until a further $\delta$ time units have elapsed, and we perform as many nested simulations as possible until the time limit $\delta$ is reached. In our numerical experiments in Section \ref{sec:numerical} we set $\delta=0.01$ seconds, so that state transitions occur very frequently. Each nested simulation proceeds as follows: first we simulate a single transition of the system under the chosen action $a^*$ and observe a new state $\vec{x}'$. Then, for each of the states $\vec{y}\in F(\vec{x}')$, we simulate the evolution of the system starting from $\vec{y}$, assuming that the base policy $\theta$ is used for selecting actions, until eventually a state contained in the array $U$ is reached. At this point an estimate is obtained for $h_\theta(\vec{y})$ using equation (\ref{offline_equation}), as in the offline part. The estimates $\hat{h}_\theta(\vec{y})$ obtained from the nested simulations supplement the estimates already obtained from the offline part. Thus, the purpose of the nested simulations is to anticipate what the next state of the system might be (by sampling a new state $\vec{x}'$) and enhance the PI policy's ability to make a good choice of action at the next state (by improving the estimates $\hat{h}_\theta(\vec{y})$ for states $\vec{y}\in F(\vec{x}')$). The nested simulations are carried out repeatedly until $\delta$ time units have elapsed, at which point the `actual' realization of the next state $\vec{x}'$ is observed and the main simulation continues. 

By simulating a large number of time steps, with actions chosen using the rules above, we can evaluate the performance of the approximate PI policy. It should be noted that as the simulation progresses we continue to acquire new estimates $\hat{h}_\theta(\vec{x})$ for states $\vec{x}$ that are visited and the confidence intervals $[\hat{h}_\theta^{-}(\vec{x}),\hat{h}_\theta^{+}(\vec{x})]$ become narrower, which means that we revert to the base policy less often. Thus, the policy becomes stronger as we acquire more experience from interacting with the environment. \textcolor{black}{This is a common theme in RL in general, as epitomized by the convergence properties of Q-learning  (\cite{Watkins1992}) and the local improvement properties of other RL methods, such as policy gradient algorithms (\cite{Kakade2001}).}\\

\section{Numerical results}\label{sec:numerical}

In this section we report the results of computational experiments involving 1,000 randomly generated problem instances. Each problem instance includes a randomly generated network layout and also a randomly sampled set of system parameters. We generate the network layout using the method described in \cite{tian2025}. Specifically, we construct a $5\times 5$ integer lattice with 25 nodes connected by horizontal and vertical edges. We then select a random subset of $m$ nodes, where $2 \leq m \leq 8$, and designate these as machines. Thus, each machine $i \in M$ is assigned coordinates $(a_i,b_i)$ with $a_i,b_i \in \{1,...,5\}$. The remaining (unselected) nodes serve as intermediate stages, although some of these may be redundant (in the sense that they would never be visited under an optimal policy) and can be ignored. Figure \ref{graph_example} shows an example in which machines are located at positions $(1,3)$, $(2,5)$ and $(3,1)$. It is clear from the figure that, given these machine locations, the intermediate stages at $(2,1)$, $(2,2)$, $(2,3)$ and $(2,4)$ are the only ones that the repairer needs to visit in order to move between the machines in the most efficient manner possible, so there is no need to include intermediate stages at any other lattice points. The length of the shortest path between machines $i$ and $j$ is the Manhattan distance $|a_i-a_j|+|b_i-b_j|$. 

\begin{figure}[htbp]
    \centering
    \begin{tikzpicture}
        \draw[step=1cm, gray, very thin] (0,0) grid (5,5); 

        \draw[thick,->] (0,5.2) -- (0,-0.3) node[below] {a}; 
        \draw[thick,->] (-0.3,5) -- (5.5,5) node[right] {b};  

        \foreach \x in {1,2,3,4,5} {
            \draw (\x,5) -- (\x,5.1); 
            \node at (\x-0.5,5.3) {\x}; 
        }

        \foreach \y in {1,2,3,4,5} {
            \draw (0,5-\y) -- (-0.1,5-\y); 
            \node at (-0.5,5-\y+0.5) {\y}; 
        }

        \draw[fill=lightgray, draw=black] (1-0.5,3-0.5) circle (6pt); 
        \draw[fill=lightgray, draw=black] (3-0.5,5-0.5) circle (6pt); 
        \draw[fill=lightgray, draw=black] (5-0.5,4-0.5) circle (6pt); 

        \draw[fill=white, draw=black] (1-0.5,4-0.5) circle (6pt); 
        \draw[fill=white, draw=black] (2-0.5,4-0.5) circle (6pt); 
        \draw[fill=white, draw=black] (3-0.5,4-0.5) circle (6pt); 
        \draw[fill=white, draw=black] (4-0.5,4-0.5) circle (6pt); 

        \draw[black] (1-0.5,3-0.29) -- (1-0.5,4-0.71); 
        \draw[black] (1-0.29,4-0.5) -- (2-0.71,4-0.5);
        \draw[black] (2-0.29,4-0.5) -- (3-0.71,4-0.5); 
        \draw[black] (3-0.29,4-0.5) -- (4-0.71,4-0.5); 
        \draw[black] (4-0.29,4-0.5) -- (5-0.71,4-0.5);
        \draw[black] (3-0.5,4-0.29) -- (3-0.5,5-0.71);

    \end{tikzpicture}

    \caption{A network with machines at $(1,3)$, $(2,5)$ and $(3,1)$ and intermediate stages at $(2,1)$, $(2,2)$, $(2,3)$ and $(2,4)$.} 
    \label{graph_example} 
\end{figure}
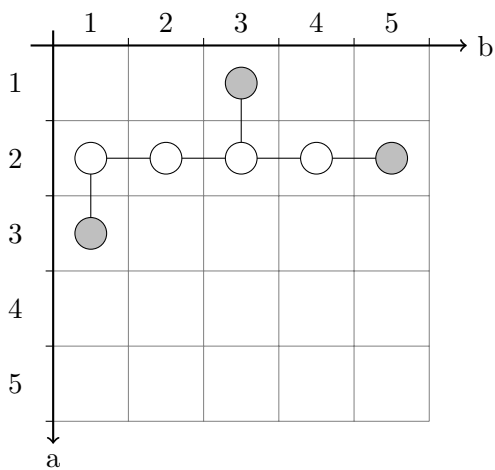

After the network layout has been generated for a specific problem instance, values of the system parameters are sampled randomly. We give a brief summary of this procedure here, but defer full details to Appendix \ref{AppGen}. We assume the number of possible degradation states, $K_i$, is the same for each machine $i\in M$, but sample this value randomly from the set $\{1,2,3,4,5\}$. We also consider three different cases for cost functions $f_i(x_i)$, as follows: (i) $f_i(x_i)$ increases linearly with the degradation state $x_i$; (ii) $f_i(x_i)$ is convex and increases quadratically with $x_i$; (iii) $f_i(x_i)$ increases linearly for $x_i\in\{0,...,K_i-1\}$, but then increases much more sharply when $x_i$ reaches the `failed' state $K_i$. The type of cost function is the same for each machine within a particular problem instance, but the specific parameters used in the cost functions are machine-dependent, so some machines are associated with higher degradation costs than others. The degradation rates $\lambda_i$ and repair rates $\mu_i$ are randomly generated in such a way that the overall traffic intensity, given by $\rho=\sum_i \lambda_i/\mu_i$, lies within the range $[0.1,1.5]$. Thus, we consider systems where the repairer is very under-utilized (i.e. $\rho$ is near zero) and also systems where it is in very high demand (i.e. $\rho>1$). We also define the parameter $\eta:=\tau/(\sum_{i\in M}\lambda_i)$ as a measure of how fast the repairer is able to switch between adjacent nodes relative to the total of the machine degradation rates, and generate the value of $\tau$ in order to consider a wide range of cases for the value of $\eta$. Please see Appendix \ref{AppGen} for full details of our parameter generation methods.

Recall that $\theta^{\text{IND}}$ denotes the index policy defined in Section \ref{sec:index} (see Definition \ref{index_definition}). We use $\theta^{\text{OPI}}$ to denote the policy followed during implementation of the online policy improvement (OPI) method discussed in Section \ref{sec:polimp}, with a time limit $\delta=0.01$ used for decision-making; that is, the system changes state every $0.01$ seconds and the OPI heuristic is forced to make a decision within that time (please see Appendix \ref{AppPI} for further details of how the OPI heuristic is implemented in our experiments). Within each problem instance, we use simulation experiments to assess the performances of $\theta^{\text{IND}}$ and $\theta^{\text{OPI}}$. Also, in instances with $2\leq m\leq 4$, we use dynamic programming (specifically, exact policy improvement) to calculate optimal long-run average costs and rewards and use these to evaluate suboptimality percentages for these heuristics. As $m$ increases, it becomes computationally infeasible to calculate an optimal policy using dynamic programming, so we do not calculate suboptimality percentages for systems with $5\leq m\leq 8$. 

Our experiments were carried out on an Intel(R) Core(TM) i5-6500T desktop computer, with 8GB of RAM, running Windows 10. The software used was Python 3.9.10, with the PyPy just-in-time compiler (\url{http://pypy.org}) used to speed up computations. During online implementation, the $\theta^{\text{POL}}$ and $\theta^{\text{IND}}$ policies require a negligible amount of time to make a decision at any individual time step. On the other hand, $\theta^{\text{OPI}}$ carries out as many nested simulations as it can within the time limit $\delta$ in order to improve the value function estimates for nearby states, as described in Section \ref{sec:polimp}. Since all of our heuristics can make decisions quickly in an online setting, the number of experiments that we are able to perform is limited not by the decision-making processes of the heuristics themselves, but rather by the number of time steps that we need to simulate in order to accurately evaluate the performances of the heuristics. We simulate each heuristic over $500,000$ time steps within each problem instance (please see Appendix \ref{AppSim} for further details).

The remainder of this section is organized as follows. In Section \ref{num:polling} we describe a simple heuristic based on polling systems that can be used to provide a useful benchmark for assessing the performances of the index and OPI heuristics. In Section \ref{num:comparison of the heuristic}, we evaluate the performances of our heuristic policies by calculating their suboptimalities relative to the optimal values in systems of modest size ($2\leq m\leq 4$) and also comparing them with the polling system policies. We also report the improvements that the OPI heuristic is able to achieve over the index heuristic in larger systems ($5\leq m\leq 8$). In Section \ref{num:comparison of parameters} we investigate the effects of varying the system parameters on the performances of our heuristic policies.

\subsection{Polling system heuristic}\label{num:polling}

In order to provide an additional benchmark for comparisons, we introduce a simple `polling system' heuristic that works by touring a subset of the $m$ machines in a cyclic pattern and always fully repairing each machine it visits before moving to the next one. \textcolor{black}{A polling system is a type of queueing model in which a single server provides service to multiple queues according to a pre-specified rule (\cite{Yechiali1993,Levy2002,Boon2011}). Generally, it visits the queues cyclically in a fixed order and remains at each queue until a certain condition is met---which might be that the queue becomes empty (this is known as the `exhaustive' regime) or that the queue no longer includes any jobs that were present when the server arrived (this is the `gated' regime). Standard assumptions of polling systems include random job arrival and service times, and also a random switching time to move between any two queues. In our problem, we can develop a polling heuristic by interpreting machine degradation events as job arrivals and supposing that the repairer visits a subset of the machines in a fixed cyclic pattern and provides exhaustive service at each one. For each possible subset, the order is chosen by minimizing the total travel distance over one cycle. Specifically,} consider a non-empty subset of machines $\mathcal M\subseteq M$ and let $\theta^{\text{POL}}(\mathcal M)$ denote a policy under which the repairer visits the machines in a cyclic pattern $(i_1,...i_{|\mathcal M|}, i_1, ...)$, where $i_1,....,i_{|\mathcal M|}$ are distinct nodes in $\mathcal M$. We assume that the order of visiting the machines is chosen in order to minimize the total travel distance,
\begin{equation}D(i_1,...,i_{|\mathcal M|}):=\sum_{j=1}^{|\mathcal M|-1}d(i_j,i_{j+1})+d(i_{|\mathcal M|},i_1),\label{travel_dist}\end{equation}
where $d(i,j)$ is the length of a shortest path between machines $i$ and $j$. In order to minimize this travel distance we must solve a `traveling salesman' problem (TSP), but since we only implement the polling system heuristic for systems with $m\leq 4$ the solution to the TSP is easy to find using brute force. Thus, for any given subset $\mathcal M\subseteq M$ we calculate the performance of $\theta^{\text{POL}}(\mathcal M)$ by first solving a TSP to find a sequence $(i_1,...,i_{|\mathcal M|})$ that minimizes the travel distance (\ref{travel_dist}), then simulating the average cost incurred when the repairer follows the resulting cyclic pattern, enforcing the rule that damaged machines are always fully repaired before the repairer moves to the next machine. For any given system, we can obtain a useful benchmark for our index and OPI heuristics by simulating the performances of all possible `polling system' policies $\theta^{\text{POL}}(\mathcal M)$ (taking into account all possible subsets $\mathcal M\subseteq M$) and then reporting the lowest average cost among all of these. We use $\theta^{\text{POL}}\in\argmin_{\mathcal M} \{g_{\theta^{\text{POL}}(\mathcal M)}(\vec{x})\}$ to denote the best-performing polling system heuristic within a particular instance (with ties broken arbitrarily).

\textcolor{black}{Although we use a polling heuristic primarily for benchmarking purposes, we note that heuristics based on polling systems have been successfully applied to dynamic routing problems based on order-picking in warehouses (\cite{Ran2020}) and mobile edge computing (\cite{Wang2023}). Various other applications of these methods are discussed in the survey by \cite{Borst2018}.}


\subsection{Performances of heuristic policies in systems with $2\leq m\leq 8$} 
\label{num:comparison of the heuristic}

In systems with $2\leq m\leq 4$ we measure the percentage suboptimalities of the heuristic policies $\theta^{\text{POL}}$, $\theta^{\text{IND}}$ and $\theta^{\text{OPI}}$ by comparing the average costs under these policies with the optimal values given by dynamic programming. In addition, we report the percentage suboptimalities when the average reward $u_\theta(\vec{x})$ defined in (\ref{avg_reward}) is used to measure the performance of a policy, rather than the average cost $g_\theta(\vec{x})$. Note that by Proposition \ref{prop2} and Corollary \ref{corollary} we have $g_\theta(\vec{x})=\sum_{j\in M} f_j(K_j)-u_\theta(\vec{x})$, and therefore the \emph{absolute} suboptimality is the same regardless of whether it is measured using costs or rewards. That is, for any policy $\theta$ and $\vec{x}\in S$ we have
$$g_\theta(\vec{x})-g^*=u^*-u_\theta(\vec{x}),$$
where $g^*$ and $u^*$ denote the average cost and average reward, respectively, under the optimal policy. However, the \emph{percentage} suboptimalities under the cost and reward formulations are $100\times(g_\theta(\vec{x})-g^*)/g^*$ and $100\times(u^*-u_\theta(\vec{x}))/u^*$ respectively, and these are different because $g^*\neq u^*$ in general. Although our original problem formulation in Section \ref{sec:formulation} was given in terms of costs, we would argue that measuring percentage suboptimalities in terms of rewards is also useful, because $u^*$ is more robust than $g^*$ with respect to the problem size (represented by the number of machines $m$). To see this, note that as $m$ increases, it becomes increasingly likely that some machines are never visited by the repairer; that is, they are left permanently broken. These non-visited machines cause the optimal average cost $g^*$ to increase, but they have no effect on $u^*$. Hence, the ratio $(g_\theta(\vec{x})-g^*)/g^*$ decreases purely because of the inflation in $g^*$, whereas the ratio $(u^*-u_\theta(\vec{x}))/u^*$ remains unaffected. In truth, both measures of suboptimality could be insightful depending on how one chooses to view the problem, so we consider both in this section.

\textcolor{black}{To find the optimal values $g^*$ and $u^*$ in systems with $2\leq m\leq 4$, the dynamic programming method we use is the policy iteration algorithm for unichain, average reward models, as described in \cite{Puterman1994} (p. 378). This algorithm works by taking an initial base policy (which we set as the modified index policy $\theta^{\text{M-IND}}$, as discussed in Section \ref{sec:polimp}) and then iterating between an `evaluation' step and an `improvement' step. The evaluation step is used to calculate the relative value function of the latest candidate policy, and the improvement step is used to find a superior policy by applying the Bellman equation at each state. The algorithm is guaranteed to converge to an optimal policy in a finite number of iterations. For details of the algorithm and associated computational times, please refer to Appendix \ref{AppPIA}.}



Tables \ref{EPI_cost} and \ref{EPI_reward} show, for each heuristic policy, 95\% confidence intervals for the mean percentage suboptimality relative to the optimal value $g^*$ under the cost-based formulation and $u^*$ under the reward-based formulation, respectively, based on results from 412 ‘small’ problem instances. 
We also report the 10th, 25th, 50th, 75th, and 90th percentiles of the percentage suboptimality distributions for each heuristic. These results indicate that the online policy improvement heuristic ($\theta^{\text{OPI}}$) consistently outperforms the other two heuristics, with an average percentage suboptimality of $2.51\%$ under the cost formulation and $1.02\%$ under the reward formulation. The index heuristic ($\theta^{\text{IND}}$) performs well in many instances, but is clearly inferior to $\theta^{\text{OPI}}$, most likely due to its short-sighted method of making decisions (as discussed in Section \ref{sec:index}). \textcolor{black}{It is unsurprising that $\theta^{\text{OPI}}$ improves upon $\theta^{\text{IND}}$ because the method of constructing $\theta^{\text{OPI}}$ involves using $\theta^{\text{IND}}$ as a base and then applying an improvement step.} The polling system heuristic ($\theta^{\text{POL}}$) is clearly weaker than $\theta^{\text{IND}}$ and $\theta^{\text{OPI}}$, despite the fact that $\theta^{\text{POL}}$ is a type of `ensemble' heuristic that involves taking the best average cost (or reward) from all of the possible polling heuristics $\theta^{\text{POL}}(\mathcal M)$ within each problem instance. The weak performance of $\theta^{\text{POL}}$ essentially shows the limitations of failing to respond dynamically to the latest machine degradation levels.

\begin{table}[H] 
	\centering
        \color{black}
	\caption{\black{Percentage suboptimalities of heuristic policies in 412 `small' problem instances with $2 \leq m \leq 4$, evaluated under the cost-based formulation.}}
\footnotesize 
	\label{EPI_cost}
		\begin{tabular}{lcccccccccccccccccccc} 
			\toprule 
			Heuristic &Mean&10th pct.&25th pct.&50th pct.&75th pct.&90th pct.\\
			\midrule
			POL & $29.64\pm3.72$ & 0.00 & 5.78 & 17.36 & 35.25 & 71.14 \\
                IND & $8.11\pm1.04$ &0.01 & 0.93 & 4.60 & 10.81 & 21.23 \\
                OPI & $2.51\pm0.55$ & 0.00 
               & 0.03 & 0.66 & 2.45 & 6.60\\
			\bottomrule 
	    \end{tabular}
\end{table}

\begin{table}[H] 
	\centering
        \color{black}
	\caption{\black{Percentage suboptimalities of heuristic policies in 412 `small' problem instances with $2 \leq m \leq 4$, evaluated under the reward-based formulation.}}
\footnotesize 
	\label{EPI_reward}
		\begin{tabular}{lcccccccccccccccccccc} 
			\toprule 
			Heuristic &Mean&10th pct.&25th pct.&50th pct.&75th pct.&90th pct.\\
			\midrule
			POL & $8.36\pm0.72$ & 0.00 & 2.07 & 7.43 & 12.64 & 18.95\\
                IND & $3.01\pm0.45$ & 0.00 & 0.37 & 1.70 & 3.89 & 7.30\\
                OPI & $1.02\pm0.32$ & 0.00 
                 & 0.02 & 0.25 & 0.83 & 2.09 \\
			\bottomrule 
	    \end{tabular}
\end{table}

Tables \ref{improvements with varying m_cost} and \ref{improvements with varying m_reward} show the percentage improvements achieved by $\theta^{\text{OPI}}$ against $\theta^{\text{IND}}$ over all 1,000 problem instances, with the results categorized according to the number of machines $m$, under the cost and reward formulations respectively. In larger problem sizes (with $m\geq 5$) it becomes impractical to compute exact percentage suboptimalities, but the results in these tables provide some assurance that the extra benefit given by applying online policy improvement remains fairly consistent as the size of the problem increases.

\begin{table}[H] 
	\centering
	\color{black}
	\caption{\black{Percentage improvements of OPI over the index policy for different values of $m$, evaluated under the cost-based formulation.}}
	\label{improvements with varying m_cost}
\footnotesize 
		\begin{tabular}{lcccccccccccccccccccc} 
			\toprule 
			 & $m=2$ & $m=3$ & $m=4$ & $m=5$ & $m=6$ & $m=7$ & $m=8$\\
              & [132 instances] & [148] & [132] & [146] & [148] &[149] &[145]\\
			\midrule
		OPI vs. IND & $3.45\pm1.13$ & $5.43\pm1.13$ & $5.00\pm0.91$ & $5.96\pm0.97$ & $5.12\pm0.76$ & $6.33\pm1.02$ & $4.98\pm0.94$\\
			\bottomrule 
	    \end{tabular}
\end{table}

\begin{table}[H] 
	\centering
	\color{black}
	\caption{\black{Percentage improvements of OPI over the index policy for different values of $m$, evaluated under the reward-based formulation.}}
	\label{improvements with varying m_reward}
\footnotesize 
		\begin{tabular}{lcccccccccccccccccccc} 
			\toprule 
			 & $m=2$ & $m=3$ & $m=4$ & $m=5$ & $m=6$ & $m=7$ & $m=8$\\
              & [132 instances] & [148] & [132] & [146] & [148] &[149] &[145]\\
			\midrule
		OPI vs. IND & $1.83\pm1.23$ & $2.97\pm1.53$ & $2.12\pm0.42$ & $2.37\pm0.37$ & $2.32\pm0.42$ & $2.12\pm0.36$ & $2.08\pm0.39$\\
			\bottomrule 
	    \end{tabular}
\end{table}

Table \ref{time proportions with varying m} shows, for each $m\in\{2,...,8\}$, the percentage of time steps at which the heuristic $\theta^{\text{OPI}}$ resorts to a `safe' action choice by selecting the same action given by the index heuristic $\theta^{\text{IND}}$. Recall from Section \ref{sec:polimp} that, in the online implementation of the OPI heuristic, the policy makes a `safe' choice of action if it cannot identify the best choice of action with a sufficient level of confidence. Table \ref{time proportions with varying m} shows that `safe' action choices are selected with increasing frequency as $m$ increases. This is due to the fact that, as the number of dimensions in the state space increases, it becomes more difficult to obtain reliable estimates for the values at individual states during the offline stage because there are relatively few states that are visited sufficiently often. This problem could be addressed by allocating more time to the offline stage as $m$ becomes larger, or increasing the online time limit $\delta$ (currently set at $0.01$ seconds) so that the algorithm can gather more information about nearby states at each time step during the online stage. 

\begin{table}[H] 
	\centering
	\color{black}
	\caption{\black{Percentage of `safe' action choices under $\theta^{\text{OPI}}$ for different values of $m$.}}
	\label{time proportions with varying m}
\footnotesize 
		\begin{tabular}{lcccccccccccccccccccc} 
			\toprule 
			 & $m=2$ & $m=3$ & $m=4$ & $m=5$ & $m=6$ & $m=7$ & $m=8$\\
              & [132 instances] & [148] & [132] & [146] & [148] &[149] &[145]\\
			\midrule
		\makecell{Pct. of `safe' \\ action choices} & $0.27\pm0.16$ & $0.86\pm0.35$ & $1.81\pm0.47$ & $4.53\pm0.93$ & $5.76\pm0.77$ & $8.37\pm1.12$ &
        $11.59\pm1.21$\\
			\bottomrule 
	    \end{tabular}
\end{table}

\subsection{Effects of varying the system parameters} \label{num:comparison of parameters}

Next, we investigate how the results from our experiments depend on the values of important system parameters. \textcolor{black}{The parameters that we focus on are the traffic intensity $\rho = \sum_{i\in M} \lambda_i / \mu_i$, the relative switching speed $\eta = \tau / (\sum_{i \in M} \lambda_i)$, the type of cost function $f_i(x_i)$ for $i\in M$ and the maximum degradation state $K$. By varying these parameters we can investigate a wide range of different scenarios, including those in which degradation events occur very frequently (or infrequently) relative to the speed of repairs, those in which the repairer travels very slowly (or quickly) through the network, those in which costs increase superlinearly with the level of degradation, and those in which there are several different stages of degradation before a machine fails completely. Thus, we aim to identify conditions under which our various heuristics perform well (or poorly). To the best of our knowledge, there are no publicly available benchmark instances that could be used to generate these parameter values, so we adopt a random generation method which ensures good coverage within the ranges of values that we consider; see Appendix \ref{AppGen} for full details.}

In Tables \ref{newtab:various_rho1} and \ref{newtab:various_rho2} we categorize our results according to the traffic intensity, $\rho = \sum_{i\in M} \lambda_i / \mu_i$. The first three rows of the tables show percentage suboptimalities of the heuristic policies $\theta^{\text{POL}}$, $\theta^{\text{IND}}$ and $\theta^{\text{OPI}}$ over the 412 instances with $2\leq m\leq 4$, and the fourth rows show the percentage improvements of $\theta^{\text{OPI}}$ over $\theta^{\text{IND}}$ over all 1,000 instances with $2\leq m\leq 8$. The columns correspond to disjoint, equal-length intervals for $\rho$. Within the table cells we report $95\%$ confidence intervals for the relevant values under both the cost formulation (shown in regular font) and the reward formulation (shown in bold font). Table \ref{newtab:various_eta} follows the same format as Tables \ref{newtab:various_rho1} and \ref{newtab:various_rho2} except that the parameter of interest is $\eta = \tau / (\sum_{i \in M} \lambda_i)$, which measures the relative switching speed compared to the sum of the machine degradation rates. Table \ref{newtab:various_cost} again follows the same format except that we categorize the results according to the cost function $f_i(x_i)$. The three cost functions that we consider are: linear, quadratic and piecewise linear (see Appendix \ref{AppGen} for full details). Finally, Table \ref{newtab:various_K} again follows the same format except that the parameter of interest is $K$, which is the maximum degradation state used for all machines in the system.

\begin{table}[htbp] 
    \centering
    \color{black}
    \caption{Summary of numerical results categorized according to the value of $\rho=\sum_i \lambda_i/\mu_i$ and $0.1\leq \rho < 0.7$, with 95\% confidence intervals shown under both the cost formulation (regular font) and the reward formulation (bold font)}
    \label{newtab:various_rho1}
    \footnotesize 
        \begin{tabular}{lcccccccc}
            \toprule 
             &$0.1\leq \rho < 0.3$ & $0.3\leq \rho < 0.5$ & $0.5\leq \rho < 0.7$\\
            \midrule 
            $\theta^{\text{POL}}$ Subopt. & $72.03\pm13.88$ & $42.98\pm10.65$ & $24.16\pm5.09$\\
            $$[412 instances]$$ & $\mathbf{[9.11\pm2.09]}$ & $\mathbf{[9.91\pm2.30]}$ & $\mathbf{[8.66\pm2.23]}$\\
            \midrule 
            $\theta^{\text{IND}}$ Subopt. & $8.92\pm2.31$ & $13.29\pm3.87$ & $11.32\pm3.70$\\
            $$[412 instances]$$ & $\mathbf{[1.78\pm0.73]}$ & $\mathbf{[3.23\pm1.24]}$ & $\mathbf{[3.73\pm1.37]}$\\
            \midrule 
            $\theta^{\text{OPI}}$ Subopt. & $2.90\pm1.27$ & $3.51\pm1.63$ & $3.35\pm1.54$\\
            $$[412 instances]$$ & $\mathbf{[0.92\pm0.79]}$ & $\mathbf{[0.64\pm0.31]}$ & $\mathbf{[1.23\pm0.91]}$\\
            \midrule 
            $\theta^{\text{OPI}}$ imp. vs.\ $\theta^{\text{IND}}$ & $6.26\pm1.02$ & $7.39\pm1.34$ & $6.82\pm1.04$\\
            $$[1,000 instances]$$ &  $\mathbf{[1.21\pm0.55]}$ &  $\mathbf{[2.35\pm0.71]}$ &  $\mathbf{[2.67\pm0.50]}$\\
            \bottomrule
        \end{tabular}
\end{table}

\begin{table}[htbp] 
    \centering
    \color{black}
    \caption{Summary of numerical results categorized according to the value of $\rho=\sum_i \lambda_i/\mu_i$ and $0.7\leq \rho < 1.5$, with 95\% confidence intervals shown under both the cost formulation (regular font) and the reward formulation (bold font)}
    \label{newtab:various_rho2}
    \scriptsize
    \footnotesize 
        \begin{tabular}{lcccccccc}
            \toprule 
             &$0.7\leq \rho < 0.9$ & $0.9\leq \rho < 1.1$ & $1.1\leq \rho < 1.3$ & $1.3\leq \rho < 1.5$\\
            \midrule 
            $\theta^{\text{POL}}$ Subopt. &$19.57\pm4.33$ & $16.56\pm4.10$ & $13.96\pm3.61$ & $8.13\pm1.66$\\
            $$[412 instances]$$ & $\mathbf{[7.28\pm1.75]}$ & $\mathbf{[8.96\pm1.97]}$ & $\mathbf{[8.96\pm1.87]}$ & $\mathbf{[6.27\pm1.21]}$\\
            \midrule 
            $\theta^{\text{IND}}$ Subopt. &$9.20\pm2.91$ & $6.46\pm3.24$ & $5.74\pm1.80$ & $2.86\pm0.76$\\
            $$[412 instances]$$ & $\mathbf{[2.95\pm1.43]}$ & $\mathbf{[2.94\pm0.99]}$ & $\mathbf{[4.30\pm1.81]}$ & $\mathbf{[2.74\pm0.86]}$\\
            \midrule 
            $\theta^{\text{OPI}}$ Subopt. & $2.98\pm1.51$ & $2.87\pm3.04$ & $1.51\pm0.63$ & $0.87\pm0.47$\\
            $$[412 instances]$$ &  $\mathbf{[0.80\pm0.36]}$ & $\mathbf{[1.29\pm1.33]}$ & $\mathbf{[1.04\pm0.51]}$ & $\mathbf{[1.27\pm1.18]}$\\
            \midrule 
            $\theta^{\text{OPI}}$ imp. vs.\ $\theta^{\text{IND}}$ &$5.80\pm1.05$ & $3.81\pm0.81$ & $3.74\pm0.75$ & $2.61\pm0.51$\\
            $$[1,000 instances]$$ &    $\mathbf{[2.49\pm1.09]}$ &  $\mathbf{[2.13\pm0.60]}$ &  $\mathbf{[3.06\pm1.29]}$ &  $\mathbf{[2.19\pm0.80]}$\\
            \bottomrule
        \end{tabular}
\end{table}

\begin{table}[H] 
    \centering
    \color{black}
    \caption{Summary of numerical results categorized according to the value of $\eta = \tau / (\sum_{i \in M} \lambda_i)$, with 95\% confidence intervals shown under both the cost formulation (regular font) and the reward formulation (bold font)}
    \label{newtab:various_eta}
    \footnotesize
    \resizebox{\textwidth}{!}{
        \begin{tabular}{lcccccccc}
            \toprule 
             &$0.1\leq \eta < 0.4$ & $0.4\leq \eta < 0.7$ & $0.7\leq \eta < 1$ & $1\leq \eta < 4$ & $4\leq \eta < 7$ & $7\leq \eta < 10$\\
            \midrule 
            $\theta^{\text{POL}}$ Subopt. & $5.89\pm2.16$ & $16.96\pm5.09$ & $23.78\pm6.19$ & $41.71\pm11.45$ & $37.87\pm10.52$ & $50.90\pm10.82$\\
            $$[412 instances]$$ & $\mathbf{[4.89\pm1.66]}$ & $\mathbf{[10.50\pm2.38]}$ & $\mathbf{[11.57\pm2.18]}$ & $\mathbf{[10.72\pm1.61]}$ & $\mathbf{[7.17\pm1.05]}$ & $\mathbf{[5.88\pm0.82]}$\\
            \midrule 
            $\theta^{\text{IND}}$ Subopt. & $4.09\pm1.82$ & $5.33\pm1.54$ & $6.26\pm2.27$ & $10.43\pm2.62$ & $9.85\pm3.15$ & $12.43\pm3.10$\\
            $$[412 instances]$$ & $\mathbf{[4.26\pm1.92]}$ & $\mathbf{[3.91\pm1.22]}$ & $\mathbf{[3.62\pm0.97]}$ & $\mathbf{[3.02\pm0.60]}$ & $\mathbf{[2.03\pm0.53]}$ & $\mathbf{[1.31\pm0.29]}$\\
            \midrule 
            $\theta^{\text{OPI}}$ Subopt. & $1.59\pm1.07$ & $1.24\pm0.49$ & $2.13\pm1.11$ & $2.84\pm1.10$ & $2.95\pm1.45$ & $4.20\pm2.12$\\
            $$[412 instances]$$ & $\mathbf{[2.18\pm1.55]}$ & $\mathbf{[0.71\pm0.25]}$ & $\mathbf{[1.35\pm0.87]}$ & $\mathbf{[0.90\pm0.37]}$ & $\mathbf{[0.55\pm0.22]}$ & $\mathbf{[0.39\pm0.19]}$\\
            \midrule 
            $\theta^{\text{OPI}}$ imp. vs.\ $\theta^{\text{IND}}$ & $2.37\pm0.78$ & $4.09\pm0.63$ & $4.02\pm0.89$ & $7.76\pm0.95$ & $6.75\pm1.15$ & $5.83\pm0.79$\\
            $$[1,000 instances]$$ & $\mathbf{[3.26\pm1.66]}$ & $\mathbf{[3.33\pm0.59]}$ & $\mathbf{[2.58\pm0.60]}$ & $\mathbf{[2.29\pm0.32]}$ & $\mathbf{[1.27\pm0.22]}$ & $\mathbf{[0.88\pm0.13]}$\\
            \bottomrule
        \end{tabular}
}
\end{table}

\begin{table}[H] 
    \centering
    \color{black}
    \caption{Summary of numerical results categorized according to the type of cost function $f_i(x_i)$ used, with 95\% confidence intervals shown under both the cost formulation (regular font) and the reward formulation (bold font)}
    \label{newtab:various_cost}
    \footnotesize
        \begin{tabular}{lccc}
            \toprule 
             & Linear & Quadratic & Piecewise linear\\
            \midrule 
            $\theta^{\text{POL}}$ Subopt. & $26.47\pm4.96$ & $32.76\pm7.41$ & $29.09\pm6.31$\\
            $[412 \text{ instances}]$ & $\mathbf{[10.16\pm1.42]}$ & $\mathbf{[7.55\pm1.07]}$ & $\mathbf{[7.55\pm1.27]}$\\
            \midrule 
            $\theta^{\text{IND}}$ Subopt. & $5.38\pm1.55$ & $9.19\pm1.99$ & $9.50\pm1.72$\\
            $[412 \text{ instances}]$ & $\mathbf{[2.98\pm1.05]}$ & $\mathbf{[2.82\pm0.59]}$ & $\mathbf{[3.27\pm0.70]}$ \\
            \midrule 
            $\theta^{\text{OPI}}$ Subopt. & $1.42\pm0.54$ & $3.20\pm1.23$ & $2.78\pm0.83$\\
            $[412 \text{ instances}]$ & $\mathbf{[0.95\pm0.49]}$ & $\mathbf{[0.77\pm0.33]}$ & $\mathbf{[1.38\pm0.80]}$\\
            \midrule 
            $\theta^{\text{OPI}}$ imp. vs.\ $\theta^{\text{IND}}$ & $3.58\pm0.52$ & $5.55\pm0.64$ & $6.49\pm0.73$\\
            $[1,000 \text{ instances}]$ & $\mathbf{[2.39\pm0.72]}$ & $\mathbf{[2.19\pm0.30]}$ & $\mathbf{[2.23\pm0.52]}$\\
            \bottomrule
        \end{tabular}
\end{table}

\begin{table}[H] 
    \centering
    \color{black}
    \caption{Summary of numerical results categorized according to the maximum degradation state $K$, with 95\% confidence intervals shown under both the cost formulation (regular font) and the reward formulation (bold font)}
    \label{newtab:various_K}
    \footnotesize
        \begin{tabular}{lccccc}
            \toprule 
             & $K=1$ & $K=2$ & $K=3$ & $K=4$ & $K=5$ \\
            \midrule 
            $\theta^{\text{POL}}$ Subopt. & $21.81\pm5.95$ & $33.79\pm11.74$ & $28.90\pm7.34$ & $33.91\pm8.89$ & $29.48\pm7.74$\\
            $$[412 instances]$$ & $\mathbf{[11.64\pm2.29]}$ & $\mathbf{[9.21\pm1.81]}$ & $\mathbf{[8.12\pm1.51]}$ & $\mathbf{[7.12\pm1.30]}$ & $\mathbf{[6.39\pm1.04]}$\\
            \midrule 
            $\theta^{\text{IND}}$ Subopt. & $3.26\pm1.11$ & $4.74\pm1.58$ & $6.94\pm1.52$ & $12.01\pm2.66$ & $12.06\pm3.09$ \\
            $$[412 instances]$$ & $\mathbf{[2.86\pm1.11]}$ & $\mathbf{[2.01\pm0.86]}$ & $\mathbf{[2.85\pm0.78]}$ & $\mathbf{[3.66\pm1.15]}$ & $\mathbf{[3.42\pm1.04]}$ \\
            \midrule 
            $\theta^{\text{OPI}}$ Subopt. & $1.19\pm0.89$ & $0.96\pm0.47$ & $1.83\pm0.96$ & $4.10\pm1.36$ & $3.93\pm1.70$\\
            $$[412 instances]$$ & $\mathbf{[1.57\pm1.33]}$ & $\mathbf{[0.50\pm0.33]}$ & $\mathbf{[0.95\pm0.73]}$ & $\mathbf{[1.07\pm0.53]}$ & $\mathbf{[0.98\pm0.27]}$\\
            \midrule 
            $\theta^{\text{OPI}}$ imp. vs.\ $\theta^{\text{IND}}$ & $2.16\pm0.58$ & $4.33\pm0.63$ & $5.08\pm0.67$ & $7.02\pm0.84$ & $7.03\pm1.09$\\
            $$[1,000 instances]$$ & $\mathbf{[1.71\pm0.86]}$ & $\mathbf{[1.89\pm0.41]}$ & $\mathbf{[2.18\pm0.48]}$ & $\mathbf{[2.75\pm0.78]}$ & $\mathbf{[2.70\pm0.73]}$\\
            \bottomrule
        \end{tabular}
\end{table}

From Tables \ref{newtab:various_rho1} and \ref{newtab:various_rho2} we can see that the polling system heuristic $\theta^{\text{POL}}$ is very weak when $\rho$ is small, but improves as $\rho$ increases. Indeed, when $\rho$ is small, we often have a situation where only one machine is degraded (that is, $x_j>0$ for only one $j\in M$) and the repairer should go directly to that machine under an optimal policy, but the polling heuristic forces the repairer to make wasteful visits to all machines according to a fixed cyclic pattern. The other heuristics $\theta^{\text{IND}}$ and $\theta^{\text{OPI}}$ also seem to achieve lower suboptimality values as $\rho$ increases under the cost formulation, but this may be caused by the inflation in $g^*$ that occurs when the overall degradation burden increases (as discussed earlier). The bold-font figures in Tables \ref{newtab:various_rho1} and \ref{newtab:various_rho2} show that when the reward formulation is used, there is no clear trend for the performances of $\theta^{\text{IND}}$ and $\theta^{\text{OPI}}$ to get better or worse as $\rho$ increases. Instead, $\theta^{\text{IND}}$ is consistently within $5\%$ of optimality, while $\theta^{\text{OPI}}$ is consistently within $2\%$.

Table \ref{newtab:various_eta} shows that, for all three heuristic policies, there is a general trend for suboptimality to increase with $\eta$ when the cost formulation is used. The same reasoning given above also applies in this case, since smaller $\eta$ values imply that machine degradation happens more frequently (relative to the repairer's movement speed), so the optimal value $g^*$ will tend to increase. Under the reward formulation, however, it seems that the trend for $\theta^{\text{IND}}$ and $\theta^{\text{OPI}}$ is the opposite: as $\eta$ increases, these heuristics seem to improve their performance. This may be due to the fact that the indices used by $\theta^{\text{IND}}$ (which also influence the decisions of $\theta^{\text{OPI}}$) are calculated under the assumption that the repairer will move directly to a particular demand point without changing direction, which is more likely to be an optimal course of action if $\eta$ is large (since there is less chance of new degradation events occurring while the repairer is moving). 

Table \ref{newtab:various_cost} shows that under the cost formulation, the three heuristic policies seem to achieve the lowest suboptimality values when the linear cost function is used. However, under the reward formulation, this trend is no longer apparent. 
Recall that the linear and quadratic cost functions are defined by $f_i(x_i)=c_ix_i$ and $f_i(x_i)=c_ix_i^2$ respectively, whereas the piecewise linear cost function $f_i(x_i)=c_i(x_i+10\cdot\mathbb{I}(x_i=K_i))$ assumes an extra penalty for any machine that has reached the state of complete failure. One might imagine that the piecewise linear case would be a difficult case to optimize, as even the machines with relatively low service priorities should still (ideally) be visited regularly enough to make complete failure unlikely. From Table \ref{newtab:various_cost}, however, there is no clear evidence that the heuristics perform worse in this case than in the other two.

Finally, Table \ref{newtab:various_K} shows that the effect of measuring machine degradation on a finer scale (in other words, increasing the number of possible degradation states) is rather difficult to pinpoint based on the suboptimality percentages. When the cost formulation is used, the performances of $\theta^{\text{IND}}$ and $\theta^{\text{OPI}}$ appear to become slightly worse as $K$ increases, which may be expected since the problem becomes intrinsically more complex. However, under the reward formulation, this trend seems to disappear. Indeed, even in the $K=5$ case, $\theta^{\text{IND}}$ and $\theta^{\text{OPI}}$ are able to achieve mean suboptimalities of less than $3.5\%$ and $1\%$ respectively. Encouragingly, in all of the cases shown in Tables \ref{newtab:various_rho1}}-\ref{newtab:various_K}, $\theta^{\text{OPI}}$ is able to obtain healthy improvements over $\theta^{\text{IND}}$.\\

\section{Conclusions}\label{sec:conclusions}

\textcolor{black}{This paper considers an original model of a stochastic, dynamic repair and maintenance problem in which a single repairer provides service to machines that are spatially dispersed on a network. By allowing switching and service times to be interruptible, we obtain an MDP in which decisions can be made at any point in time. Each machine evolves through a finite set of degradation states, and switching times depend not only on the machines involved but also on the progress made during any interrupted operations. Costs are linked to machine conditions using functions that can be defined arbitrarily. Although we have presented the model in the context of machine maintenance, our mathematical model is quite general and could be applied to other settings such as cybersecurity and healthcare routing, as discussed in Section \ref{sec:lit_review}.}

\textcolor{black}{From the methodological point of view, our main contributions are (i) the development of an index heuristic which exploits the interruptibility of service and switching operations by recalculating indices at every time step, and (ii) the use of a rollout-based RL method, with a novel safety feature based on reliability weights, to improve upon the index heuristic by making decisions in a less myopic way. The index heuristic is simple to implement and can be scaled to large problem instances without difficulty, as the number of indices to be calculated at each time step increases only linearly with the number of machines. Furthermore, our results in Section \ref{sec:index} have proven that this heuristic is optimal in certain special cases. Our numerical results in Section \ref{sec:numerical} have shown that the rollout-based online policy improvement method is close to optimal in small instances and consistently yields improvements over the index heuristic in larger instances, despite being restricted to only 0.01 seconds of simulation time per system transition during the online phase. This suggests that the approach is suitable for use in fast-changing real-world systems, such as those encountered in communications and computing.}

\textcolor{black}{Future research could focus on adapting these methods to more general frameworks, such as those with multiple repairers or nonstationary dynamics. It would be interesting to explore settings with partial observability, in which the repairer has incomplete knowledge of the conditions of the various machines. This could lead to the formulation of a partially observable MDP or POMDP (\cite{Kaelbling1998}). Another interesting avenue to explore, which could be relevant to security settings, is where degradation events occur as a result of adversarial behavior by a strategic opponent, rather than occurring at random. This would lead to a game theoretical formulation. It would also be interesting to investigate the use of other RL approaches, such as policy function approximation, which might be faster to implement in the online stage.} 

\section*{Acknowledgments.}
The first author was funded by a PhD studentship from the Engineering and Physical Sciences Research Council (EPSRC), under grant EP/W523811/1.


\small
\bibliographystyle{apalike} 
\bibliography{my_collection} 


%
%
%

\newpage
\Large \textbf{Online Appendices}\\
\normalsize

 \begin{appendices}
 
 \section {Proof of Proposition \ref{prop1}.}\label{AppA}
 
We prove the result by showing that our finite-state MDP belongs to the \emph{communicating} class of MDP models, as defined in \cite{Puterman1994} (p. 348). To this end, it will be sufficient to prove the lemma below.
 
 \begin{lemma}For any two states $\vec{x},\vec{x}'\in S$ there exists a stationary policy $\theta$ such that $\vec{x}'$ is accessible from $\vec{x}$ in the Markov chain induced by $\theta$.\label{AppA_lem}
\end{lemma}

The lemma is proved as follows. Let $\vec{x}=(i,(x_1,...,x_m))$ and $\vec{x}'=(i',(x_1',...,x_m'))$ be two given states in $S$.  Let $\theta$ be a stationary policy that chooses an action under state $\vec{x}$ as follows:\\
\vspace{-2.5mm}
\begin{itemize}
\item If $x_1=x_2=...=x_m=0$ then we identify the shortest path from $i$ to $i'$ and attempt to take the first step along that path.  (If $i=i'$ then we just remain at $i$.)
\item If $x_j>0$ for at least one machine $j\in M$ then we identify the `most-damaged' machine, i.e. the machine $j$ with $x_j\geq x_l$ for all $l\in M$.  In case of ties, we assume that $j$ is selected according to some pre-determined priority ordering of the machines.  If the repairer is already at node $j$ (i.e. $j=i$), then it remains there.  Otherwise, the repairer identifies the shortest path from its current location $i$ to node $j$, and attempts to take the first step along that path.\\
\end{itemize}
\vspace{-2.5mm}

For clarity, the `shortest path' between two nodes $i,i'\in V$ means the shortest possible sequence of connected nodes from $i$ to $i'$.  In case two or more paths are tied for the shortest distance, the path can be chosen according to some fixed priority ordering of the nodes.  Also, note that the rules above are implemented dynamically at each time step, so it may happen (for example) that the repairer begins moving towards a particular machine $j$ with $x_j\geq x_l$ for all $l\in M$ but, during the switching process, another machine $r$ degrades to the point that $x_j<x_r$ and therefore the repairer changes course and moves towards node $r$ instead.

To prove the proposition, it suffices to show that there is a possible sequence of random events that causes the system to reach state $\vec{x}'$ when the policy $\theta$ described above is followed.  In the first case ($x_1=x_2=...=x_m=0$), it is clear that the state $(i',(0,0,...,0))$ can be reached if none of the machines degrade before the repairer arrives at node $i'$.  If the system is in state $(i',(0,0,...,0))$ and one of the machines (say machine $1$) degrades, then the system transitions to state $(i',(1,0,...,0))$ and according to the rules above, the repairer should attempt to move to the first node on the shortest path from $i'$ to $1$.  However, with probability $1-\tau$, the switch `fails' and the repairer remains at $i'$.  Furthermore, it is possible for other machines to deteriorate (possibly multiple times) before the switch eventually succeeds.  (It may also happen that $i'=1$, in which case the repairer remains at machine $1$ and attempts to repair it, but degradations happen while the repair is in progress.)  Using this reasoning, we can see that the state $\vec{x}'=(i',x_1',...,x_m'))$ (where each $x_j'$ is an arbitrary value in $\{0,...,K_j\}$) is accessible from $(i,(0,0,...0))$ via an appropriate sequence of random degradation events.

In the second case ($x_j>0$ for at least one machine $j\in M$), suppose that no further degradations occur until the repairer has restored all of the machines to their `pristine' states. This fortuitous sequence of events allows the system to eventually transition to state $(j,(0,0,...,0))$, where in this case $j\in M$ denotes the last machine to be fully repaired.  We can then simply repeat the arguments given for the previous case, replacing $i$ with $j$, in order to show that state $\vec{x}'=(i',(x_1',...,x_m'))$ can eventually be reached.  This completes the proof of the lemma, and the result follows from the theory in \cite{Puterman1994} (see Theorem 8.3.2, p. 351). \hfill $\Box$\\

 \section {Proof of Proposition \ref{prop2}.}\label{AppB}

In this proof we consider an arbitrary stationary policy, $\theta$. Since $\theta$ is not necessarily optimal, it may be the case that the average cost $g_\theta(\vec{x})$ depends on the initial state $\vec{x}$. Let $P_\theta$ denote the transition matrix under $\theta$, consisting of the one-step transition probabilities $p_{\vec{x},\vec{y}}(\theta(\vec{x}))$ defined in (\ref{trans_probs}), where $\theta(\vec{x})$ is the action chosen at state $\vec{x}$ under $\theta$. Since the state space $S$ is finite and $P_\theta$ is a stochastic matrix, it follows from Proposition 8.1.1 in \cite{Puterman1994} (p. 333) that 
$$g_\theta(\vec{x})=\sum_{\vec{y}\in S}\pi_\theta(\vec{x},\vec{y})c(\vec{y}),$$
where $\pi_\theta(\vec{x},\vec{y})=\lim_{k\rightarrow\infty}\mathbb{P}_\theta(\vec{x}_k=\vec{y}\;|\;\vec{x}_0=\vec{x})$ denotes the limiting probability of moving from state $\vec{x}$ to state $\vec{y}$ in $k$ transitions under policy $\theta$, and this probability can be obtained as the entry in the row corresponding to state $\vec{x}$ and the column corresponding to state $\vec{y}$ in the matrix $P_\theta^*:=\lim_{k\rightarrow\infty}(P_\theta)^k$. 
It follows that equation (\ref{prop2_eqn}) is equivalent to
\begin{equation}\sum_{\vec{y}\in S}\pi_\theta(\vec{x},\vec{y})c(\vec{y})=\sum_{\vec{y}\in S}\pi_\theta(\vec{x},\vec{y})\tilde{c}(\vec{y},\theta(\vec{y}))\;\;\;\;\forall\vec{x}\in S.\label{prop2_eqn1}\end{equation}
Let the initial state $\vec{x}$ be fixed. For each machine $j\in M$ and $k\in\{1,2,...,K_j\}$, let $G_{j,k}$ denote the set of states $\vec{y}$ under which the repairer is at node $j$, $y_j=k$ and the action chosen under $\theta$ is to remain at node $j$.  By considering the costs incurred at each of the individual machines, it may be seen that equation (\ref{prop2_eqn1}) can be expressed as
$$\sum_{j\in M}\sum_{\vec{y}\in S}\pi_\theta(\vec{x},\vec{y})f_j(y_j)=\sum\limits_{j\in M}f_j(K_j)-\sum_{j\in M}\sum_{k=1}^{K_j} \sum_{\vec{y}\in G_{j,k}}\pi_\theta(\vec{x},\vec{y})\dfrac{\mu_j}{\lambda_j}[f_j(K_j)-f_j(k-1)].$$
It will therefore be sufficient to show that, for an arbitrary fixed $j\in M$, we have
\begin{equation}\sum_{\vec{y}\in S}\pi_\theta(\vec{x},\vec{y})f_j(y_j)=f_j(K_j)-\sum_{k=1}^{K_j} \sum_{\vec{y}\in G_{j,k}}\pi_\theta(\vec{x},\vec{y})\dfrac{\mu_j}{\lambda_j}[f_j(K_j)-f_j(k-1)].\label{prop2_eqn2}\end{equation}
Using the detailed balance equations for ergodic Markov chains (see, for example, \cite{Cinlar1975}), the rate at which machine $j$ transitions from state $y_j$ to $y_{j+1}$ under $\theta$ must equal the rate at which it transitions from $y_{j+1}$ to $y_j$.  This implies that
$$\lambda_j\sum_{\vec{y}\in S:y_j=k}\pi_\theta(\vec{x},\vec{y})=\;\;\mu_j\sum_{\vec{y}\in G_{j,k+1}}\pi_\theta(\vec{x},\vec{y})\;\;\;\;\;\;\forall k\in\{0,1,...,K_j-1\}.$$
Therefore the right-hand side of (\ref{prop2_eqn2}) can be written equivalently as
\begin{equation}f_j(K_j)-\sum_{k=0}^{K_j-1} \sum_{\vec{y}\in S:y_j=k}\pi_\theta(\vec{x},\vec{y})[f_j(K_j)-f_j(k)].\label{prop2_eqn3}\end{equation}
Then, using the fact that
$$\sum_{k=0}^{K_j}\sum_{\vec{y}\in S:y_j=k}\pi_\theta(\vec{x},\vec{y})=1,$$
we can write (\ref{prop2_eqn3}) in the form
\begin{align}&\sum_{k=0}^{K_j}\sum_{\vec{y}\in S:y_j=k}\pi_\theta(\vec{x},\vec{y})f_j(K_j)-\sum_{k=0}^{K_j-1} \sum_{\vec{y}\in S:y_j=k}\pi_\theta(\vec{x},\vec{y})[f_j(K_j)-f_j(k)]\nonumber\\
&=\sum_{\vec{y}\in S:y_j=K_j}\pi_\theta(\vec{x},\vec{y})f_j(K_j)+\sum_{k=0}^{K_j-1} \sum_{\vec{y}\in S:y_j=k}\pi_\theta(\vec{x},\vec{y})f_j(k)\nonumber\\[8pt]
&=\sum_{\vec{y}\in S}\pi_\theta(\vec{x},\vec{y})f_j(y_j)\nonumber\end{align}
which equals the left-hand side of (\ref{prop2_eqn2}) as required.  This confirms that the average costs for machine $j$ are equal under both cost formulations, and since this argument can be repeated for each machine $j\in M$ and initial state $\vec{x}\in S$, the proof is complete.  \hfill $\Box$\\

\section {Proof of Proposition \ref{prop3}}\label{AppC}

We assume that (i) there are $m$ machines arranged in a complete graph with no intermediate stages, i.e., $V=M$; (ii) for all $j \in M$, $K_j = 1$; (iii) all $m$ machines have identical parameter settings. Under these conditions, the state space has the form
$$S=\left\{(i,(x_1,...,x_m))\;|\;i\in M,\;x_j\in\{0,1\}\text{ for }j\in M\right\}.$$
Machine $i\in M$ is working if $x_i=0$, and has failed if $x_i=1$. Throughout this proof we omit the subscripts on the degradation rates $\lambda_j$, the repair rates $\mu_j$ and the cost functions $f_j$, since the third assumption implies that these are the same for all $j\in M$. Also, given that each machine has only two possible conditions, we will depart from our usual notational convention and use $k$ as a time step index rather than an indicator of a machine's condition.

The general strategy of this proof is to determine the decisions made by the index policy $\theta^{\text{IND}}$ under these conditions, and then show that these coincide with the decisions made by an optimal policy. First, suppose that the system is operating under the index policy. We can consider three possible cases: (a) the repairer is at a node $i$ with $x_i=1$; (b) the repairer is at a node $i$ with $x_i=0$ and there is at least one node $j\neq i$ with $x_j=1$; (c) the repairer is at a node $i$ with $x_i=0$ and we also have $x_j=0$ for all $j\neq i$. In case (a), the repairer has a choice between remaining at the failed machine $i$ or switching to an alternative machine. Recall that the index policy makes decisions using the indices $\Phi_i^{\text{stay}}(x_i)$, $\Phi_{ij}^{\text{move}}(x_j)$ and $\Phi_{ij}^{\text{wait}}(x_j)$ defined in (\ref{stay_index}), (\ref{move_index}) and (\ref{wait_index}) respectively. Given the assumption that $K_j=1$ for all $j\in S$, the index for remaining at node $i$ simplifies to
\begin{equation}\Phi_i^{\text{stay}}(x_i)=\Phi_i^{\text{stay}}(1)=\frac{\mathbb{E}[R_i(1)]}{\mathbb{E}[T_i(1)]}=\frac{\mu}{\lambda}f(1),\label{prop3_eq0a}\end{equation}
which is simply equal to the reward rate earned while the machine is being repaired. Next, consider the index for moving to an alternative node $j\neq i$. According to the rules of the heuristic, the repairer should only consider switching to nodes belonging to the set $J$ defined in (\ref{J_formula}). Thus, we need to check whether $\Phi_{ij}^{\text{move}}(x_j)\geq\Phi_{ij}^{\text{wait}}(x_j)$ for each $j\neq i$. We can express these indices as follows:

\begin{align}&\Phi_{ij}^{\text{move}}(x_j)=\begin{cases}
\mathbb{P}(X_j=1)\;\dfrac{\mathbb{E}[R_j(1)]}{\mathbb{E}[D_{ij}|X_j=1]+\mathbb{E}[T_j(1)]}=\dfrac{\mu\tau f(1)}{\tau^2+(2\mu+\lambda)\tau+\mu\lambda},&\text{ if }x_j=0,\\[16pt]
\dfrac{\mathbb{E}[R_j(1)]}{\mathbb{E}[D_{ij}]+\mathbb{E}[T_j(1)]}=\left(\dfrac{\tau}{\mu+\tau}\right)\dfrac{\mu}{\lambda}f(1),&\text{ if }x_j=1,\\[12pt]
\end{cases}\label{prop3_eq0b}\\[8pt]
&\Phi_{ij}^{\text{wait}}(x_j)=\dfrac{\mathbb{E}[R_j(1)]}{1/\lambda+\mathbb{E}[D_{ij}]+\mathbb{E}[T_j(1)]}=\dfrac{\mu\tau f(1)}{(\mu+\lambda)\tau+\mu\lambda}.
\label{prop3_eq0c}\end{align}

By comparing the expressions in (\ref{prop3_eq0b}) and (\ref{prop3_eq0c}) we can see that $\Phi_{ij}^{\text{move}}(x_j)<\Phi_{ij}^{\text{wait}}(x_j)$ if $x_j=0$, but $\Phi_{ij}^{\text{move}}(x_j)>\Phi_{ij}^{\text{wait}}(x_j)$ if $x_j=1$. Therefore the repairer only considers switching to node $j$ if $x_j=1$. However, even in this case $\Phi_{ij}^{\text{move}}(x_j)$ is smaller than the index $\Phi_i^{\text{stay}}(x_i)$ in (\ref{prop3_eq0a}) (noting that $R_i(1)\stackrel{\text{dist}}{=}R_j(1)$ and $T_i(1)\stackrel{\text{dist}}{=}T_j(1)$ due to homogeneous parameters), so the repairer always prefers to remain at node $i$ if $x_i=1$, regardless of the $x_j$ values at other nodes. 

Next, consider case (b), in which $x_i=0$ and there is a node $j\neq i$ with $x_j=1$. In this case, from (\ref{stay_index}) we have $\Phi_i^{\text{stay}}(x_i)=0$, while from (\ref{prop3_eq0b}) and (\ref{prop3_eq0c}) we can see that $\Phi_{ij}^{\text{move}}(x_j)$ is strictly positive and greater than $\Phi_{ij}^{\text{wait}}(x_j)$. Therefore the repairer chooses to switch to move $j$ in this case.

Finally, consider case (c), in which $x_j=0$ for all $j\in M$. In this case the repairer chooses to move to the `optimal idling position', $i^*$ (or remain if it is already there). Under the stated assumptions, the quantity $\Psi(i)$ in (\ref{new_psi_formula}) is the same for all $i\in M$, so $i^*$ can be chosen arbitrarily as any of the nodes in $M$. This is a rather trivial case and we will show later in the proof that any action by the repairer is optimal in this situation.

In the next stage of the proof we will show that there exists an optimal policy that makes the same decisions as the index policy. First, we introduce some terminology. Consider two distinct states $\vec{x}=(i,(x_1,...,x_m))$ and $\vec{x}'=(i',(x_1',...,x_m'))$. We say that states $\vec{x}$ and $\vec{x}'$ are \emph{symmetric} if $x_i=x_{i'}'$ and $\sum_{j\in M} x_j=\sum_{j\in M} x_j'$. That is, the total number of failed machines is the same under both states, and the repairer is either at a working machine under both states or at a failed machine under both states. For example, in a problem with 3 machines, we say that states $(1,(1,0,1))$ and $(2,(0,1,1))$ are symmetric, but the states $(1,(0,0,1))$ and $(2,(0,1,0))$ are not symmetric (because the repairer is at a working machine in the first case, but a failed machine in the second case). 


Consider a finite-horizon version of the discrete-time MDP formulated in Section \ref{sec:formulation} (following uniformization), and let $w_k(\vec{x})$ denote the minimum possible expected total cost that can be achieved over $k$ time steps, given that the initial state is $\vec{x}\in S$. In order to make progress we will need to use the following three properties of the function $w_k$, which we label (WN1)-(WN3) for convenience:\\
\begin{enumerate}
\item (WN1) If states $\vec{x}$ and $\vec{x}'$ are symmetric, then $w_k(\vec{x})=w_k(\vec{x}')$ for all $k\in\mathbb{N}$.
\item (WN2) If states $\vec{x}=(i,(x_1,...,x_m))$ and $\vec{x}'=(i',(x_1',...,x_m'))$ are identical except that $x_j=0$ and $x_j'=1$ for some $j\in M$, then $w_k(\vec{x})\leq w_k(\vec{x}')$ for all $k\in\mathbb{N}$.
\item (WN3) Suppose states $\vec{x}=(i,(x_1,...,x_m))$ and $\vec{x}'=(i',(x_1',...,x_m'))$ are identical except that $i\neq i'$. If $x_i=1$ and $x_{i'}=0$, then $w_k(\vec{x})\leq w_k(\vec{x}')$ for all $k\in\mathbb{N}$.\\
\end{enumerate}
Property (WN1) states that the optimal expected total cost remains the same if the initial state $\vec{x}$ is replaced by one of its symmetric states. Indeed, if $\vec{x}$ and $\vec{x}'$ are symmetric then $w_k(\vec{x})$ and $w_k(\vec{x}')$ must be identical, since the decision-maker is faced with the same problem in both cases (one may observe that the states $\vec{x}$ and $\vec{x}'$ are identical up to a renumbering of the nodes in $V$). Property (WN2) states that it is better for machine $j$ to be working than failed (if all else is equal), which is obvious because a failed machine will cause an extra cost to be incurred. Property (WN3) states that it is better for the repairer to be located at a failed machine than a working machine (if all else is equal), which makes sense because being at a failed machine allows the repairer to commence repairs without any delay. All three of these properties can be proved using induction on $k$. We omit the induction proofs for brevity.

Next, we consider the action that should be chosen in order to minimize the expected total cost over $(k+1)$ time steps, starting from state $\vec{x}=(i,(x_1,...,x_m))$. In general, we have
\begin{equation}w_{k+1}(\vec{x})=\min_{j\in M}\{q_{k+1}(\vec{x},j)\},\label{prop3_eq1}\end{equation}
where $q_{k+1}(\vec{x},j)$ is the expected total cost for choosing action $j\in M$ and then following an optimal policy for the remaining $k$ time steps. The expression for $q_{k+1}(\vec{x},j)$ takes different forms depending on whether the action $j$ is to remain at the current node $i$ or to switch to an alternative node. In the former case, it is also useful to distinguish between the cases $x_i=0$ and $x_i=1$. Let $M_0(\vec{x})$ and $M_1(\vec{x})$ be the total numbers of machines that are working and failed under state $\vec{x}$, respectively. Then, using the transition probabilities in (\ref{trans_probs}), we have

\begin{equation}q_{k+1}(\vec{x},j)=\begin{cases}
M_1(\vec{x})f(1)+M_0(\vec{x})\lambda w_k(\vec{x}^{j+})\\[8pt]
+\left(1-M_0(\vec{x})\lambda\right)w_k(\vec{x}),&\text{ if }j=i\text{ and }x_i=0,\\[12pt]
M_1(\vec{x})f(1)+M_0(\vec{x})\lambda w_k(\vec{x}^{j+})+\mu w_k(\vec{x}^{j-})\\[8pt]
+\left(1-M_0(\vec{x})\lambda-\mu\right)w_k(\vec{x}),&\text{ if }j=i\text{ and }x_i=1,\\[12pt]
M_1(\vec{x})f(1)+M_0(\vec{x})\lambda w_k(\vec{x}^{j+})+\tau w_k(\vec{x}^{\rightarrow j})\\[8pt]
+\left(1-M_0(\vec{x})\lambda-\tau\right)w_k(\vec{x}),&\text{ if }j\neq i,\end{cases}\label{prop3_eq2}\end{equation}
where $\vec{x}^{j+}$ denotes a state identical to $\vec{x}$ except that the component $x_j$ is changed from 0 to 1, $\vec{x}^{j-}$ denotes a state identical to $\vec{x}$ except that the component $x_j$ is changed from 1 to 0, and $\vec{x}^{\rightarrow j}$ denotes a state identical to $\vec{x}$ except that the repairer's location is changed to $j$.

Next, we consider the same three cases (a)-(c) discussed earlier and, for each case, show that the action chosen by the index policy (identified earlier) is also the action that should be chosen in order to minimize the finite-horizon expected total cost. Case (a) is where the repairer is at a machine $i$ with $x_i=1$. In order to show that remaining at machine $i$ minimizes the expected total cost over $k+1$ stages, we need to show that $q_{k+1}(\vec{x},i)\leq q_{k+1}(\vec{x},j)$ for all $j\neq i$. Using the expressions in (\ref{prop3_eq2}), it can be checked that
\begin{equation}q_{k+1}(\vec{x},i)- q_{k+1}(\vec{x},j)=\mu(w_k(\vec{x}^{j-})-w_k(\vec{x}))-\tau(w_k(\vec{x}^{\rightarrow j})-w_k(\vec{x})).\label{prop3_eq3}\end{equation}
By Property (WN2) we have $w_k(\vec{x}^{j-})-w_k(\vec{x})\leq 0$. If $j$ is a node with $x_j=1$ then $x$ and $x^{\rightarrow j}$ are symmetric states, so it follows from (WN1) that $w_k(\vec{x}^{\rightarrow j})-w_k(\vec{x})=0$. On the other hand, if $j$ is a node with $x_j=0$ then by (WN3) we have $w_k(\vec{x}^{\rightarrow j})-w_k(\vec{x})\geq 0$. It follows that the expression in (\ref{prop3_eq3}) is non-positive in all cases, so an optimal policy will choose to remain at node $i$ if $x_i=1$. Next, consider case (b), where the repairer is at a node $i$ with $x_i=0$ and there is at least one node $j$ with $x_j=1$. In this case, from (\ref{prop3_eq2}) we obtain
\begin{equation}q_{k+1}(\vec{x},i)- q_{k+1}(\vec{x},j)=-\tau(w_k(\vec{x}^{\rightarrow j})-w_k(\vec{x})),\label{prop3_eq4}\end{equation}
and by (WN3) we have $w_k(\vec{x}^{\rightarrow j})\leq w_k(\vec{x})$, so the expression in (\ref{prop3_eq4}) is non-negative. Hence, an optimal policy will choose to switch to node $j$ in this scenario. Finally, consider case (c), where the repairer is at a node $i$ with $x_i=0$ and we also have $x_j=0$ for all $j\neq i$. The expression in (\ref{prop3_eq4}) also applies to this case, but this time (WN1) implies that $w_k(\vec{x}^{\rightarrow j})-w_k(\vec{x})$ is equal to zero because states $\vec{x}$ and $\vec{x}^{\rightarrow j}$ are symmetric. Essentially, this is a trivial case in which any action is optimal, including the action taken by the index policy.

In summary, we have shown that the actions chosen by the index policy in the various possible cases (a)-(c) coincide with the actions chosen by an optimal policy in a finite-horizon problem in which the objective is to minimize total cost over $k$ stages (for arbitrary $k\in\mathbb{N}$). If we define $h(\vec{x}):=\lim_{k\rightarrow \infty}(w_k(\vec{x})-w_k(\vec{z}))$, where $\vec{z}$ is a fixed reference state in $S$, then it follows from standard theory (see \cite{Puterman1994}, p. 373) that the function $h$ satisfies the optimality equations (\ref{opt_eqs}). Hence, by our preceding arguments, the actions chosen by the index policy always attain the minimum on the right-hand side in (\ref{opt_eqs}), and the index policy is indeed optimal. \hfill $\Box$\\

\section{\textcolor{black}{Necessity of the conditions (i)-(iii) in Proposition \ref{prop3}} \label{Appagainstprop3}}

\textcolor{black}{In Proposition \ref{prop3}, three conditions are stated for the optimality of the index policy: (i) all machines are directly connected to each other; that is, $V$ is a complete graph; (ii) each machine has only two possible states; that is, $K_j=1$ for each $j\in M$; (iii) the degradation rates, repair rates and cost functions are the same for all machines. In this appendix we consider each of the conditions (i)-(iii) in turn and, for each one, present a counter-example showing that the index policy may not be optimal if the condition is relaxed while the other two conditions are retained.}

\textcolor{black}{The results described in examples (a)-(c) below are obtained by using dynamic programming (DP) methods to calculate the optimal average costs $g^*$ and comparing these with the average costs $g_{\theta^{\text{IND}}}$ given by the index policy. To obtain the optimal values $g^*$ we use policy iteration, as specified in Appendix \ref{AppPIA}. To obtain the index policy values $g_{\theta^{\text{IND}}}$ we simply evaluate the index policy by using the evaluation step in policy iteration, omitting any improvement steps.}

\textcolor{black}{Our counter-examples are described below.}
\textcolor{black}{
\begin{enumerate}[(a)]
    \item Consider a case where conditions (ii) and (iii) in Proposition \ref{prop3} hold, but condition (i) is violated. We consider a star network as defined in Definition \ref{star_def}, with radius $r = 1$ and $m = 3$, and $K_1 = K_2 = K_3 = 1$. The degradation rates are $\lambda_1 = \lambda_2 = \lambda_3 = 0.04$, the repair rates are $\mu_1 = \mu_2 = \mu_3 = 0.12$, and the switching rate is $\tau = 0.024$. The cost functions are of type 1 (see Appendix \ref{AppGen}), with $c_1 = c_2 = c_3 = 1$. In this case, $g_{\theta^{\text{IND}}} = 2.37$, whereas $g^* = 2.25$.
    \item Consider a case where conditions (i) and (iii) in Proposition \ref{prop3} hold, while condition (ii) is violated. We consider a system with $m = 3$ machines, all directly connected to each other, and $K_1 = K_2 = K_3 = 2$. The degradation rates are $\lambda_1 = \lambda_2 = \lambda_3 = 0.089$, the repair rates are $\mu_1 = \mu_2 = \mu_3 = 0.52$, and the switching rate is $\tau = 0.11$. The cost functions are of type 1 with $c_1 = c_2 = c_3 = 1$. In this case, $g_{\theta^{\text{IND}}} = 2.62$, whereas $g^* = 2.58$.
    \item Consider a case where conditions (i) and (ii) hold, while condition (iii) is violated. We consider a system with $m=3$ machines, all directly connected to each other, and $K_1 = K_2 = K_3 = 1$. We describe three subcases:
        \begin{itemize}
            \item (Heterogeneous degradation rates.) We take $\lambda_1 = 0.034$, $\lambda_2 = 0.16$, and $\lambda_3 = 0.055$. The repair rates are $\mu_1 = \mu_2 = \mu_3 = 0.74$, the switching rate is $\tau = 0.22$, and the cost functions are of type 1 with $c_1 = c_2 = c_3 = 1$. In this case, $g_{\theta^{\text{IND}}} = 0.85$, whereas $g^* = 0.80$.
            \item (Heterogeneous repair rates.) We take $\mu_1 = 0.82$, $\mu_2 = 0.12$, and $\mu_3 = 0.63$. The degradation rates are $\lambda_1 = \lambda_2 = \lambda_3 = 0.056$, the switching rate is $\tau = 0.15$, and the cost functions are of type 1 with $c_1 = c_2 = c_3 = 1$. In this case, $g_{\theta^{\text{IND}}} = 1.22$, whereas $g^* = 1.18$.
            \item (Heterogeneous cost parameters.) We consider cost functions of type 1 with $c_1 = 8.6$, $c_2 = 13.0$, and $c_3 = 8.1$. The degradation rates are $\lambda_1 = \lambda_2 = \lambda_3 = 0.14$, the repair rates are $\mu_1 = \mu_2 = \mu_3 = 0.56$, and the switching rate is $\tau = 0.36$. In this case, $g_{\theta^{\text{IND}}} = 13.15$, whereas $g^* = 12.98$.\\
        \end{itemize}
\end{enumerate}} 

\section {Proof of Proposition \ref{prop4}}\label{AppD}

Let $V$ be a star network as defined in Definition \ref{star_def}, with radius $r\geq 1$. The other conditions of the proposition are that $K_j = 1$ for all $j\in M$, all $m$ machines have identical parameter settings, and $\tau>2r\lambda$. The state space has the form
$$S=\left\{(i,(x_1,...,x_m))\;|\;i\in V,\;x_j\in\{0,1\}\text{ for }j\in M\right\}.$$
Machine $i\in M$ is working if $x_i=0$, and has failed if $x_i=1$. As in the proof of Proposition \ref{prop3} we will omit the subscripts on the degradation rates $\lambda_j$, repair rates $\mu_j$ and cost functions $f_j$, since these are assumed to be the same for all $j\in M$. Also, given that each machine has only two possible conditions, we will depart from our usual notational convention and use $k$ as a machine or time step index rather than an indicator of a machine's condition.

The proof of Proposition \ref{prop3} used dynamic programming arguments to show that the decisions made by the index policy are optimal, but this approach becomes more complicated in a network with intermediate stages, so instead we will use a slightly different approach. We will show that the decisions made by the index policy are \emph{greedy}, in the sense that they always minimize the expected amount of time until the repairer is next repairing a machine, and then show that a greedy policy is optimal under the conditions of the proposition.

The first step is to determine the decisions made under the index policy. In the arguments that follow, it will be useful to define $d(i,j)$ as the length of a shortest path between two nodes $i,j\in V$. We can consider five possible cases: (a) the repairer is at a machine $i\in M$ with $x_i=1$; (b) the repairer is at a machine $i\in M$ with $x_i=0$ and there is at least one machine $j\neq i$ with $x_j=1$; (c) the repairer is at a machine $i\in M$ with $x_i=0$ and we also have $x_j=0$ for all $j\neq i$; (d) The repairer is at an intermediate stage $i\in N$ and there is at least one machine $j\in M$ with $x_j=1$; (e) the repairer is at an intermediate stage $i\in N$ and we have $x_j=0$ for all $j\in M$. 

In case (a), the repairer has a choice between remaining at the failed machine $i$ or switching to an alternative machine. Recall that the index policy makes decisions using the indices $\Phi_i^{\text{stay}}(x_i)$, $\Phi_{ij}^{\text{move}}(x_j)$ and $\Phi_{ij}^{\text{wait}}(x_j)$ defined in (\ref{stay_index}), (\ref{move_index}) and (\ref{wait_index}) respectively. As in the proof of Proposition \ref{prop3}, the index for remaining at node $i$ simplifies to
\begin{equation}\Phi_i^{\text{stay}}(x_i)=\Phi_i^{\text{stay}}(1)=\frac{\mathbb{E}[R_i(1)]}{\mathbb{E}[T_i(1)]}=\frac{\mu}{\lambda}f(1).\label{prop4_eq0a}\end{equation}
Next, consider the index for moving to an alternative node $j\neq i$. According to the rules of the heuristic, the repairer should only consider switching to nodes belonging to the set $J$ defined in (\ref{J_formula}). The indices $\Phi_{ij}^{\text{move}}(x_j)$ and $\Phi_{ij}^{\text{wait}}(x_j)$ for $j\neq i$ can be expressed as
\begin{align}&\Phi_{ij}^{\text{move}}(x_j)=\begin{cases}
\mathbb{P}(X_j=1)\;\dfrac{\mathbb{E}[R_j(1)]}{\mathbb{E}[D_{ij}|X_j=1]+\mathbb{E}[T_j(1)]},&\text{ if }x_j=0,\\[16pt]
\dfrac{\mathbb{E}[R_j(1)]}{\mathbb{E}[D_{ij}]+\mathbb{E}[T_j(1)]},&\text{ if }x_j=1,\\[12pt]
\end{cases}\label{prop4_eq0b}\\[8pt]
&\Phi_{ij}^{\text{wait}}(x_j)=\dfrac{\mathbb{E}[R_j(1)]}{1/\lambda+\mathbb{E}[D_{ij}]+\mathbb{E}[T_j(1)]}.
\label{prop4_eq0c}\end{align}
It can be seen that if $x_i=1$ then $\Phi_i^{\text{stay}}(x_i)>\Phi_{ij}^{\text{move}}(x_j)$, regardless of whether or not node $j$ is included in the set $J$. Hence, the repairer will remain at node $i$ in this case. In case (b), where the repairer is at a node $i$ with $x_i=0$ and we also have $x_j=1$ for some $j\neq i$, it can be seen from the expressions above that $\Phi_{ij}^{\text{move}}(x_j)>\Phi_{ij}^{\text{wait}}(x_j)$, so machine $j$ is included in the set $J$ and furthermore the index for remaining at node $i$ is zero, so the repairer moves towards node $j$ in this case (this implies taking a step towards the center of the star network).

In case (c), where the repairer is at a machine $i$ with $i=0$ and we also have $x_j=0$ for all $j\neq i$, it is necessary to identify the node $i^*$ that minimizes the quantity $\Psi(i)$ defined in (\ref{new_psi_formula}). Given the layout of the star network (see Figure \ref{star_figure}), it is easy to see that $\Psi(i)$ is minimized by setting $i=s$, where $s$ is the center node. Hence, in this case (as in the previous case), the repairer should move from machine $i$ towards the center.

\begin{figure}[hbtp]
    \begin{center}
        \includeinkscape[scale = 0.6]{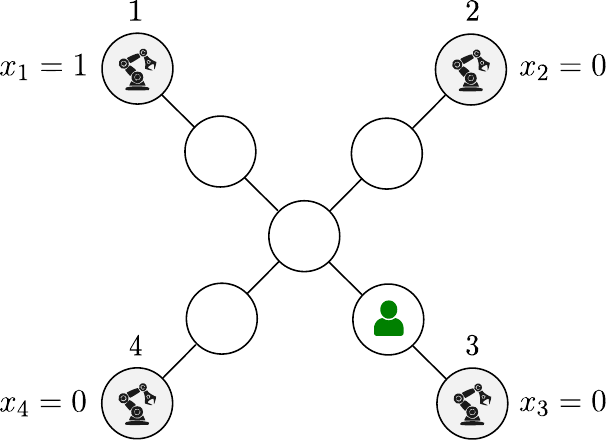}
    \end{center}
    \caption{A star network with 4 machines and a radius $r=2$, where the repairer is at an intermediate stage, machine 1 has failed with $x_1 = 1$ and other machines are working with $x_2 =x_3=x_4 = 0$.}
    \label{prop4_fig}
\end{figure}

In case (d), where the repairer is at an intermediate stage $i$ and there exists $j\in M$ with $x_j=1$, the rules of the heuristic imply that the repairer needs to identify the node $j$ that minimizes $\Phi_{ij}^{\text{move}}(x_j)$. Let $j$ be a fixed node in $V$ such that $x_j=1$. If $1\leq d(i,j)\leq r$ (implying that the repairer is somewhere between the center node and $j$), then it is clear that $\mathbb{E}[D_{ij}]\leq \mathbb{E}[D_{ik}]$ for any machine $k\neq j$ and therefore from (\ref{prop4_eq0b}) we have $\Phi_{ij}^{\text{move}}(x_j)\geq \Phi_{ik}^{\text{move}}(x_k)$. Therefore the repairer should move towards node $j$ in this situation. On the other hand, suppose we have $1\leq d(i,k)<r$ for some machine $k$ such that $x_k=0$ (see Figure \ref{prop4_fig}). In this case, the decision is less clear. The repairer could move towards the nearest machine $k$, in the hope that a degradation event will occur at that machine before the repairer arrives. However, we will show that the decision under the index policy is to move towards the damaged machine $j$, which is further away. In order to show that $\Phi_{ij}^{\text{move}}(1)\geq\Phi_{ik}^{\text{move}}(0)$, we begin by establishing some useful bounds. Firstly, we use the general fact that $1-x^p\leq k(1-x)$ for $k\in(0,1)$ and $p\in\mathbb{N}$ to show that
\begin{equation}\mathbb{P}(X_k=1)=1-\left(\frac{\tau}{\tau+\lambda}\right)^{d(i,k)}\leq \frac{d(i,k)\lambda}{\tau+\lambda}.\label{prop4_bound1}\end{equation}
Notably, this bound becomes tighter as $\tau/(\tau+\lambda)$ approaches one, which is relevant to our case because we assume $\tau>2r\lambda$. Secondly, we note that $\mathbb{E}[D_{ik}|X_k=1]$ is the expected total amount of time for $d(i,k)+1$ events to occur, where each event is either a completed switch or a degradation event at node $k$, and the latter type of event can only occur once. The expected amount of time for each of these $d(i,k)+1$ events is either $1/(\tau+\lambda)$ or $1/\tau$, depending on whether or not the degradation event has occurred yet. Hence,
\begin{equation}\mathbb{E}[D_{ik}|X_k=1]\geq\frac{d(i,k)+1}{\tau+\lambda}.\label{prop4_bound2}\end{equation}
Using the bounds in (\ref{prop4_bound1}) and (\ref{prop4_bound2}), we have
\begin{align}\Phi_{ik}^{\text{move}}(0)&=\mathbb{P}(X_k=1)\;\dfrac{\mathbb{E}[R_k(1)]}{\mathbb{E}[D_{ik}|X_k=1]+\mathbb{E}[T_k(1)]}\nonumber\\[8pt]
&\leq \frac{d(i,k)\lambda}{\tau+\lambda}\cdot\frac{\mathbb{E}[R_k(1)]}{(d(i,k)+1)/(\tau+\lambda)+\mathbb{E}[T_k(1)]}\nonumber\\[8pt]
&=\frac{d(i,k)\mu f(1)}{d(i,k)\mu+\lambda+\mu+\tau}.\label{prop4_eq1}\end{align}
On the other hand, using $E[D_{ij}]=(2r-d(i,k))/\tau$, the index for switching to machine $j$ is
\begin{equation}\Phi_{ij}^{\text{move}}(1)=\frac{\mathbb{E}[R_j(1)]}{\mathbb{E}[D_{ij}]+\mathbb{E}[T_j(1)]}=\frac{\tau\mu f(1)}{\lambda(\tau+(2r-d(i,k))\mu)}.\label{prop4_eq2}\end{equation}
Using (\ref{prop4_eq1}) and (\ref{prop4_eq2}), we find (after some algebra steps) that $\Phi_{ij}^{\text{move}}(1)$ is greater than or equal to the upper bound for $\Phi_{ik}^{\text{move}}(0)$ provided that
\begin{equation}\tau^2+(d(i,k)+1)\mu\tau\geq (d(i,k)-1)\lambda\tau+(2rd(i,k)-d(i,k)^2)\lambda\mu.\label{prop4_eq3}\end{equation}
Given that $\tau>2r\lambda$, it is sufficient to show that (\ref{prop4_eq3}) holds with $\lambda$ replaced by $\tau/(2r)$. After making this substitution and rearranging, we obtain
$$\frac{2r+d(i,k)^2}{2r}\mu\tau+\frac{2r-(d(i,k)-1)}{2r}\tau^2\geq 0,$$
which holds because $1\leq d(i,k)\leq r$. In conclusion, under the index policy the repairer chooses to move towards the damaged machine $j$ rather than the (nearer) undamaged machine $k$.

Finally, consider case (e), where the repairer is at an intermediate stage and $x_j=0$ for all $j\in M$. As noted in case (c), the quantity $\Psi(i)$ defined in (\ref{new_psi_formula}) is minimized when $i=s$, so the repairer moves towards the center in this situation.

In the next part of the proof, we consider cases (a)-(e) again and show that, in each case, the repairer's action under the index policy also minimizes the expected amount of time until repair. For clarity, we define the `time until repair' as the amount of time until the repairer chooses to remain at a damaged machine, and abbreviate it as TUR. In case (a), the TUR is trivially equal to zero under the index policy since the repairer remains at the damaged machine $i$. 

In case (b), the repairer is at a machine $i$ with $x_i=0$ and there exists $j\neq i$ with $x_j=1$. Suppose (for a contradiction) that the TUR is minimized by remaining at machine $i$. If machine $i$ is still undamaged at the next time step then, logically, the TUR is still minimized by remaining at node $i$, since the repairer is still in the same situation. (It may happen that a degradation event occurs at another machine $k\notin\{i,j\}$, but this does not affect the minimization of the TUR, since there is no reason for the repairer to prefer machine $k$ to machine $j$). Continuing this argument, it follows that the TUR is minimized by remaining at node $i$ until a degradation event occurs there. However, the expected time until this happens is $1/\lambda$, while the expected amount of time to move directly to machine $j$ is $2r/\tau$, and we have $2r/\tau<1/\lambda$ by assumption. Therefore, in fact, the repairer should not remain at machine $i$ and should take the only possible alternative action, which is to move towards the center of the network. Note that we have not shown, at this stage, that the repairer should move to machine $j$ without any interruption in order to minimize the TUR; we have only shown that if the repairer is at machine $i$ then it should take the first step towards the center. The actions chosen at intermediate stages will be discussed in case (d).

Next, we address cases (c) and (e) at the same time. In both of these cases we have $x_j=0$ for all $j\in M$. Let $i$ denote the repairer's current location, which could be either a machine or an intermediate stage. We note that the quantity $\Psi(i)$ in (\ref{new_psi_formula}) is the expected amount of time for the repairer to travel directly towards the next machine that degrades. It can easily be checked that $\Psi(i)$ increases with $d(i,s)$; that is, the further away the repairer gets from the center, the larger $\Psi(i)$ becomes. It follows that in order to minimize the TUR the repairer should travel towards the center node $s$. The index policy also directs the repairer to move to the center when all nodes are working, so again we find that it minimizes the TUR.

The only remaining case is (d), where the repairer is at an intermediate stage $i$ and there is at least one node $j$ with $x_j=1$. If $d(i,j)\leq r$ then it is obvious that the TUR is minimized by moving to machine $j$ without any interruption, which the index policy does. Thus, in order to avoid triviality, we should consider again the situation shown in Figure \ref{prop4_fig}, where the repairer's nearest node $k$ has $x_k=0$. In this case we can use an inductive argument to show that the TUR is minimized by moving towards the center. Initially, suppose $d(i,k)=1$, so the repairer is only one step from machine $k$. The TUR cannot be minimized by remaining at the intermediate stage $i$, since (using similar reasoning to case (b)) this would imply that the repairer should remain at node $i$ until a degradation event occurs at machine $k$, in which case the TUR is at least $1/\lambda$, which is greater than $(2r-1)/\tau$ (the TUR given by moving directly to machine $j$). On the other hand, the TUR cannot be minimized by moving to machine $k$ either, since we already showed in case (b) that if the repairer is at a machine $k$ with $x_k=0$ then it should move towards the center in order to minimize the TUR; therefore, moving from node $i$ to machine $k$ would result in back-and-forth movements between nodes $i$ and $k$ until a degradation occurs at machine $k$, and (again) the TUR would be at least $1/\lambda$. Hence, by process of elimination over the action set, we conclude that if $d(i,k)=1$ then the TUR is minimized by moving towards the center. An identical argument can be used when $d(i,k)=2$, using the fact that moving towards machine $k$ results in back-and-forth behavior between two intermediate stages and a TUR of at least $1/\lambda$. Continuing this argument, we conclude that the repairer should move towards the center and then continue towards machine $j$ (or another damaged machine) in order to minimize the TUR.

In summary, we have shown that in all of the cases (a)-(e) the action chosen by the index policy also minimizes the expected time until repair (TUR). The final stage of the proof is to show that a `greedy' policy that always minimizes the TUR is also optimal. In general, we can think of the evolution of the system as a sequence of time periods, where some of these periods are spent repairing a machine and some are spent performing other actions (such as moving between nodes in the network). Let these time periods be written as follows:
$$T_1,S_1,T_2,S_2,...,T_k,S_k,...$$
where each quantity $S_k$ (for $k\in\mathbb{N}$) is an amount of time spent repairing a machine, and each quantity $T_k$ is an amount of time spent performing other actions. 
For example, the repairer might begin at a machine $i$ with $x_i=0$, but then move towards a damaged machine $j$ and repair it until it is fixed, in which case $T_1$ is the amount of time spent moving towards machine $j$ and $S_1$ is the time spent performing the repair. The repairer earns a reward of $(\mu/\lambda)f(1)$ during each of the $S_k$ periods, and earns no reward during the $T_k$ periods. Thus, in order to maximize long-run average reward, the ratio $\sum_k S_k/(\sum_k (S_k+T_k))$ over a long time horizon should be as large as possible. It is clear from this argument that the repairer should never interrupt a repair in progress, since doing so would imply that (possibly) fewer degradation events would occur at other machines during the repair period, and therefore the repairer would not only earn a smaller reward from the current machine (due to interrupting the repair) but would also have fewer opportunities to earn rewards at other machines. We deduce that in order to maximize long-run average reward, the repairer should remain at a damaged machine if it is currently located there.

We will show that in order to maximize long-run average reward, the repairer should always greedily minimize the time until the next repair. This is not necessarily obvious, because if the repairer allows a particular non-repair time $T_k$ to increase then this might enable the next non-repair time $T_{k+1}$ to become smaller due to more degradation events occurring during the $k^{\text{th}}$ non-repair period and more opportunities for earning rewards becoming available. 

Let the system be initialized in an arbitrary state and let $j$ denote the first machine repaired by the index policy, given a particular sample path (i.e. a particular sequence of random events). Hence, $T_1$ is the time until the repairer begins repairing machine $j$ (this may be zero if the repairer is initially located at $j$). By our previous arguments, $\mathbb{E}[T_1]$ is minimized under the index policy, while $\mathbb{E}[S_1]$ is maximized (because the repair is not interrupted). When the repairer finishes repairing machine $j$, there are two possibilities: either (i) all machines are undamaged and the repairer has to wait until the next degradation event occurs, or (ii) there is at least one damaged machine, $j\neq i$, in the system. This idea is shown in Figure \ref{prop4_fig2}, in which the second non-repair period has been divided into subperiods, $A_2$ and $B_2$. We use $A_2$ to denote the waiting time until there is at least one damaged machine in the system (which is zero in the case that there is already a damaged machine when the first repair finishes), and $B_2=T_2-A_2$ to denote the additional time until the second repair period begins. In Figure \ref{prop4_fig2} we also use $t^{\text{deg}}$ to denote the time that the next degradation event occurs, given that there are no damaged machines at time $T_1+S_1$.

\begin{figure}[hbtp]
    \begin{center}
        \includegraphics[scale = 0.6]{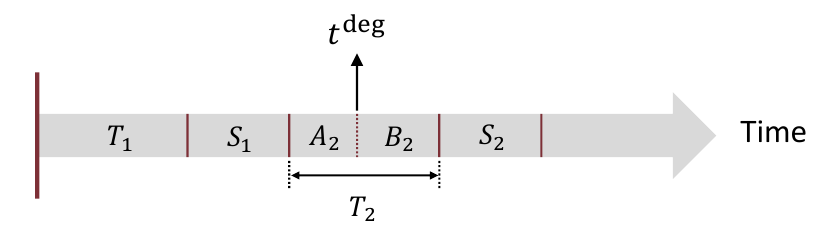}
    \end{center}
    \caption{Division of time into periods $T_1,S_1,T_2,S_2,...$ under the index heuristic.}
    \label{prop4_fig2}
\end{figure}

Given that we assume repairs are non-interrupted, $\mathbb{E}[S_1]$ is fixed and minimizing $\mathbb{E}[T_1]$ corresponds to maximizing $\mathbb{E}[A_2]$. If $A_2>0$ then the actions taken by the repairer during the time period $(T_1+S_1, T_1+S_1+A_2)$ ensure that the minimum TUR starting from time $T_1+S_1+A_2$ (denoted by $B_2$) is itself minimized. Intuitively, because the repairer moves towards the center of the network while it is waiting for the next degradation event to occur, it is in the best possible position to begin the second repair as early as possible. This ensures that the expected amount of non-repair time before starting the second repair, $\mathbb{E}[T_1+T_2]$, is minimized. We emphasize the subtle point that the index policy does not necessarily minimize $\mathbb{E}[T_2]$ (because $T_2$ includes a possible `waiting period' $A_2$), but the time spent in the waiting period works to the advantage of the index policy and ensures that the repairer spends as much time as possible moving towards the center, from which it follows that $\mathbb{E}[B_2]$ and $\mathbb{E}[T_1+T_2]$ are both minimized.

This argument can be continued in an inductive manner. In general we can replace $T_1$ with $\sum_{k=1}^K T_k$ in Figure \ref{prop4_fig2}, where $K\geq 1$, and use the inductive assumption that the index policy minimizes $\mathbb{E}[\sum_{k=1}^K T_k]$. Then the same arguments as above can be used to show that $\mathbb{E}[\sum_{k=1}^{K+1} T_k]$ is also minimized. Given that the $S_k$ values are independent and identically distributed (as noted above), it follows that the index policy maximizes the amount of time spent repairing over any finite time horizon, and therefore maximizes the expected total reward over a finite horizon. Therefore it must also maximize the long-run average reward. \hfill $\Box$\\

\section{Details of Online Policy Improvement}
\label{AppPI}

In this appendix we provide further details and pseudocode for the online policy improvement (OPI) method described in Section \ref{sec:polimp}. The method consists of an offline part, in which we acquire estimates of the value function $h_{\theta^{\text{M-IND}}}(\vec{x})$ under the modified index policy (referred to as the base policy) $\theta^{\text{M-IND}}$, and an online part, in which we apply an improvement step to the base policy (with the base policy also providing an insurance option in cases where an improving action cannot be identified with sufficient confidence). Both parts involve sampling trajectories of possible future events starting from a given state, and we present the pseudocode for this procedure as a separate module, called `SampleTrajectory'. We also present the offline part as two separate modules (a preparatory phase and a main phase), and the online part as a single module. In summary, the OPI method is presented here as 4 separate modules:

\begin{itemize}
\item SampleTrajectory (called upon by other modules)
\item Offline part - preparatory phase
\item Offline part - main phase
\item Online part
\end{itemize}

We briefly explain the steps of these modules in subsections \ref{AppE1}-\ref{AppE4}, and provide full pseudocode for all 4 modules in subsection \ref{AppE5}.

To simplify notation, we will denote the modified index policy throughout this appendix by $\theta \equiv \theta^{\textup{M-IND}}$. The offline and online parts use a dynamic array $U$ to store and update estimated average cost values $\hat{h}_{\theta}(\vec{x})$ for states $\vec{x}\in S$. This array also stores other statistics, including a sum of squared costs $\widehat{SS}_{\theta}(\vec{x})$, sum of squared weights $\widehat{W}_{\theta}(\vec{x})$, and observation count $\hat{s}_{\theta}(\vec{x})$ for each $\vec{x}\in U$. For readability, we drop the hats and $\theta$-subscripts in both the discussion and accompanying pseudocode, referring to these statistics as $h$-values, $SS$-values, $W$-values and $s$-values.\\

\subsection{SampleTrajectory Module}\label{AppE1}

The SampleTrajectory module requires certain inputs from the other modules, including a policy for selecting actions $\theta$, an estimate for the average cost $g_\theta$, an initial state $\vec{z}$, a reference state $\vec{u}_0$, an array $U$ with the latest gathered statistics (as described above) and an integer $p$ that determines how much information is recorded from each simulated trajectory. It generates a sample trajectory starting from $\vec{z}$ under policy $\theta$, and uses this to update the value estimates $h(\vec{x})$ and other statistics for newly encountered states $\vec{x}$. 

Trajectories are sampled by simulating the system evolution over a finite time, recording the cumulative cost $C$ and the total number of elapsed time steps $T$, until a stopping condition is satisfied. This condition is met if the simulated trajectory reaches either (i) a designated reference state $\vec{u}_0$, or (ii) a state already stored in the value estimate array $U$, excluding the initial state $\vec{z}$. At this point, the simulation terminates and the stopping state is labeled as $\vec{u}$.

During the simulation, each newly encountered state $\vec{x}$ (i.e. one not previously visited in the same trajectory) is recorded in a \texttt{visited} list, along with the cumulative cost $C_{\vec{x}}$ and time step $T_{\vec{x}}$ when $\vec{x}$ was first visited. Here, $C_{\vec{x}}$ represents the total accumulated cost from the start of the trajectory up to the first visit to $\vec{x}$, and $T_{\vec{x}}$ denotes the number of steps taken to reach $\vec{x}$ from the initial state.

Once the simulated trajectory ends, the algorithm processes the first $p$ visited states. For each such state $\vec{x}$, if it is not already included in $U$, \textcolor{black}{a new entry $U(\vec{x})=(h(\vec{x}), SS(\vec{x}), W(\vec{x}), s(\vec{x}))$ is created and initially set to $(0, 0, 0, 0)$} Otherwise, the current values are extracted directly from $U$ for update. Each update is computed using the bootstrapped estimate
$$
h^{(s)}(\vec{x}) \leftarrow (C - C_{\vec{x}}) + h(\vec{u}) - g_\theta \cdot (T - T_{\vec{x}}),
$$
where $s=s(\vec{x})$ is the observation count for state $\vec{x}$ and $h(\vec{u})$ denotes the latest available estimate for state $\vec{u}$, computed using the values $h^{(r)}(\vec{u})$ for $r=1,...,s(\vec{u})$. This corresponds to equation (\ref{offline_equation}) (please see the accompanying explanation in that section). We then update the estimate $h(\vec{x})$ using the rule
\begin{equation}h(\vec{x})\leftarrow (1-\alpha^{(s)}(\vec{x}))h(\vec{x})+\alpha^{(s)}(\vec{x})h^{(s)}(\vec{x}),\label{h_update_equation}\end{equation}
where $\alpha^{(s)}(\vec{x})$ is a learning parameter, given by
$$\alpha^{(s)}(\vec{x}) = \frac{10}{10 + s - 1}.$$
We note that (\ref{h_update_equation}) is a standard exponential averaging rule used in RL algorithms, which effectively states that the updated estimate $h(\vec{x})$ is computed as a weighted average of the old estimate and the newly-acquired observation (from the latest simulation). The second moment $SS(\vec{x})$, and sum of squared weights $W(\vec{x})$ are updated using a rule analogous to (\ref{h_update_equation}), and $U(\vec{x})$ is updated accordingly. The module returns the updated array $U$, the stopping state $\vec{u}$, and the total amount of time elapsed during the simulation, $\tau_\delta$.\\ 

\subsection{Offline part - preparatory phase}\label{AppE2}

In the preparatory phase we estimate the long-run average cost $ g_{\theta} $ and construct a representative set of system states. This module is divided into three stages:

\begin{enumerate}[(i)]
\item We begin by generating a core set of states $ Z^{\text{core}} $, consisting of one state $ \vec{z}_i $ for each machine $ i \in M $. Specifically, for each $ i\in M $, we simulate events starting from state $(i,(0,...,0))$ under the base policy $ \theta $ and identify the state visited most frequently during this simulation, which we denote by $ \vec{z}_i $. 
\item Subsequently, a separate long-run simulation from initial state $ (1, (0,...,0)) $ is used to estimate the average cost $ g_{\theta} $. During this simulation, we track how often the repairer visits the machines in $M$. The machine $ j^* $ with the highest visitation count is identified, and its associated core state $ \vec{z}_{j^*} $ is designated as the reference state $ \vec{u}_0 $. This reference state then becomes the initial member of the array $ U $ mentioned previously.
\item Next, we create a larger set $Z\supseteq Z^{\text{core}}$ to improve coverage of the state space. 
For each machine $ i $ we add the core state  $ \vec{z}_i \in Z^{\text{core}}$ to $Z$ and also add all states of the form $ \vec{z}_i^{\rightarrow j} $ for all nodes $j$ adjacent to $i$ in the network, where $ \vec{z}_i^{\rightarrow j} $ is identical to $\vec{z}_i$ except that the repairer's location is changed to $j$. Also, if $\vec{z}_i$ is a state with $i^{\text{th}}$ component greater than or equal to one, we include the state  $ \vec{z}_i^{i-} $ in $Z$, where $ \vec{z}_i^{i-} $ is identical to $\vec{z}_i$ except that the $i^{\text{th}}$ component is reduced by one. Including states of the form $ \vec{z}_i^{\rightarrow j} $ and $ \vec{z}_i^{i-} $ in $Z$ is beneficial for the online part that follows later.\\
\end{enumerate}

\vspace{0.5em}

\subsection{Offline part - main phase}\label{AppE3}

This module takes the set $Z$ from the preparatory phase as an input. It is divided into two stages:
\begin{enumerate}[(i)]
\item We consider each of the states $\vec{z}\in Z$ in turn. For each $\vec{z}\in Z$ we repeatedly sample trajectories of future events starting from $\vec{z}$ by using the SampleTrajectory module, and we update the array $U$ with the statistics ($h$-values, $SS$-values, $W$-values, and $s$-values) gathered from these trajectories. We also set $p=1$ within the SampleTrajectory module in this stage. This process continues (returning to the initial state $\vec{z}$ at the end of each trajectory) until either $R^{\text{off}}$ trajectories have been simulated (where $R^{\text{off}}$ is a large integer), or the cumulative simulation time exceeds a large threshold $ \tau_{\text{max}}$. At this point, we consider the next state in $Z$, or (if all states in $Z$ have been considered) move on to the next stage.
\item We consider each of the core states $ \vec{z} \in Z^{\text{core}}\subseteq Z $ in turn. For each state $ \vec{z} \in Z^{\text{core}} $ we sample a trajectory of future events starting from $\vec{z}$ and update the array $U$ with the statistics gathered, in the same way as in the previous stage. However, when the trajectory ends (which happens when the system reaches a state $\vec{u}\in U$), we use $\vec{u}$ as the initial state for the next trajectory, rather than returning to state $\vec{z}$. This creates a chain of successive trajectories, with the ending point of one trajectory acting as the starting point for the next. We also set $p=5$ within the SampleTrajectory module in this stage. As in the previous stage, this process continues until either $R^{\text{off}}$ trajectories have been simulated or the cumulative simulation time exceeds the threshold $ \tau_{\text{max}}$, at which point we consider the next state in $Z^{\text{core}}$, or (if all states in $Z^{\text{core}}$ have been considered) end the offline part.\\
\end{enumerate}

\subsection{Online part}\label{AppE4}

The online part is based on the paradigm of a `real' system that we are able to observe and control in a continuous manner over time. The aim is to apply an improvement step to the base policy $\theta$, making use of the statistical information gained from simulating the system evolving under $\theta$ in the offline part. In the absence of being able to observe a `real' system, we conduct a simulation of the system, and perform `nested' simulations within this in order to gather extra information about nearby states, as described in Section \ref{sec:polimp}.

At each iteration, the current system state $\vec{x}$ is observed and a local neighborhood $F(\vec{x})$ is constructed. This neighborhood includes (i) the current state $\vec{x}$ itself, (ii) states reachable by relocating the repairer to adjacent nodes, and (iii) a state resulting from repair completion if the repairer is currently at machine $i$ with $x_i \geq 1$. We then use the empirical data stored in $U$ to obtain, for each $\vec{y}\in F(\vec{x})$, a confidence interval for $h_\theta(\vec{y})$. The confidence interval is given by
\begin{equation}
    [h^-(\vec{y}), h^+(\vec{y})]= \left[h(\vec{y}) \pm z_{0.025} \sqrt{
        \left(\frac{SS(\vec{y}) - h(\vec{y})^2}{1 - W(\vec{y})}\right) W(\vec{y})}
    \right],
    \label{confidence_interval}
\end{equation}
where $z_{0.025}\approx 1.96$ is a Normal distribution percentile. The formula in (\ref{confidence_interval}) is derived from the theory of weighted sample variances and reliability weights (cf. Sec.~21.7 of \cite{galassignu}). 
In the following lines, we provide some additional details of how the formula is derived. Let $h(\vec{y})$ be the $h$-value estimate for state $\vec{y}$ after $s(\vec{y})$ observations have been collected. It can be checked using (\ref{h_update_equation}) that 
$$h(\vec{y}) = \sum_{r=1}^{\,s(\vec{y})} w^{(r)}(\vec{y}) \, h^{(r)}(\vec{y}),$$
where the weights $w^{(r)}(\vec{y})$ for $r=1,...,s(\vec{y})$ satisfy
$$w^{(r)}(\vec{y}) = \alpha^{(r)}(\vec{y})\prod_{j=r+1}^{\,s(\vec{y})} \Bigl(1 - \alpha^{(j)}(\vec{y})\Bigr).$$
We can compute the sample variance as
$$\begin{aligned}
\hat{\sigma}(\vec{y})^2 &= \frac{\sum_{r=1}^{s(\vec{y})} w^{(r)}(\vec{y}) \bigl(h^{(r)}(\vec{y}) - h(\vec{y})\bigr)^2}{\sum_{r=1}^{s(\vec{y})} w^{(r)}(\vec{y})} \\[6pt]
&= \sum_{r=1}^{s(\vec{y})} w^{(r)}(\vec{y})\cdot h^{(r)}(\vec{y})^2-h(\vec{y})^2\\[6pt]
&= SS(\vec{y}) - h(\vec{y})^2.
\end{aligned}$$
Then, to obtain an unbiased estimate $\sigma(\vec{y})$, we use
$$ \sigma(\vec{y})^2 = \frac{\hat{\sigma}(\vec{y})^2}{1 - (V_2/V_1^2)},$$
where
$$V_1 = \sum_{r=1}^{\,s(\vec{y})} w^{(r)}(\vec{y}), \qquad V_2 = \sum_{r=1}^{\,s(\vec{y})} w^{(r)}(\vec{y})^2.$$
As noted above, $V_1=1$ and we also have $V_2=W(\vec{y})$. Hence,
$$\sigma(\vec{y})^2 = \frac{SS(\vec{y}) - h(\vec{y})^2}{1 - W(\vec{y})}.$$
Thus, the effective variance of the weighted sample mean is 
$$\frac{\sigma(\vec{y})^2}{1/\sum_{r=1}^{s(\vec{y})} w^{(r)}(\vec{y})^2} = \left(\frac{SS(\vec{y}) - h(\vec{y})^2}{1 - W(\vec{y})}\right) W(\vec{y}),$$
which justifies (\ref{confidence_interval}). After computing the confidence intervals, the algorithm selects an action $a^* \in A_{\vec{x}}$ by, firstly, checking to see whether there is an action $ a^* $ satisfying the inequality
$$\sum_{\vec{y} \in S} p_{\vec{x},\vec{y}}(a^*)\, h(\vec{y}) < \sum_{\vec{y} \in S} p_{\vec{x},\vec{y}}(a')\, h(\vec{y})$$
for all alternative actions $ a' \neq a^* $ and for all possible values of $h(\vec{y})$ (for $\vec{y}\in F(\vec{x})$) within their respective confidence intervals $ [h^-(\vec{y}), h^+(\vec{y})] $. If no such $a^*$ exists, we resort to the base policy as a `safety measure' and set $a^* = \theta(\vec{x})$.

To continuously improve the accuracy of the value estimates, the SampleTrajectory module is called at every iteration. On each iteration, it is allowed to run for a maximum $\delta = 0.01$ seconds and the statistics in $U$ are updated during this period. The motivation for this step is that in the `real' system, there would be a certain amount of time ($\delta$) in between any pair of state transitions, and therefore we can make use of this time to improve value function estimates for states that are near the current state. Finally, each iteration ends with a real-time transition from the current state $\vec{x}$ to a successor state based on the chosen action $a^*$ (note that this is the action given by the `improving' step, not the action given by the base policy). After a large number of iterations, $R^{\text{on}}$, have been completed, the online part ends and the average cost over all time steps is returned as a performance metric for the OPI policy. \\

\subsection{Pseudocode of OPI}\label{AppE5}

\textcolor{black}{On the next four pages we provide pseudocode for the following modules, which are used as part of the online policy improvement method:
\begin{itemize}
\item SampleTrajectory
\item Offline part - preparatory phase
\item Offline part - main phase
\item Online part
\end{itemize}
Please note that full descriptions of these modules (including their purposes, inputs and outputs) are provided in Appendices \ref{AppE1}, \ref{AppE2}, \ref{AppE3} and \ref{AppE4} respectively.}\\

\begin{algorithm}[H] 
\captionsetup{labelformat=empty} 
\caption{\textbf{SampleTrajectory module}}
\label{alg:module_simulation}
\begin{algorithmic}[1]
\State \textbf{Input:} State $\vec{z}$; reference state $\vec{u}_0$; array $U$; recording length $p$; policy $\theta$; average cost $g_\theta$
\State \textbf{Initialize:}
\State Total cost incurred $C \gets 0$; total time steps $T \gets 0$
\State Current state $\vec{z}_{\text{curr}} \gets \vec{z}$; visited list $\texttt{visited} \gets [(\vec{z}, 0, 0)]$
\State Start time: $t_{\text{start}} \gets \text{current time}$

\While{true}
    \State $C \gets C + c(\vec{z}_{\text{curr}})$; $T \gets T + 1$ 
    \State Simulate transition: $\vec{z}_{\text{curr}} \xrightarrow{\theta(\vec{z}_{\text{curr}})} \vec{x}$
    \If{$(\vec{x} \neq \vec{z} \textbf{ or } \vec{x} = \vec{u}_0)$ \textbf{and} $\vec{x} \in U$}
        \State Set stopping state $\vec{u} \gets \vec{x}$; \textbf{break}
    \Else
        \State $\vec{z}_{\text{curr}} \gets \vec{x}$
        \If{$\vec{x} \notin \texttt{visited}$}
            \State Append $(\vec{x}, C, T)$ to $\texttt{visited}$
        \EndIf
    \EndIf
\EndWhile

\For{each $(\vec{x}, C_{\vec{x}}, T_{\vec{x}})$ in $\texttt{visited}\;[\;{:}\;p]$} \Comment{first $p$ visited during simulation}
    \If{$\vec{x} \notin U$}
        \State Initialize: $U(\vec{x}) \gets (0, 0, 0, 0)$
    \EndIf
    \State Unpack: $(h(\vec{x}), SS(\vec{x}), W(\vec{x}), s(\vec{x})) \gets U(\vec{x})$; $(h(\vec{u}), SS(\vec{u}), W(\vec{u}), s(\vec{u})) \gets U(\vec{u})$
    \State $s \gets s(\vec{x}) + 1$; $s(\vec{x}) \gets s$
    \State Compute learning parameter: $\alpha^{(s)}(\vec{x}) \gets \frac{10}{10 + s - 1}$
    \State Compute estimate: $h^{(s)}(\vec{x}) \gets (C-C_{\vec{x}}) + h(\vec{u}) - g_\theta \cdot (T-T_{\vec{x}})$
    \State Update:
    \[
    \begin{aligned}
    h(\vec{x}) &\gets (1 - \alpha^{(s)}(\vec{x})) \cdot h(\vec{x}) + \alpha^{(s)}(\vec{x}) \cdot h^{(s)}(\vec{x}) \\
    SS(\vec{x}) &\gets (1 - \alpha^{(s)}(\vec{x})) \cdot SS(\vec{x}) + \alpha^{(s)}(\vec{x}) \cdot (h^{(s)}(\vec{x}))^2 \\
    W(\vec{x}) &\gets (1 - \alpha^{(s)}(\vec{x}))^2 \cdot W(\vec{x}) + (\alpha^{(s)}(\vec{x}))^2
    \end{aligned}
    \]
    \State $U(\vec{x}) \gets (h(\vec{x}), SS(\vec{x}), W(\vec{x}), s(\vec{x}))$
\EndFor
\State End time: $t_{\text{end}} \gets \text{current time}$; simulation time $\tau_{\delta} \gets t_{\text{end}} - t_{\text{start}} $
\State \textbf{Output:} Array $U$; stopping state $\vec{u}$; simulation time $\tau_{\delta}$
\end{algorithmic}
\footnotesize\noindent\textbf{Note:} For each state $\vec{x}$, $s(\vec{x})$ is the number of estimates so far. The current estimate $h^{(s)}(\vec{x})$ and learning parameter $\alpha^{(s)}(\vec{x})$ correspond to the $s$-th update.
\end{algorithm}

\begin{algorithm}[H] 
\captionsetup{labelformat=empty} 
\caption{\textbf{Offline part - preparatory phase}}
\label{alg:offline_prep}
\begin{algorithmic}[1]
\State \textbf{Input:} Base policy $\theta$; simulation lengths $R^{(1)}$, $R^{(2)}$

\vspace{0.3em}

\State \textbf{Stage 1: Construct core state set $Z^{\text{core}}$}
\State Initialize $Z^{\text{core}} \gets \emptyset$
\For{each machine $i \in M$}
    \State Initialize $r \gets 0$, $\vec{x} \gets (i,(0,...,0))$, $S_i \gets \emptyset$; set $N_i(\vec{x}) \gets 0$ upon adding $\vec{x}$ to $S_i$ 
    \While{$r < R^{(1)}$}
        \State Simulate transition: $\vec{x} \xrightarrow{\theta(\vec{x})} \vec{y}$
        \State $\vec{x}=(j,(x_1,...,x_m)) \gets \vec{y}$; $r \gets r + 1$
        \If{$j= i$}
            \State $S_i \gets S_i \cup \{\vec{x}\}$; $N_i(\vec{x}) \gets N_i(\vec{x}) + 1$
        \EndIf
    \EndWhile
    \State $\vec{z}_i \gets \arg\max_{\vec{x} \in S_i} N_i(\vec{x})$; $Z^{\text{core}} \gets Z^{\text{core}} \cup \{\vec{z}_i\}$
\EndFor

\vspace{0.5em}

\State \textbf{Stage 2: Estimate average cost $g_{\theta}$ and choose reference state $\vec{u}_0$}
\State Initialize $C \gets 0$, $r \gets 0$, $\vec{x} \gets (1, (0,...,0))$, and counters $n_i \gets 0$ for all $i \in M$
\While{$r < R^{(2)}$}
    \State $C \gets C + c(\vec{x})$
    \State Simulate transition: $\vec{x} \xrightarrow{\theta(\vec{x})} \vec{y}$
    \State $\vec{x}=(i,(x_1,...,x_m)) \gets \vec{y}$; $r \gets r + 1$
    \If{$i \in M$}
        \State $n_i \gets n_i  + 1$
    \EndIf
    
\EndWhile
\State $g_{\theta} \gets C/R^{(2)}$; $j^* \gets \arg\max_i n_i$; $\vec{u}_0 \gets \vec{z}_{j^*}$; move $\vec{z}_{j^*}$ to the front of $Z^{\text{core}}$

\vspace{0.3em}

\State \textbf{Stage 3: Construct representative state set $Z$}
\State Initialize $Z \gets \emptyset$
\For{each $\vec{z}_i=(i,(z_1,...,z_m)) \in Z^{\text{core}}$}
    \State $Z \gets Z \cup \left\{ \vec{z}_i \right\}
        \cup \left\{ \vec{z}_i^{\rightarrow j} : j \text{ adjacent to } i \right\}$
    \If{$z_i\geq 1$}
        \State $Z \gets Z \cup \{\vec{z}_i^{i-}\}$ 
    \EndIf
\EndFor

\vspace{0.3em}

\State \textbf{Output:} Average cost $g_{\theta}$; reference state $\vec{u}_0$; core state set $Z^{\text{core}}$; representative state set $Z$
\end{algorithmic}
\end{algorithm}

\begin{algorithm}[H] 
\captionsetup{labelformat=empty} 
\caption{\textbf{Offline part - main phase}}
\label{alg:offline_simulation}
\begin{algorithmic}[1]
\State \textbf{Input:} Average cost $g_{\theta}$; reference state $\vec{u}_0$; core state set $Z^{\text{core}}$; representative state set $Z$; simulation length $R^{\text{off}}$; time limit $\tau_\text{max}$
\State Initialize array $U \gets \emptyset$; set $U(\vec{u}_0) \gets (0, 0, 1, 1)$

\vspace{0.3em}
\State \textbf{Stage 1: Value estimation for representative states}
\For{each $\vec{z} \in Z$}
    \State Initialize $r \gets 0$; $\tau \gets 0$; recording length $p \gets 1$
    \While{$r < R^{\text{off}}$ \textbf{and} $\tau < \tau_\text{max}$}
        \State Run \textbf{SampleTrajectory}($\vec{z}$) $\rightarrow$ $U$, stopping state $\vec{u}$, time $\tau_{\delta}$ \Comment{Ignore $\vec{u}$}
        \State $r \gets r + 1$; $\tau \gets \tau + \tau_{\delta}$
    \EndWhile
\EndFor

\vspace{0.3em}
\State \textbf{Stage 2: System evolution from core states}
\For{each $\vec{z} \in Z^{\text{core}}$}
    \State Initialize $r \gets 0$; $\tau \gets 0$; recording length $p \gets 5$
    \While{$r < R^{\text{off}}$ \textbf{and} $\tau < \tau_\text{max}$}
        \State Run \textbf{SampleTrajectory}($\vec{z}$) $\rightarrow$ $U$, stopping state $\vec{u}$, time $\tau_{\delta}$
        \State $\vec{z} \gets \vec{u}$; $r \gets r + 1$; $\tau \gets \tau + \tau_{\delta}$
    \EndWhile
\EndFor

\vspace{0.3em}
\State \textbf{Output:} Array $U$
\end{algorithmic}
\end{algorithm}

\begin{algorithm}[H] 
\captionsetup{labelformat=empty} 
\caption{\textbf{Online part}}
\label{alg:online_part}
\begin{algorithmic}[1]
\State \textbf{Input:} Average cost $g_{\theta}$; reference state $\vec{u}_0$; array $U$; simulation length $R^{\text{on}}$; $\delta$-length time limit
\State Initialize $C \gets 0$, $r \gets 0$, $\vec{x} \gets \vec{u}_0$
\While{$r < R^{\text{on}}$}
    \State $C \gets C + c(\vec{x})$
    \State \textbf{Stage 1: Action selection}
    \State Define local neighborhood $F(\vec{x})$
    \For{each $\vec{y} \in F(\vec{x})$}
        \State Compute confidence interval $[h^-(\vec{y}), h^+(\vec{y})]$ from $U$
    \EndFor
    \State Select $a^*\in A_{\vec{x}}$ satisfying:
    $$\sum_{\vec{y} \in S} p_{\vec{x},\vec{y}}(a^*)\, h(\vec{y}) < \sum_{\vec{y} \in S} p_{\vec{x},\vec{y}}(a')\, h(\vec{y}), \quad \forall a' \neq a^*, \quad \forall h(\vec{y}) \in [h^-(\vec{y}), h^+(\vec{y})].$$
     \If{no such \( a^* \) exists}
        \State Set \( a^* \gets \theta(\vec{x}) \)
    \EndIf
    \vspace{0.5em}

    \State \textbf{Stage 2: Information supplement}
    \State Initialize $\tau \gets 0$; recording length $p \gets 1$
    \While{$\tau < \delta$}
        \State Simulate transition: $\vec{x} \xrightarrow{a^*} \vec{x}'$
        \State Define local neighborhood $F(\vec{x}')$
        \For{each $\vec{y} \in F(\vec{x}')$}
            \State Run \textbf{SampleTrajectory}($\vec{y}$) $\rightarrow$ $U$, $\vec{u}$, $\tau_{\delta}$
            \State $\tau \gets \tau + \tau_{\delta}$
        \EndFor
    \EndWhile
    \vspace{0.5em}

    \State \textbf{Stage 3: Real-time transition}
    \State Simulate transition: $\vec{x} \xrightarrow{a^*} \vec{x}_{\text{new}}$
    \State $\vec{x} \gets \vec{x}_{\text{new}};\; r \gets r + 1$
\EndWhile
\State $g \gets C / R^{\text{on}}$
\State \textbf{Output:} Average cost $g$
\end{algorithmic}
\footnotesize \textbf{Note:} Confidence intervals are computed using the method of reliability weights (\cite{galassignu}, Sec. 21.7)
\end{algorithm}

\section{Methods for generating the parameters for the numerical experiments in Section \ref{sec:numerical}}\label{AppGen}

For each of the 1,000 instances considered in Section \ref{sec:numerical}, the system parameters are randomly generated as follows:
\begin{itemize}
    \item \textcolor{black}{The type of cost function is sampled unbiasedly from the set $\{1, 2, 3\}$, where:}
    \begin{enumerate}[(i)]
        \item \textcolor{black}{Type 1 represents a linear cost function, i.e., the cost incurred per unit time for machine $i \in M$ in state $x_i$ is given by $f_i(x_i) = c_i x_i$;}
        \item \textcolor{black}{Type 2 represents a quadratic cost function, i.e., the cost incurred per unit time for machine $i \in M$ in state $x_i$ is given by
        $f_i(x_i) = c_i x_i^2$;}
        \item \textcolor{black}{Type 3 represents a piecewise linear cost function, i.e., the cost incurred per unit time for machine $i \in M$ in state $x_i$ is given by $f_i(x_i) = c_i \left( x_i + 10 \cdot \mathbb{I}(x_i = K_i) \right),$
        where $\mathbb{I}(x_i = K_i)$ is the indicator function that equals 1 if $x_i = K_i$, and 0 otherwise.}
    \end{enumerate}

    \item The `failed' state, $K$, is sampled unbiasedly from the set $\{1,2,3,4,5\}$.
    \item \textcolor{black}{The number of machines, $m$, is sampled unbiasedly from the set $\{2,3,4,5,6,7,8\}$}. 
    \item The $a$-coordinate and $b$-coordinate of each machine $i\in M$, denoted as $(a_i,b_i)$, are independently sampled unbiasedly from the set $\{1,2,3,4,5\}$, with resampling used if two or more machines share the same pair of coordinates.
    \item The machines in $M$ are re-numbered according to the $a$-coordinate first and then (in case of ties) by the $b$ coordinate.
    \item The overall traffic intensity, $\rho$, is sampled from a continuous uniform distribution between 0.1 and 1.5. Subsequently, the degradation rates $\lambda_i$ and repair rates $\mu_i$ for $i\in M$ are generated as follows:
    \begin{itemize}
        \item Each repair rate $\mu_i$ is initially sampled from a continuous uniform distribution between 0.1 and 0.9.
        \item For each machine $i\in M$, an initial value for the degradation rate $\lambda'_i$ is sampled from a continuous uniform distribution between $0.1\mu_i$ and $\mu_i$.
        \item For each machine $i\in M$, the actual traffic intensity $\rho_i$ is obtained by re-scaling the initial traffic intensity $\lambda'_i/\mu_i$, as follows:
        $$\rho_i:=\frac{\lambda'_i/\mu_i}{\sum_{i\in M}\lambda'_i/\mu_i}\rho.$$
        This ensures that $\sum_{i\in M}\rho_i=\rho$.
        \item For each machine $i\in M$, the actual degradation rate $\lambda_i$ is obtained as follows:
        $$\lambda_i:=\rho_i\mu_i.$$
        \item All of the $\lambda_i$ and $\mu_i$ values are rounded to 2 significant figures (also causing a small change to the value of $\rho$).
\end{itemize}

    \item For each machine $i\in M$, the cost rate, $c_i$, is sampled from a continuous uniform distribution between 0.1 and 0.9.
    \item In order to generate the switching rate $\tau$, we first define $\eta:=\tau/(\sum_{i\in M}\lambda_i)$ and generate the value of $\eta$ as follows:
        \begin{itemize}
        \item Sample a value $p$ from a continuous uniform distribution between 0 and 1.
        \item If $p < 0.5$, sample $\eta$ from a continuous uniform distribution between 0.1 and 1.
        \item If $p\geq 0.5$, sample $\eta$ from a continuous uniform distribution between 1 and 10.
        \end{itemize}
    We then define  $\tau := \eta \sum_{i\in M}\lambda_i$.\\
    \end{itemize}

\textcolor{black}{\section{Policy Iteration Algorithm for systems with $2\leq m\leq 4$}\label{AppPIA}}

\textcolor{black}{The following pseudocode shows how policy iteration is used to calculate the optimal value of $g^*$ for systems with $2\leq m\leq 4$ in our numerical experiments. Specifically, we implement the policy iteration algorithm described in \cite{Puterman1994} (p. 378), with the relative value function for each candidate policy computed via successive approximations until a convergence criterion is met.}

\begin{algorithm}[H]
\color{black}
\captionsetup{labelformat=empty}
\caption{\textbf{Policy Iteration Algorithm}}
\label{alg:optimal_unichain_pia}
\begin{algorithmic}[1]

\State \textbf{Input:} base policy $\theta$, reference state $\vec{x}_0$, tolerance $\varepsilon$
\vspace{0.5em}

\State \textbf{Initialization:}
\State $v(\vec{x}) \gets 0$, $g(\vec{x}) \gets 0$ for all $\vec{x} \in \theta$
\State $v^{+}(\vec{x}) \gets 0$, $g^{+}(\vec{x}) \gets 0$ for all $\vec{x} \in \theta$
\State $g^{\textup{error}} \gets \infty$

\vspace{0.8em}
\State \textbf{Stage 1: Policy Evaluation Algorithm (PEA)}
\While{$g^{\textup{error}} \ge \varepsilon$}
    \State $g^{\textup{error}} \gets 0$
    \For{each state $\vec{x} \in \theta$}
        \State Compute $p_{\vec{x},\vec{y}}$ and cost $c(\vec{x})$
        \State
        \[
        v^{+}(\vec{x}) \gets
        c(\vec{x})
        + \sum_{\vec{y}} p_{\vec{x},\vec{y}} v(\vec{y})
        + \Bigl(1-\sum_{\vec{y}} p_{\vec{x},\vec{y}}\Bigr)v(\vec{x})
        \]
        \State $g^{+}(\vec{x}) \gets v^{+}(\vec{x}) - v(\vec{x})$
        \State $g^{\textup{error}} \gets
        \max\!\left\{ g^{\textup{error}},
        |g^{+}(\vec{x})-g(\vec{x})| \right\}$
    \EndFor
    \State $v \gets v^{+}$, $g \gets g^{+}$
\EndWhile
\State $g^{\theta} \gets g(\vec{x}_0)$
\State $v(\vec{x}) \gets v(\vec{x}) - v(\vec{x}_0)$

\vspace{0.8em}
\State \textbf{Stage 2: Successive Policy Improvement Algorithm (nPIA)}
\State $n \gets 0$, $\theta^{n} \gets \theta$
\Repeat
    \For{each state $\vec{x} \in \theta^{n}$}
        \State For each admissible action $a$ at $\vec{x}$, compute
        $$
        Q(\vec{x},a)
        =
        c(\vec{x})
        + \sum_{\vec{y}} p_{\vec{x},\vec{y}}(a)\, v(\vec{y})
        + \Bigl(1-\sum_{\vec{y}} p_{\vec{x},\vec{y}}(a)\Bigr)v(\vec{x})
        $$
        \State $\theta^{n+1}(\vec{x}) \gets \arg\min_{a} Q(\vec{x},a)$
    \EndFor
    \State Evaluate $\theta^{n+1}$ using \textbf{Stage~1 (PEA)}
    \State $n \gets n+1$

\Until{$\theta^{n} = \theta^{n-1}$}
\State $g^{*} \gets g(\vec{x}_0)$
\vspace{0.8em}
\State \textbf{Output:} $g^*$

\end{algorithmic}
\end{algorithm}

\textcolor{black}{Computational times for this algorithm (categorized according to the value of $m$) are reported in Table \ref{PIA_comptimes} below.}

\begin{table}[htbp] 
	\centering
	\color{black}
	\caption{\textcolor{black}{Average running times (in seconds) for policy iteration algorithm, over 412 problem instances with $2 \leq m \leq 4$ on randomly-generated networks.}}
	\label{PIA_comptimes}
    \footnotesize
		\begin{tabular}{lcccccccccccccccccccc} 
			\toprule 
			&$m=2$ &$m=3$ &$m=4$\\
             & [132 instances] &[148 instances] &[132 instances]\\
			\midrule
                Av. run times & $0.091$ & $1.21$ & $11.73$ \\
			\bottomrule 
	    \end{tabular}
\end{table}

\section{Simulation methods for the numerical experiments in Section \ref{sec:numerical}}\label{AppSim}

In each of the 1,000 problem instances, the performances of the index policy and the policies given by the polling heuristic are estimated by simulating the discrete-time evolution of the uniformized MDP described in Section \ref{sec:formulation}. We also use a `common random numbers' method to ensure that machine degradations occur at the same times under each of these policies. More specifically, the simulation steps are as follows:
\begin{enumerate}
\item Generate a set of random system parameters as described in Appendix \ref{AppGen}.
\item Set $R^{\text{sim}}:=500,000$ as the number of time steps in the simulation.
\item Generate a list $Z$ of length $R^{\text{sim}}$, consisting of uniformly-distributed random numbers between 0 and 1.
\item Consider each heuristic policy in turn. For each one, we set $\vec{x}_0:=(1,(0,0,...,0))$ as the initial state, and then use the random numbers in $Z$ to simulate events during these $R^{\text{sim}}$ time steps, with the system beginning in state $\vec{x}_0$. The statistics collected during these time steps are used to quantify the heuristic's performance.
\end{enumerate}

When implementing the OPI policy, we need to decide on the values of $R^{(1)}$, $R^{(2)}$, $R^{\text{off}}$, $\tau_{\max}$ and $R^{\text{on}}$ used in the offline and online stages (see the pseudocode in Appendix \ref{AppPI} for details). We use the values $R^{(1)}=10,000$, $R^{(2)}=500,000$, $R^{\text{off}}=100,000$, $\tau_{\max}=100$ and $R^{\text{on}}=500,000$.

\end{appendices}


%

\end{document}